\documentclass[10pt,paper=a4,reqno]{amsart}
\usepackage[inner=2.25cm,outer=3cm,
			top=3.25cm,bottom=3.5cm]{geometry} 
\usepackage{graphicx} 
\usepackage[T1]{fontenc}
\usepackage[utf8]{inputenc}
\usepackage{amsfonts,amsmath,amsthm,amssymb}
\usepackage{mathtools}
\usepackage{hyperref}
\usepackage[capitalise]{cleveref}
\usepackage{bm}
\usepackage{enumitem}
\usepackage{varwidth}
\usepackage{booktabs,subfig} 
\usepackage{caption}
\usepackage{subcaption}
\usepackage[dvipsnames]{xcolor}
\usepackage{verbatim}
\usepackage{tikz-cd}
\usepackage{chemfig}
\usepackage{multicol}
\usepackage[toc,page]{appendix}

\numberwithin{equation}{section}

\newcommand{\ie}{\emph{i.e.}\ }

\newtheorem{definition}{Definition}[section]
\newtheorem{remark}[definition]{Remark} 
\newtheorem{theorem}{Theorem}[section]
\newtheorem{lemma}[theorem]{Lemma}
\newtheorem{corollary}{Corollary}[theorem]
\newtheorem{proposition}{Proposition}[section]
\newtheorem{assumption}{Assumptions} 
\theoremstyle{plain}

\def\half{\frac{1}{2}}
\newcommand{\bigo}{{\mathcal O}}

\newcommand\bfb{{\mathbf b}}

\newcommand\bfe{{\mathbf e}}

\newcommand\bff{{\bm f}}

\newcommand\bfg{{\mathbf g}}

\newcommand\bfn{{\mathbf n}}
\newcommand\bfng{{\mathbf{n}_{\Gamma}}}
\newcommand\bfm{{\mathbf m}}
\newcommand\bfr{{\mathbf r}}

\newcommand\bfu{{\mathbf{u}}}
\newcommand\bfv{{\mathbf{v}}}
\newcommand\bfw{{\mathbf w}}

\newcommand\bfz{{\mathbf z}}
\newcommand\bfA{{\mathbf A}}
\newcommand\bfB{{\mathbf B}}

\newcommand\bfH{\bm{H}}

\newcommand\bfI{{\bm I}}
\newcommand\bfF{{\mathbf F}}
\newcommand\bfG{{\mathbf G}}
\newcommand\bfK{{\mathbf K}}

\newcommand\bfV{{\mathbf V}}
\newcommand\bfP{{\mathbf{P}}}

\newcommand{\eu}{\bfe_\bfu}

\newcommand{\ep}{\bfe_p}
\newcommand{\el}{\bfe_{\lambda}}

\newcommand{\Ga}{\Gamma}
\newcommand{\Gah}{\Gamma_h}
\newcommand{\Galin}{\Gamma^{(1)}_h}
\newcommand{\ds}{d\sigma}

\newcommand{\dsh}{d\sigma_h}
\newcommand{\nbg}{\nabla_{\Gamma}}
\newcommand{\nbgcov}{\nabla_{\Gamma}^{cov}}
\newcommand{\nbgcovh}{\nabla_{\Gamma_h}^{cov}}

\newcommand{\nbgh}{\nabla_{\Gamma_h}}
\newcommand{\nbglin}{\nabla_{\Gamma_h^{(1)}}}

\newcommand{\divg}{\textnormal{div}_{\Gamma}}
\newcommand{\divgh}{\textnormal{div}_{\Gah}}

\newcommand{\nb}{\nabla}

\newcommand{\R}{\mathbb{R}}

\def \to {\rightarrow}
\newcommand{\Th}{\mathcal{T}_h}
\newcommand{\Thlin}{\mathcal{T}_h^{(1)}}
\newcommand{\Ttilde}{\Tilde{T}} 
\newcommand{\Ih}{\widetilde{I}_h}
\newcommand{\Ihz}{\widetilde{I}_h^{SZ}}

\newcommand{\Ihl}{{I}_h}

\newcommand{\vhl}{\mathbf{v}_h^{\ell}}
\newcommand{\vh}{\mathbf{v}_h} 

\newcommand{\uh}{\mathbf{u}_h}

\newcommand{\wh}{\mathbf{w}_h}
\newcommand{\whl}{\mathbf{w}_h^{\ell}}
\newcommand{\zh}{\mathbf{z}_h}

\newcommand{\qh}{q_h}

\newcommand{\qhl}{q_h^\ell}
\newcommand{\ph}{p_h}

\newcommand{\lh}{\lambda_h}

\newcommand{\thbfu}{\theta_{\bfu}}

\newcommand{\nh}{\bfn_h}
\newcommand{\nhl}{\bfn_h^{\ell}}
\newcommand{\nhtil}{\Tilde{\bfn}_h}

\newcommand{\muh}{\mu_h}
\newcommand{\mh}{\bfm_h}

\newcommand\bfPg{{\mathbf{P}_{\Ga}}}
\newcommand\bfPh{{\mathbf{P}_h}}

\newcommand\Bhg{{\bm B_h}}

\newcommand{\ah}{a_h}

\newcommand{\bhtil}{b_h^{L}}
\newcommand{\btil}{\Tilde{b}}
\newcommand{\bh}{b_h}
\newcommand{\chtil}{\Tilde{c}_h}

\newcommand\norm[1]{||#1||}

\begin{document}

\author{Charles M. Elliott, Achilleas Mavrakis}
\title{SFEM for the Unsteady Navier-Stokes Equations On A Stationary Surface}
\address{Mathematics Institute, Zeeman Building, University of Warwick, Coventry CV4 7AL, UK}
\email{\href{mailto:c.m.elliott@warwick.ac.uk}{c.m.elliott@warwick.ac.uk}, \href{mailto:Achilleas.Mavrakis@warwick.ac.uk}{Achilleas.Mavrakis@warwick.ac.uk}}
\date{}

\maketitle

\begin{abstract}
In this paper we consider a fully discrete numerical method for the unsteady Navier-Stokes  equations on a smooth closed stationary surface in $\mathbb{R}^3$. We use the surface finite element method (SFEM) 
with a generalized Taylor-Hood finite element pair 
$\mathrm{\mathbf{P}}_{k_u}$-- $\mathrm{P}_{k_{pr}}$-- $\mathrm{P}_{k_{\lambda}}$, where we enforce the tangential condition of the velocity field weakly, by introducing an extra Lagrange multiplier $\lambda$. Depending on the richness of the finite element space involving this extra Lagrange multiplier we present a fully discrete stability and error analysis. For the velocity, we establish optimal $L^{2}(a_h)$-norm bounds ($a_h$ - an energy norm) when $k_\lambda=k_u$ and suboptimal with respect to the geometric approximation error when $k_{\lambda} = k_u-1$ (optimal when \emph{super-parametric finite elements} are used). For the pressure, optimal $L^2(L^2)$-norm error bounds are established when $k_\lambda=k_u$. Assuming further regularity assumptions for our continuous problem, we are also able to show optimal convergence (using \emph{super-parametric finite elements} again) when $k_\lambda=k_u-1$. Numerical simulations that confirm the established theory are provided, along with a comparative analysis against a penalty approach.

\end{abstract}

\section{Introduction}
In mathematical modeling, fluid equations posed on surfaces emerge naturally as ways to describe biological phenomena, like lipid membranes, material science in foams and biophysics, see \cite{ArroyoSimone2009,Reuther2018,TorresMilArroyo2019}. For example, assuming fluid films fixed in space leads to the surface Navier--Stokes equations as described in \cite{Miura2020,scriven1960dynamics}. Here, we are interested in surface finite elements schemes for the Navier--Stokes equations posed on sufficiently smooth closed stationary surfaces $\Ga \subset \mathbb{R}^3$. The  \emph{(unsteady) generalised incompressible  surface Navier-Stokes} problem is to: Find the velocity field $\bfu(\cdot,t)  : \Ga\times[0,T] \to \R^3$, surface pressure $p(\cdot,t)  : \Ga\times[0,T] \to \R$ with $\int_\Ga p ds =0$  and Lagrange multiplier $\lambda(\cdot,t) : \Ga\times[0,T] \to \R$ that solves on $\Gamma$ 
\begin{align}
\begin{cases}
    \label{eq: unstead NV Lagrange}
    \rho\big(\partial_t\bfu + (\bfu\cdot \nbgcov) \bfu \big)  - 2\mu\divg(E(\bfu)) + \nbg p + \bfu  + \lambda \bfng = \bff, \\
			\divg \bfu =0, \\
            \bfu \cdot \bfng =0,
    \end{cases}    
\end{align}
for all $t \in [0,T]$, where  $\rho$ the density distribution, $\mu$ is the viscosity coefficient, $\bff \in (L^2(\Gamma))^3$ a given force field, and
$\lambda$ a normal force  field to be determined. Essentially we view $\lambda$ as a Lagrange multiplier for the constraint that the normal component of $\bfu$ is zero. Such models have been presented and studied in \cite{jankuhn2018incompressible}, \cite{elliott2024sfem} for the \emph{steady Stoke} case. Here   the   rate-of-strain (deformation) tensor $E(\cdot)$ \cite{Miura2020}, is defined for an arbitrary vector field $\bfv$, by 	
\begin{equation}
		\begin{aligned}\label{eq: Deformation tensor}
			E(\bfv)  = \half (\nbgcov \bfv + \nbg^{cov,t}\bfv).
		\end{aligned}
	\end{equation}
where  $ \nbgcov \bfv := \bfPg \nbg \bfv $ (with adjoint $\nbg^{cov,t}\bfv=(\nbg \bfv)^t\bfPg$) and  $\bfPg$ denotes the projection to the tangent space of $\Gamma$.

Similarly, for $\bfu =  \bfu_T+ (\bfu\cdot \bfng) \bfng$  and $\bfu_T:=\bfPg \bfu$, we may consider the  \emph{(unsteady) tangential  surface Navier-Stokes} problem, that is : Find the velocity field $\bfu(\cdot,t)  : \Ga\times[0,T] \to \R^3$, surface pressure $p(\cdot,t)  : \Ga\times[0,T] \to \R$ with $\int_\Ga p ds =0$  that solves on $\Gamma$ that solves on $\Gamma$ 
\begin{align}
\begin{cases}
    \label{eq: unstead NV}
    \rho\big(\partial_t\bfu_T + (\bfu_T \cdot \nbgcov) \bfu_T\big)  - 2\mu\bfPg\divg(E(\bfu_T)) + \nbg p + \bfu_T= \bfPg \bff, \\
			\divg \bfu_T =0.
    \end{cases}    
\end{align}
for all $t \in [0,T]$. Similar systems have been established and studied in \cite{WilkeNS2021,Chan2016NSManifold} and references therein.

In this work, we  consider a finite element discretization of \eqref{eq: unstead NV Lagrange} using \emph{Taylor-Hood} surface finite elements. Surface finite elements methods (SFEM) have been studied extensively in the literature; see \cite{DziukElliott_acta, EllRan21,DziukElliott_SFEM} for elliptic and parabolic problems on a stationary and evolving surfaces $\Ga$. There has also been extentions to vector or tensor-valued functions for elliptic equations \cite{hansbo2020analysis,Hardering2022}. Recently, there have been numerous studies around surface Stokes problems \cite{demlow2024tangential,bonito2020divergence,reusken2024analysis,elliott2024sfem} and Navier-Stokes \cite{krause2023surface,KrauseDeform_2023,Reuther2018,fries2018higher}
on evolving and stationary surfaces applying the surface finite element methods.

Discretizations of fluid problems on surfaces pose a lot of difficulties, such as  the approximation of  covariant derivatives, geometric quantities, e.g. mean curvature, and the imposition of the tangential condition on the velocity field. There have been a plethora of papers and techniques dealing with such problems, each with its own advantages and disadvantages. The tangential surface Stokes problem has been studied using TraceFem in \cite{jankuhn2021Higherror,jankuhn2021trace,olshanskii2019penalty}, which uses a background bulk mesh, not fitted to the surface. 
Moreover, surface finite element methods have been applied and analyzed, with different formulations, various finite element spaces, and techniques. The main difficulty in all of them is enforcing  the tangential constraint in the discrete space. This would not  lead to $H^1$-conforming finite elements, due to the discontinuity across the edges of the triangulated mesh. One approach that has been employed to address this issue is using the stream function formulation \cite{BranReusSteam2020}, where the tangentiality condition is satisfied exactly.
This approach is $H^1$-conforming, however, it becomes necessary to compute various geometric quantities. Regarding approximations of solutions in the velocity-pressure formulation, approaches that enforce the tangentiality condition strongly, has also been considered. In \cite{bonito2020divergence}, an $H(\divg)$ conforming finite element method with an interior penalty approach to enforce $H^1-$ continuity weakly is analysed, and in \cite{demlow2024tangential} an $H(\divg)$ conforming finite element method is constructed with the help of surface Piola transforms w.r.t. the set of faces of each vertex. These methods are quite complicated to implement numerically. Simpler $H^1$ conforming \emph{Taylor-Hood} finite elements for the velocity-pressure formulation have also been studied, where one relaxes the tangential constraint via a penalty parameter (penalty based formulations) \cite{reusken2024analysis,hansbo2020analysis,hardering2023parametric}
(also see \cite{jankuhn2021Higherror,jankuhn2021trace,olshanskii2019penalty} for the TraceFEM case)
or via a Lagrange multiplier (Lagrange based formulations) \cite{elliott2024sfem,highorderESFEM}. Regarding the penalty formulations, to obtain optimal convergence we see that approximations of a new higher-order normal, and geometric quantities, mainly the mean curvature,  are required. Furthermore, a correct choice of the penalty parameter $\eta$ is also important; see \cite{hardering2023parametric, reusken2024analysis, hansbo2020analysis}. On the other hand, for the Lagrange-based formulation no penalty term arises, but one has to consider the choice of approximation of the extra Lagrange multiplier, as it plays a role in both the stability of the system and the optimal convergence; see \cite{elliott2024sfem}. Depending on this choice, one either establishes optimal convergence (for the full velocity) but with worse conditioning of the underlying system of equations or suboptimal with respect to the geometric approximation error; therefore \emph{super-parametric surface finite elements} have to be used (mainly for the tangential velocity). We follow this approach.




\subsection{Main results}

In this paper, we analyze a fully discrete $H^1$ conforming \emph{Taylor-Hood} surface finite element method $\mathrm{\mathbf{P}}_{k_u}$-- $\mathrm{P}_{k_{pr}}$-- $\mathrm{P}_{k_{\lambda}}$  for the time-dependent surface Navier-Stokes \eqref{eq: unstead NV Lagrange}, where the tangential constraint is enforced via an extra Lagrange multiplier $\lh$. By, $k_u$, $k_{pr}$, $k_{\lambda}$ we represent the degree of polynomials used for the velocity field $\bfu$, pressure $p$ and extra Lagrange multiplier $\lambda$ approximations, respectively. For inf-sup stability reasons we always consider $k_u = k_{pr}+1$, see \cref{Lemma: Discrete inf-sup condition Gah Lagrange UNS,lemma: L^2 H^{-1} discrete inf-sup condition Gah Lagrange UNS}, while we also consider two options in terms of the choice of approximation of $\lambda$, $\underline{k_\lambda = k_u}$ and $\underline{k_\lambda = k_u-1}$. The time-dependent surface Navier-Stokes was studied in \cite{olshanskii2019penalty}, where the authors used a TraceFEM $\mathrm{\mathbf{P}}_1-\mathrm{P}_1$ Taylor-Hood penalty formulation, with pressure stabilization and assumed exact integration over $\Ga$. 
In our case, we consider a $k_g$-order approximation of the surface, $\Gah$, and thus we deal with geometric errors arising from this approximation. That means we have to treat the approximated terms carefully and specifically the inertia term. Moreover, we add the zeroth-order term $\bfu$ in \eqref{eq: unstead NV Lagrange} only so that we can ease some calculations. It is not difficult to see that even without this term the results that will be established still hold (despite this see also \cref{remark: Lower Regularity convergence estimate UNS} where this zeroth-order term might be important in certain cases).


In this paper, we analyze the error between the continuous primitive variables of \eqref{weak lagrange hom NV} and the discrete ones arising from the discrete weak Lagrange formulation \eqref{eq: weak lagrange fully discrete UNS}, where we first prove stability and eventually error bounds. In the  main result \cref{theorem: Velocity Error Estimates UNS} we prove 
$L^{\infty}_{L^2}$ and $L^2_{\ah}$, where $\ah$ a discrete energy norm \eqref{eq: energy norm UNS}, velocity error bounds for both choices of $k_\lambda$. This is established in the case of low regularity assumption; see \Cref{assumption: Regularity assumptions for velocity estimate}, with the help of a new modified surface Ritz-Stokes map \eqref{eq: surface Ritz-Stokes projection UNS} which helps us isolate velocity errors from pressure errors. We notice that for the $\underline{k_\lambda=k_u-1}$ case, due to a limiting geometric error $\bigo(h^{k_g-1})$; see \eqref{eq: Velocity Error Estimates kl=ku-1 UNS}, one has to further assume that $k_g\geq 2$ and use \emph{super-parametric surface finite elements} to obtain optimal convergence.
However, this is not necessary once we choose $\underline{k_\lambda=k_u}$, as is evident in \eqref{eq: Velocity Error Estimates UNS}, which we note also holds in the case of \emph{planar triangulation}, i.e. $k_g=1$. For more details, see  the numerical results in \cref{sec: Numerical results} and specifically \cref{Sec: num Varying curvature surface UNS}.



\begin{tikzcd}
                               &                                                       & \boxed{\tiny\text{Regularity of solutions}} \arrow[rdd] \arrow[ldd] &                                                                            \\
                               &                                                       &                                                        &                                                                            \\
                               & \boxed{\tiny\text{Low}} \arrow[ld] \arrow[rd, "k_{\lambda} = k_u", " \Delta t \leq ch"'] &                                                        & \boxed{\tiny\text{High}} \arrow[d, "k_{\lambda}=k_u-1",
                                                            " \Delta t \leq ch"'] \\
\boxed{\tiny\text{Velocity Error Estimate}} &                                                       &\boxed{ \tiny\text{Pressures Error Estimates}}                       & \boxed{\tiny\text{Pressures Error Estimates} }                                         
\end{tikzcd}\\

The pressure bounds, on the other hand, need a bit more consideration in regards to the regularity assumptions, the choice of the approximation of the extra Lagrange multiplier and the \emph{inf-sup} condition used. This stems from the fact that to find pressure bounds one needs to deal with the discretization of the time derivative $(\uh^n-\uh^{n-1})/\Delta t$, and in turn bound it in an appropriate norm. Considering $\underline{k_{\lambda} = k_u}$, and relatively low (standard) regularity assumptions for the continuous solutions, we can derive optimal $L^2_{L^2} \times L^2_{H_h^{-1}}$ error (stability) pressure estimates, as presented in our main pressure result \cref{theorem: Pressures Error Estimate UNS} ($H_h^{-1}$ a dual norm \eqref{eq: H^-1h definition UNS}). These bounds are established with the help of a discrete inverse Stokes operator \eqref{eq: Discrete inverse Stokes UNS}, the dual energy norm estimate \Cref{lemma: dual estimate UNS} and an $L^2\times H_h^{-1}$ discrete \textsc{inf-sup} condition \eqref{eq: L^2 H^{-1} discrete inf-sup condition Gah Lagrange UNS}, therefore only establishing error bound for $\lh$ in this weaker dual norm. 
On the other hand, if \emph{further regularity assumptions} are made for the velocity; see \cref{assumption: Regularity assumptions for velocity estimate 2 UNS}, we notice that a $W^{1,\infty}$-bound for the standard Ritz-Stokes map \eqref{eq: surface Ritz-Stokes projection std UNS} and the discrete integration by parts formula \eqref{eq: integration by parts UNS} enables us to prove convergence results when $\underline{k_{\lambda}=k_u-1}$, as presented in \cref{theorem: pressure estimate HR UNS}; see \cref{remark: Lower Regularity convergence estimate UNS} for more details. In this case using the $L^2\times L^2$ \textsc{inf-sup} condition \eqref{eq: discrete inf-sup condition Gah Lagrange UNS} instead, we prove, once again, suboptimal results w.r.t. the geometric approximation error, hence the importance of employing \emph{super-parametric finite elements} for optimal convergence.

\subsection{Outline}
In \cref{sec: Differential geometry on Surfaces UNS} we introduce our notation along with our continuous function spaces. The continuous variational formulation of \eqref{eq: unstead NV Lagrange} is briefly introduced in \cref{sec: Weak formulation UNS}, while the temporal and space discretizations are discussed in the next \Cref{sec: discretization UNS}, where we also recall some basic results involving surface lifting, geometric errors and basic interpolation estimates. In \Cref{sec: fully discrete method UNS} we define the finite element spaces and bilinear forms that will be used throughout, and introduce new discrete surface Ritz-Stokes and Leray projections that are necessary leading up to the stability and a-priori error bounds. In that section, lastly, we present our fully discrete variational formulation \eqref{eq: weak lagrange fully discrete UNS}. In \Cref{Sec: Stability analysis UNS} we present the velocity and pressure stability results. The error analysis is discussed in \Cref{sec: error analysis UNS}. In the first part, we consider relatively \emph{low regularity assumptions} and establish velocity error estimates independent of the choice of approximation of the extra Lagrange multiplier $k_\lambda$. Then, we prove error bounds for the pressure that hinge on the choice of $k_\lambda$. Namely, for $k_\lambda=k_u$, with the previous \emph{regularity assumptions} in mind, we prove optimal error pressure bounds in an $L^2_{L^2}\times L^2_{H_h^{-1}}$-norm, while for $k_\lambda=k_u-1$
we prove optimal $L^2_{L^2}\times L^2_{L^2}$-norm error bounds, considering \emph{additional regularity assumptions}. Finally, numerical results are presented in \Cref{sec: Numerical results}, which support the theoretical results.


\section{Notation and Differential Geometry}\label{sec: Differential geometry on Surfaces UNS}
\subsection{The closed smooth surface}
We will consider our surface $\Ga$ to be a closed, oriented, compact $C^m$  two-dimensional hypersurface embedded in $\mathbb{R}^3$, for some fixed $m$. In \cite{elliott2024sfem} it was required to assume that $m\geq 4$. This assumption arises due to technical results related to the proof of the discrete inf-sup stability. We also require this condition here for similar reasons. Note that  since  $\Gamma$ is the boundary of an open set we choose the orientation by setting  $\bfng$ to be the unit outward pointing normal to $\Ga$.  Let $d(\cdot) : \mathbb{R}^3 \to \mathbb{R}$ be the signed distance function, defined in \cite{DziukElliott_acta}, and
for $\delta>0 $, let $U_\delta \subset \mathbb{R}^3$ be the tubular neighborhood $U_{\delta} = \{x \in \R^3\ : |d(x)| <\delta\}$. Then we may define the closest point projection mapping to $\Ga$ as $\pi(x) = x - d(x)\bfng(\pi(x)) \in \Ga$ for each $x \in U_\delta$  for $\delta>0$ small enough, see \cite{GilTrud98,DziukElliott_acta}. Moreover, $\nb d(x) = \bfng(\pi(x))$ for $x \in U_\delta$. We can see that $d(\cdot)$ and $\pi(\cdot)$ are of class $C^m$ and $C^{m-1}$ on $\overline{U_{\delta}}$. 


Using the projection $\pi$ we may extend functions $\bfv : \Ga \to \mathbb{R}^3$  on $\Ga$ to the tubular neighbourhood $U$, by 
\begin{equation}
\label{eq: smooth extension UNS}
    \bfv^e(x) = \bfv(\pi(x)), \ \ \ x \in U_{\delta}.
\end{equation}
This extension is constant in the normal direction of $\Ga$ and thus contains certain useful additional properties in comparison to other regular extensions. We define the orthogonal projection operator onto the tangent plane $\bfPg (x) = \bfI - \bfng(x) \otimes \bfng(x)$ for $x \in \Ga$ that satisfies $\bfPg^T=\bfPg^2=\bfPg$, and $\bfPg \cdot \bfng=0$. The extension of $\bfPg$ to the $\delta-$strip $U_\delta$, is defined as 
\begin{equation}
    \bfP =  \bfI - \bfn(x) \otimes \bfn(x), \ \text{ for } x \in U_\delta,
\end{equation}
where we have used $\bfn:=\bfng^e = \bfng\circ \pi $. We can then easily see that $\bfP|_{\Ga} = \bfPg$.

\subsection{Scalar functions and vector fields}
We start by defining useful quantities for \emph{scalar} functions following \cite{DziukElliott_acta}. For a function $f \in C^1(\Ga)$ we define the \emph{tangential (surface) derivative} of $f$ as 
\begin{equation}
    \nbg f(x) = \bfPg(x) \nb f^e(x), \ \ \text{with } \ \underline{D}_if(x) =  P_{ij}\partial_jf^e(x), \ \ \ x\in\Ga, 
\end{equation}
where in the above we used the \emph{Einstein convention}, i.e. repeated indices are summed up and where $(\nb f^e)_i = \partial_i(f^e)$, $1 \leq i \leq 3$ is the Euclidean derivative and $\nbg f(x) = (\underline{D}_1f(x),$ $ \underline{D}_2f(x),$ $\underline{D}_3f(x))^t$ is a column vector.
As noted in \cite[Lemma 2.4]{DziukElliott_acta} the definition of the tangential gradient is independent of the extension. \\

We now introduce derivatives of surface \emph{vector fields} as in \cite{fries2018higher,Miura2020,elliott2024sfem}. We begin by defining tangential derivatives, analogous to the scalar case, as Euclidean derivatives of extended quantities (this time-vector fields) and proceed to establish the covariant derivatives. 

\begin{definition}[Tangential derivative]
Let $\bfv : \Ga \to \mathbb{R}^3$, $\bfv = (v_1, v_2, v_3)^T$ be a smooth vector field, and $\bfv^e$ the smooth extension in \eqref{eq: smooth extension UNS}, the tangential gradient of this vector field  is defined by 
\begin{equation}\label{eq: tangential derivative UNS}
    (\nbg \bfv)_{ij} = (\nb \bfv^e \bfPg)_{ij} =\partial_l v^e_iP_{jl}  ,
\end{equation}
 where $(\nbg \bfv)_{ij} = \underline{D}_j v_i$. This derivative is expressed as a $3 \times 3$ matrix, which may be presented as 
\begin{equation}\label{eq: tangential derivative matrix UNS}
\nbg \bfv = \nbg \begin{bmatrix}
v_1 \\
v_2 \\
v_3
\end{bmatrix} = \begin{bmatrix}
\underline{D}_1 v_1 & \underline{D}_2 v_1 & \underline{D}_3 v_1\\
\underline{D}_1 v_2 & \underline{D}_2 v_2 & \underline{D}_3 v_2\\
\underline{D}_1 v_3 & \underline{D}_2 v_3 & \underline{D}_3 v_3
\end{bmatrix}.
\end{equation}

\end{definition}
In the above we used the \emph{Einstein convention}, i.e. repeated indices are summed up. We note again that due to the extension \eqref{eq: smooth extension UNS}, $\nbg \bfv = \nb \bfv^e|_{\Ga}$ holds. With the help of the tangential derivatives for vector fields we may define the Weingarten map $\bfH$ on the $\delta-$strip $U_{\delta}$  (see \cite{DziukElliott_acta, GilTrud98}), to quantify the curvature, as
\begin{equation}\label{eq: Weingarten map}
    \bfH := \nbg \bfn =\nbgcov \bfn, \ \text{ with } \ \bfH \bfn=0, \ \ \ \bfH \bfP = \bfP \bfH =\bfH,
\end{equation}
where we observe that it is a tangential (in-plane) tensor to $\Ga$. Note that $\bfH$ is of class $C^{m-2}$ and thus bounded on $\Ga$. It has two non-zero eigenvalues $\kappa_1, \, \kappa_2$, called principal curvatures and the \emph{mean curvature} of $\Ga$ is then obtained as $\kappa:= tr(\bfH)$.

\begin{definition}[Covariant derivatives]\label{def: Covariant derivatives UNS}
Let $\bfv : \Ga \to \mathbb{R}^3$, $\bfv = (v_1, v_2, v_3)^T$ be a smooth vector field, and $\bfv^e$ the smooth extension in \eqref{eq: smooth extension UNS}, the covariant derivative $\nbgcov \bfv$ is then defined by
\begin{equation}
    (\nbgcov \bfv)_{ij} = (\bfPg \nbg \bfv)_{ij} = (\bfPg \nb \bfv^e|_{\Ga} \bfPg)_{ij} = P_{jk}\partial_k v^e_lP_{li},
\end{equation}
where the derivative can be expressed as a  $3 \times 3$ matrix.
\end{definition}
One has to distinguish these two different gradient operators. Notice that the covariant derivative, compared to the tangential derivative defined in \eqref{eq: tangential derivative UNS}, is a tangential tensor field. Using now the Einstein summations we see that the surface divergence of a vector function $\bfv : \Ga \to \mathbb{R}^3$ is given by
\begin{equation}\label{eq: divg vector definition UNS}
\divg(\bfv) = tr(\nbgcov \bfv) = tr(\bfPg \nabla \bfv^e \bfPg)= P_{ik}\partial_kv_l^eP_{li} = tr(\nabla \bfv^e \bfPg) =  tr(\nbg \bfv ).
\end{equation}
One can always split the a vector field $\bfv$ into a tangent and a normal component \ie  $\bfv = \bfv_T + \bfv_n$ with $\bfv_T = \bfPg \bfv$ and $\bfv_n = (\bfv \cdot \bfng)\, \bfng$, where also the following useful formulae holds 
\begin{equation}
    \begin{aligned}\label{eq: split cov UNS}
       \nbgcov \bfv = \nbgcov \bfv_T + \bfH v_n.
       \end{aligned}
\end{equation}
We also define with the help of \eqref{eq: divg vector definition UNS}  the surface divergence of a tensor function $\bfF : \Ga \to \mathbb{R}^{3 \times 3}$:
\begin{equation}
    \begin{aligned}
       \divg(\bfF) :=  \begin{bmatrix}
\divg(\bfF_{1,j}) \\
\divg(\bfF_{2,j}) \\
\divg(\bfF_{3,j})
\end{bmatrix}, \quad j=1,2,3.
       \end{aligned}
\end{equation}
By $(\bfv\cdot\nbgcov)\bfv$ we denote the tangent vector
\begin{equation}
    \begin{aligned}
        \big((\bfv\cdot\nbgcov)\bfv\big)_{i} = v_jP_{jk}\partial_kv_l^eP_{li} = \big(\nbgcov\bfv\bfv_T\big)_{i} = \big(\bfPg\nbg\bfv \bfv_T\big)_{i}.
    \end{aligned}
\end{equation}
Finally, we present the following integration by parts rule \cite{DziukElliott_acta,fries2018higher} for smooth enough functions $\bfv$ and $\xi$:
\begin{equation}
    \begin{aligned}\label{eq: integration by parts cont UNS}
        \int_{\Ga}\bfv\cdot \nbg\xi \, \ds = -\int_{\Ga}\xi \divg\bfv \, \ds + \int_{\Ga}\kappa\xi(\bfv\cdot \bfng) \, \ds.
    \end{aligned}
\end{equation}

\subsection{Function spaces}\label{sec: Function spaces UNS}

We now define function spaces on the surface $\Ga$ as in \cite{DziukElliott_acta}. By $L^p(\Ga)$, $p \in [1,\infty]$, we define the function space that consists of functions $\xi : \Ga \to \mathbb{R}$ that are measurable w.r.t. the surface measure $\ds$, endowed with the standard $L^p$-norms, \ie $\norm{\cdot}_{L^p(\Ga)}$. By $(\cdot,\cdot)_{L^2(\Ga)}$ we denote the usual $L^2(\Ga)$ inner product.
By $W^{k,p}(\Ga)$  we denote the standard function Sobolev spaces as presented in \cite{DziukElliott_acta}, while  we consider $\mathbf{W}^{k,p}(\Ga)=(W^{k,p}(\Ga))^3$  to be the natural extension to vector-valued functions. Specifically, for $p=2$ we write $(H^{k}(\Ga))^n$ with $n=1,3$, which are, in fact, Hilbert spaces. So, for a function  $\xi : \Ga \to \mathbb{R}$ with the corresponding norm is given by
\begin{equation}
    \begin{aligned}
    \norm{\xi}_{H^k(\Ga)}^2 = \sum_{j=0}^m \norm{\nbg^j \xi}^2_{L^2(\Ga)},
    \end{aligned}
\end{equation}
where $\nbg^j$ denote all the weak tangential derivative of order $j$, i.e. $\nbg^j \xi = \underbrace{\nbg \cdot \cdot \cdot \nbg}_{\text{ j times}} \xi$. One can easily extend in the case vector-valued functions pointwise. More specifically for a vector fields $\bfv : \Ga \to \mathbb{R}^3$ the space $\bfH^1(\Ga) = (H^1(\Ga))^3$ is equipped with the following norm
\begin{equation}
    \begin{aligned} \label{eq: H1 norm definition UNS}
    \norm{\bfv}_{H^1(\Ga)}^2 = \norm{\bfv}_{L^2(\Ga)}^2 + \norm{\nbg \bfv}_{L^2(\Ga)}^2.
    \end{aligned}
\end{equation}

\noindent We set 
$$L^2_0(\Ga) :=\{\xi \in L^2(\Ga) | \ \int_\Ga \xi \, \ds =0\},$$
equipped with the standard $L^2$-norm, while the subspace of tangential vector fields is given by 
$$\bfH^1_T := \{\bfv \in H^1(\Ga)^3 \: | \: \bfv \cdot \bfng =0\},$$
endowed with the $H^1$-norm \eqref{eq: H1 norm definition UNS}. We also consider the subspace of \emph{weakly tangential divergence-free} vectors 
\begin{equation}
    \bfV^{div} = \{\bfv \in \bfH^1(\Ga) \, :\, b^L(\bfv,\{q,\xi\}) =0 \text{ for all } \{q,\xi\} \in L^2_0(\Ga)\times L^2(\Ga) \}.
\end{equation}


\section{Variational formulation}\label{sec: Weak formulation UNS}

\subsection{Preliminaries and Bilinear forms}
 \begin{definition}
      For $\bfu, \bfv \in  \bfH^1(\Ga)  $  and $\{q,\lambda\} \in  L_0^2(\Ga) \times L^2(\Ga)$ we define  the following 
      \begin{align}
          \label{eq: regular bilinear a UNS}
          a(\bfu,\bfv) &:= \int_\Ga E(\bfu)  : E(\bfv) \, \ds + \int_{\Ga}  \bfu\cdot\bfv \, \ds,\\
          \label{eq: regular bilinear c UNS}
         c(\bfu;\bfu,\bfv) &:= \int_\Ga (\bfu \cdot \nbgcov) \bfu \cdot \bfv \, \ds = \int_\Ga (\nbgcov \bfu \bfu) \cdot \bfv \, \ds,\\
         b^L(\bfu,\{q,\lambda\}) &:= -\int_\Ga q \ \divg \bfu_T  \, \ds + \int_\Ga \lambda u_n\, \ds.
      \end{align}
Specifically, for divergence-free $\bfu \in \bfH^1_T$ velocities $(\divg(\bfu)=0)$, using integration by parts \eqref{eq: integration by parts cont UNS} we can \emph{skew-symmetrize} $c(\bullet;\bullet,\bullet)$ in the following way
\begin{equation}
    \label{eq: continuous c error formula init UNS}
        c(\bfu;\bfu,\bfv) = \frac{1}{2}\Big(\int_{\Ga}((\bfu\cdot\nbgcov)\bfu) \cdot\bfv \, \ds - \int_{\Ga}((\bfu\cdot\nbgcov)\bfPg\bfv )\cdot\bfu \, \ds\Big).
\end{equation}
 \end{definition}


 The following inequality, known as \emph{\bf surface Korn's inequality}, was established in  \cite{jankuhn2018incompressible}. There  exists a constant $c_K>0$  such that
	\begin{equation}
		\label{eq: Korn inequality UNS}
		 c_K \norm{\bfw}_{H^1(\Ga)}\le \norm{E(\bfw)}_{L^2(\Ga)} + \norm{\bfw}_{L^2(\Ga)}~~\text{ for all } \bfw\in \bfH^1_T.
	\end{equation}
In the case where $\bfw\in \bfH^1_T$ we also define the energy norm 
\begin{equation}\label{eq: cont energy norm UNS}
    \norm{\bfw}_{a} = \norm{E(\bfw)}_{L^2(\Ga)} + \norm{\bfw}_{L^2(\Ga)}.
\end{equation}
We present the following Helmholtz-Leray decomposition, which will be useful later on in order to prove a stability estimate in a discrete dual energy norm; see \cref{lemma: dual estimate UNS} in \Cref{Sec: Stability analysis UNS}.
\begin{lemma}[Helmholtz-Leray decomposition]\label{lemma: Helmholtz-Leray decomposition UNS}
For every $\bfv \in \bfH^1(\Ga)$ there exists unique $\phi \in H^2(\Ga)$, $\bfv_n \in \bfH^1(\Ga)$ and $\Pi^{div}(\bfv)\in \bfV^{div}$ such that 
\begin{equation*}
\bfv = \bfv_n + \Pi^{div}(\bfv) + \nbg\phi.
\end{equation*}

\end{lemma}
\noindent See Appendix \ref{appendix: leray UNS} for a proof of this decomposition.

\subsection{Lagrange multiplier variational formulations}\label{sec: Lagrange formulation}

Consider \eqref{eq: unstead NV Lagrange} with density distribution $\rho=1$ and viscosity coefficient $\mu=1/2$. By  applying integration by parts, we look for a solution of the following variational problem:\\

\noindent {\bf (LP)}: Given $\bff \in (L^2(\Ga))^3$, determine $\bfu \in L^{\infty}([0,T];L^2(\Ga))\cap L^2([0,T];\bfH^1(\Ga))$, with $\partial_t\bfu \in L^2([0,T];\bfH^{-1}(\Ga))$ and $\{p,\lambda\} \in ( L^2([0,T];L^2_0(\Ga))\times  L^2([0,T];L^2(\Ga)))$ satisfying $ \bfu(\cdot,0)= \bfu_0 \in L^2(\Ga)$ on $\Ga$ such that,
\begin{align}
\begin{cases}
    \label{weak lagrange hom NV}
        \langle\partial_t\bfu,\bfv\rangle_{H^{-1}(\Ga),H^1(\Ga)} + a(\bfu,\bfv) \ + c(\bfu;\bfu,\bfv)\ + \!\!\!\!&b^L(\bfv,\{p,\lambda\}) = (\bff,\bfv)_{L^2(\Ga)} \ \ \ \ \ \text{for all } \bfv\in \bfH^1(\Ga),\\
        &b^L(\bfu,\{q,\xi\})=0 \ \ \  \text{ for all } \{q,\xi\}\in (L_0^2(\Ga)\times L^2(\Ga)),
    \end{cases}
\end{align} 
and for a.e. $t\in [0,T]$. For uniqueness and well-posedness of non-time dependent surface Stokes analogue see \cite{elliott2024sfem}.
In our case, the existence and well-posedness of the second system \eqref{eq: unstead NV} has been proved in \cite{WilkeNS2021}. From this, the uniqueness and existence of our system \eqref{weak lagrange hom NV} can be derived, since their tangent solutions (and therefore pressure) agree; see also \cite{olshanskii2019penalty} for a similar argument in the surface Stokes case. We also have the following \emph{inf-sup} condition result presented in  \cite{elliott2024sfem}. 
\begin{lemma}[Inf-sup Condition]\label{lemma: inf-sup cont lagrange UNS}
   Let $\Ga\in C^2$ and compact. Then there exists a constant $c_b$ such that,
\begin{equation}
    \begin{aligned}\label{infsup lagrange}
         \inf_{\{p,\lambda\} \in (L^2_0(\Ga)\times H^{-1}(\Ga))} \ \sup_{\bfv\in \bfH^1(\Ga)} \frac{\btil(\bfv,\{p,\lambda\})}{\norm{\bfv}_{H^1(\Ga)} \norm{\{p,\lambda \}}_{L^2(\Ga)\times H^{-1}(\Ga)}} \geq c_b > 0.
    \end{aligned}
\end{equation}
\end{lemma}

 Considering the above lemma and following standard Faedo-Galerkin techniques \cite{Temam1979NavierStokesET}, well-posedness can be established. We omit further details, but one can see that the following a-priori estimates hold; see also \cite{olshanskii2019penalty},
 \begin{equation}
     \begin{aligned}
         \norm{\partial_t\bfu}_{L^2([0,T];H^{-1}(\Ga))} &+ \norm{\bfu}_{L^{\infty}([0,T];L^2(\Ga))} +  \norm{\bfu}_{L^{2}([0,T];H^1(\Ga))}\\
         &+  \norm{\{p,\lambda\}}_{L^{2}([0,T];L^2(\Ga))} \leq \norm{\bfu_0}_{L^2(\Ga)} + \norm{\bff}_{L^{2}([0,T];L^2(\Ga))}.
     \end{aligned}
 \end{equation}

\section{Preliminaries}\label{sec: discretization UNS}
In this section we begin by presenting our temporal discretization scheme (based on the backward Euler time discretization), along with some basic results involving the space discretization, mainly geometric errors arising from the approximation of the discrete surface $\Gah$, surface lifting to $\Ga$ etc., cf. \cite{DziukElliott_acta,EllRan21}. Finally, we introduce some known interpolants and their estimates.

\subsection{Time-stepping scheme}\label{sec: Time-stepping scheme UNS}
For the time discretization we split the time interval $I=[0,T]$ uniformly into $N$ time intervals $I_n=[t_{n-1},t_{n}]$ with nodes $t_n = n\Delta t$, $n=0,1,...,N$ and uniform time step size $\Delta t = t_n - t_{n-1} = T/N$. We approximate the first derivative as followed
\begin{equation}
    \label{eq: discr first derivative}
    \partial_t \bfu \approx \frac{\bfu^n - \bfu^{n-1}}{\Delta t}, 
\end{equation}
where $\bfu^n = \bfu(t^n)$, which corresponds to a $BDF1$ type time discretization. So the semi-discrete variational formulation of \eqref{weak lagrange hom NV} reads as:\\

\noindent For $n=0,1,...\,,N$ find $\bfu^n \in \bfH^1(\Ga)$, $\{p^n,\lambda^n\} \in L^2_0(\Ga)\times L^2(\Ga)$, given $\bfu^{n-1} \in \bfH^1(\Ga)$, $\{p^{n-1},\lambda^{n-1}\} \in L^2_0(\Ga)\times L^2(\Ga)$, such that 
\begin{align}
\begin{cases}
    \label{eq: discretized weak lagrange hom NV}
        (\frac{\bfu^n-\bfu^{n-1}}{\Delta t},\bfv)_{L^2(\Ga)} + a(\bfu^n,\bfv) \ + c(\bfu^{n-1};\bfu^n,\bfv)\ + \!\!\!\!&b^L(\bfv,\{p^n,\lambda^n\}) = (\bff^n,\bfv)_{L^2(\Ga)},\\
        &b^L(\bfu^n,\{q,\xi\})=0,
    \end{cases}
\end{align}
for all $\bfv\in \bfH^1(\Ga)$ and $\{q,\xi\}\in L^2_0(\Ga)\times L^2(\Ga)$.
Notice that since we use $BDF1$ time discretization we have linearized the trilinear form in the following sense, by replacing $\bfu^{n}\nbgcov \bfu^n$ with $\bfu^{n-1}\nbgcov \bfu^n$.


\subsection{Space Discretization}\label{subsection: Space discretization UNS}

\subsubsection{Triangulated surfaces}
We now want to discretize our surface $\Ga$. We construct the discrete triangulated surface $\Gah$.
But first, one approximates $\Ga$ by a polyhedral approximation $\Galin$ contained in the tubular neighborhood $U_\delta$ i.e. $\Galin \subset U_\delta$, whose vertices lie on $\Ga$ and which  is equipped with an admissible and conforming subdivision of triangulation $\Thlin$ so that $\Galin = \cup_{\Ttilde \in \Thlin}\Ttilde$, see also \cite[Section 6.2]{EllRan21}.

We work with polynomial Lagrange finite elements, thus point evaluations (on each vertex) act as our degrees of freedom. So, for $k_g \geq 2$ we consider $\{\phi_1,...,\phi_{n_{k_g}}\}$ be the Lagrange basis functions of degree $k_g$ associated to a set of nodal points $\{\tilde{a}_j\}_{j=1}^{n_{k_g}}$ of a triangle $\Ttilde \in \Thlin.$ Then with the help of the $k_g-$th order Lagrange interpolant \cite{DziukElliott_acta} we define  the higher-order geometry
\begin{equation*}
    \begin{aligned}
        \pi_{k_g}(\tilde{x})  = \mathcal{I}_h^{k_g}\pi(\tilde{x}) = \sum_{i=1}^{n_{k_g}} \pi(\tilde{a}_j)\phi_i(\tilde{x}), \qquad \text{ for } \tilde{x} \in \Ttilde.
    \end{aligned}
\end{equation*}
Applying the above on each element $\Ttilde \in \Thlin$ yields the \emph{admissible} and \emph{conforming} high-order triangulation $\Th$, written as $\Th =  \cup\{F_T(\Ttilde) \ | \ \Ttilde \in \Thlin\}$ and thus generates the high-order surface approximation 
$$\Gah = \bigcup_{\Ttilde\in \Th}\{F_T(\Tilde{x}) \ | \ \tilde{x} \in \Ttilde\}.$$
We denote by  $h_{T}$ the diameter of a simplex $T \in \Th$ and set $h := \text{max} \{ \text{diam}(T) \, : \, T \in \Th\}$. We assume the subdivision $\Th$ is shape-regular hence the following inequality holds  
\begin{equation}\label{eq: shape regularity UNS}
    c_1 h_{T_i} \leq h_{T} \leq  c_2 h_{T_i}, \quad \text{for all } T_i \in \omega_{T}.
\end{equation} 
for fixed constants $c_1, \,c_2$, where $\omega_{T}$ denotes the number of neighboring elements to $T$ assumed to be bounded by a fixed constant; see \cite{camacho20152}. We further assume that the family of triangulation is also \emph{quasi-uniform} \cite[Definition 6.22]{EllRan21}, hence associated with a quasi-uniform constant $\sigma$ such that
\begin{equation}\label{eq: quasi-uniform}
 {\max_{T\in \Th} h_T} / {\min_{T\in \Th} h_T} \leq \sigma.
\end{equation}

We now denote element-wise outward unit normal to $\Gah$ by $\nh$ and define the discrete projection $\bfP_h$ onto the tangent space of $\Gah$ by
\begin{equation*}
    \bfPh(x) = \bfI - \nh(x) \otimes \nh(x), \quad x \in T, \text{ where } T \in \Th.
\end{equation*}
Also, the corresponding discrete Weingarten map is $\bfH_h := \nbgcovh\nh.$ As such, one can define the discrete tangential surface operators $\nbgh, \, \nbgcovh$ \eqref{eq: Eh forms UNS}.
For the curved triangulation, we may also consider the set of all the edges of the triangulation $\Th$, denoted by  $\mathcal{E}_h$. 
For each edge, we define the outward pointing unit co-normals $\bfm_h^{\pm}$ with respect to the two adjacent triangles $T^{\pm}$. Notice that on discrete surfaces in general
$$[\mh]|_E = \mh^+ + \mh^- \neq 0.$$ 
For more details on the construction of the triangulated surface, see \cite{elliott2024sfem,EllRan21} and references therein.

\subsubsection{Geometric approximation errors}
The approximation of a smooth surface $\Ga$ by a discrete triangulated surface $\Gah$ of order $k_g$ introduces geometric errors. Here we present some known key geometric estimates, see \cite{highorderESFEM,Olshanskii2014ASurfaces}. Define $\bfP, \bfn$ the extension of the $\bfPg,\,\bfng$ to the discrete surface $\Gah$.
\begin{lemma}[Geometric Errors]\label{lemma: Geometric errors UNS}
For $\Gah$ and $\Ga$ as above, we have the following estimates
\begin{equation}
    \begin{aligned}\label{eq: geometric errors 1 UNS}
        \norm{d}_{L^\infty(\Gah)} \leq ch^{k_g + 1}, \quad &\norm{\bfn- \nh}_{L^\infty(\Gah)} \leq ch^{k_g}, \quad \norm{\bfH- \bfH_h}_{L^\infty(\Gah)} \leq ch^{k_g -1}, \\
        &\norm{[\mh]}_{L^{\infty}(\mathcal{E}_h)} \leq ch^{k_g}, \quad \norm{\bfP[\mh]}_{L^{\infty}(\mathcal{E}_h)} \leq ch^{2k_g}.
    \end{aligned}
\end{equation}
and thus we may derive the approximations,
\begin{equation}
    \begin{aligned}\label{eq: geometric errors 2 UNS}
       &\norm{\bfP- \bfPh}_{L^\infty(\Gah)} \leq ch^{k_g}, \quad \norm{\bfP\cdot\nh}_{L^\infty(\Gah)} \leq ch^{k_g}, \quad \norm{1-\bfn \cdot \nh}_{L^\infty(\Gah)} \leq ch^{k_g+1} \\
       &  \norm{\bfPh\cdot \bfn}_{L^\infty(\Gah)} \leq ch^{k_g}. \quad 
    \end{aligned}
\end{equation}
It also follows that 
\begin{equation}
    \begin{aligned}\label{eq: geometric errors 3 UNS}
      \norm{\bfPh \bfP - \bfP}_{L^\infty(\Gah)} \leq c h^{k_g}, \text{ and } \qquad \norm{\bfP \bfPh \bfP - \bfP}_{L^\infty(\Gah)} \leq c h^{k_g+1}.
    \end{aligned}
\end{equation}
All the above constants $c$ are independent of the mesh parameter $h$.
\end{lemma}
\begin{proof}
 The estimates in the first line of \eqref{eq: geometric errors 1 UNS} are already known; see \cite{DziukElliott_acta, Demlow2009}. The results in the second line are shown in \cite{Olshanskii2014ASurfaces}. The rest can be easily computed as a consequence of these first results, see \cite{EllRan21,Demlow2009} and references therein.
 \end{proof}
 

Next, let  $\muh(x)$ be the Jacobian of the transformation $\pi(x)|_{\Gah}$ : $\Gah \to \Ga$ such that $\muh(x)\dsh = \ds$, then the following estimate holds (see \cite{Demlow2009})
\begin{equation}
    \begin{aligned}\label{eq: muhkg estimate UNS}
        \norm{1-\muh}_{L^{\infty}(\Gah)} \leq ch^{k_g+1},
    \end{aligned}
\end{equation}
from which we get the uniform bounds $\norm{\muh}_{L^{\infty}(\Gah)} \leq c$ and $\norm{\muh^{-1}}_{L^{\infty}(\Gah)} \leq c$.

\subsubsection{Finite Element Function Spaces}\label{FEMspaces UNS}
Since certain quantities on the discrete surface $\Gah$, e.g. projections and their derivatives, are only defined elementwise on each surface finite element $T$, it is convenient to introduce broken surface Sobolev spaces, with their respected norm \cite{EllRan21}. The norm $\norm{\cdot}_{H^m(\Th)}$ related to the \emph{broken Sobolev space} $\bfH^m(\mathcal{T}_h)=(H^m(\mathcal{T}_h))^3$ is defined by
\begin{equation}
    \begin{aligned}
        \norm{\wh}_{H^m(\Th)}^2 = \sum_{T\in\Th}\norm{\wh}^2_{H^m(T)},
    \end{aligned}
\end{equation}
for a vector field $\wh \in \bfH^m(\Gah)$. The same notation holds in the case of scalar quantities.

Given a triangulation $\Th$ of $\Gah$, as described above, we define the $H^1-$conforming Lagrange finite element space on $\Gah$, with finite elements of degree $k\geq 0$ :
\begin{equation}\label{eq: discrete fem space UNS}
    S_{h,k_g}^{k} := \{v_h \in C^0(\Gah) : v_h|_{T} = \hat{v}_h\circ \hat{F}_{T}^{-1} \text{ for some } \hat{v}_h \in \mathbb{P}^{k}(\hat{T}), \ \text{ for all } T \in \Th \},
\end{equation}
where $\mathbb{P}^{k}$ is the space of piecewise polynomials on the reference element of degree $k$ and $\hat{F}_T$ is some reference map to some reference element; see \cite{EllRan21,elliott2024sfem} for more details. Due to the mapping $\hat{F}_T^{-1}$ we observe that the space does not necessarily consists of polynomials over each triangle $T \in \Th$, unless we consider affine finite elements \cite[Example 6.7]{EllRan21}, consisting of piecewise linear triangulation and first-order polynomial functions. Note as well that $S_{h,k_g}^{k} \subset C^0(\Gah)\cap H^1(\Gah)$.

All the above can be easily expanded for vector valued functions where now the finite element space takes the form:
\begin{equation*}
    (S_{h,k_g}^{k})^3 = S_{h,k_g}^{k}\times S_{h,k_g}^{k} \times S_{h,k_g}^{k} \subset \bfH^1(\Gah).
\end{equation*}

We denote the velocity approximation order by $k_u$, and for the pressures $p$ and $\lambda$, we represent their approximation orders as $k_{pr}$ and $k_{\lambda}$, respectively. Writing 
$$\mathrm{\mathbf{P}}_{k_u} := (S_{h,k_g}^{k_u})^3, \ \mathrm{P}_{k_{pr}}:= S_{h,k_g}^{k_{pr}}, 
\ \mathrm{P}_{k_{\lambda}} :=S_{h,k_g}^{k_\lambda}$$ 
we  use the notation $\mathrm{\mathbf{P}}_{k_u}$-- $\mathrm{P}_{k_{pr}}$-- $\mathrm{P}_{k_{\lambda}}$ to denote  the generalized \emph{Taylor-Hood} surface finite elements, where $k_{pr}= k_u-1$, $k_{\lambda} \geq 1$.

\subsubsection{Lifting to the exact surface}\label{sec: surface lifting UNS}
We want to relate quantities on the discrete surface $\Gah$ and $\Ga$, since our finite element solutions live on $\Gah$ and the continuous solutions on $\Ga$. For that we use an extension/lifting procedure; see also \cite{DziukElliott_acta,Demlow2009,EllRan21}.

We already assumed that $h$ is small enough such that $\Gah \subset U_\delta$, hence the projection $\pi(\cdot)$ is well-defined and a bijection. Therefore, for every triangle in the triangulation $T \in \Th$ there exists an induced curvilinear triangle $T^\ell = \pi(T) \subset \Ga$ by the triangulation $\Th$ of $\Gah$. We may then define the conforming quasi-uniform triangulation, see \cite{EllRan21}, $\Th^\ell$ of $\Ga$ as
\begin{equation*}
    \Ga = \cup_{T^\ell \in \Th^\ell}T^\ell.
\end{equation*}

We define the \emph{lift} operator on $\Gah$ of any finite element function $v_h: \Gah \to \mathbb{R}^m$ with the help of the closest point projection operator by
\begin{equation}\label{eq: lift UNS}
    v_h^\ell(\pi(x)) := v_h(x), \quad \text{for } x\in\Gah.
\end{equation}
Similar to the extension \eqref{eq: smooth extension UNS}, we also define the inverse lift operator for a function $v :\Ga \to \mathbb{R}^m$ by,
\begin{equation}\label{eq: inverse lift UNS}
    v^{-\ell}(x) := v(\pi(x)), \quad \text{for } x\in\Gah.
\end{equation}
Notice that the two extensions \eqref{eq: smooth extension UNS} and \eqref{eq: inverse lift UNS} agree on the discrete surface $\Gah$ i.e. $v^{-\ell} = v^e|_{\Gah}$ since we assumed $\Gah \subset U_\delta$. 
Lastly, we also introduce the space of lifted finite element functions given by
\begin{equation*}
    (S_{h,k_g}^{k})^\ell = \{u_h^\ell \in H^1(\Ga) : u_h \in S_{h,k_g}^{k} \} \subset H^1(\Ga).
\end{equation*}

We now compare functions defined on the discrete surface $\Gah$ and the exact surface $\Ga$. We relate their tangential derivatives with the help of the lift extension defined above. Set $\Bhg = \bfPh(\bfI - d\bfH)\bfPg$.
\begin{description}
    \item[\underline{Scalar Functions} :] Using the extension of a scalar function  $v_h : \Gah \to \mathbb{R}$, \eqref{eq: lift UNS} and the chain rule we have that the discrete tangential gradient of $v_h \in H^1(\Gah)$ for $x \in \Gah$ is
\begin{equation}
    \begin{aligned}
        \nbgh v_h(x) =  \nbgh v_h^{\ell}(\pi(x)) = \bfPh(\bfI - d\bfH)\bfPg\nbg v_h^\ell(\pi(x)) = \Bhg \nbg v_h^\ell(\pi(x)),
    \end{aligned}
\end{equation}
We also notice that $\Bhg : T_x(\Gah) \mapsto T_{\pi(x)}(\Ga)$ is invertible; see \cref{Lemma: Bh estimates UNS}, i.e. we have that the map $\Bhg^{-1}$ on the tangent space $T_{\pi(x)}( \Ga)$ is given by $$\Bhg^{-1}|_{T_x \Ga} = \bfPg(\bfI - d\bfH)^{-1}(\bfI - \frac{\nh\otimes\bfn}{\nh \cdot \bfn})\bfPh,$$
thus we have for $v_h : \Gah \to \mathbb{R}$ and for $x\in \Gah$ that (see \cite{demlow2007adaptive}),
\begin{equation}
    \begin{aligned}
        \nbg v_h^{\ell}(\pi(x)) = \bfPg(\bfI - d\bfH)^{-1}(\bfI - \frac{\nh\otimes\bfn}{\nh \cdot \bfn})\nbgh v_h(x) = \Bhg^{-1}\nbgh v_h(x).
    \end{aligned}
\end{equation}
Noticing that the matrix $\bfG = (\bfI - \frac{\nh\otimes\bfn}{\nh \cdot \bfn}) = (\bfI - \frac{\nh\otimes\bfn}{\nh \cdot \bfn}) \bfPh$ we can easily calculate that  $\Bhg^{-1}\Bhg = \bfPg$ and $\Bhg\Bhg^{-1} = \bfPh$.
\item[\underline{Vector Valued Functions} :] Considering the tangential gradients of vector valued function \eqref{eq: tangential derivative UNS} and \eqref{eq: tangential derivative matrix UNS}, we may apply the lift extension component-wise, for each row, and with similar calculations to the scalar case we may get, after factoring out the common map $\Bhg$, that
\begin{equation}
    \begin{aligned}
        \nbgh \vh(x) = \nbgh \vh^\ell(\pi(x)) 
        = \nbg \vh^\ell(\pi(x))\Bhg^t.
    \end{aligned}
\end{equation}
Likewise for the inverse transformation we get
\begin{equation}
    \begin{aligned}
        \nbg \vh^{\ell}(\pi(x)) = \nbgh \vh(x)(\Bhg^{-1})^t.
    \end{aligned}
\end{equation}
\end{description}

Recalling the results from \cite{Heine2004,DziukElliott_acta,hansbo2020analysis,elliott2024sfem} we have the following lemma, which also shows that $\bfB_h$ is invertible.
\begin{lemma}[Estimates for $\bfB_h$]\label{Lemma: Bh estimates UNS}
For sufficiently small $h$ we have the following bounds
\begin{align}
    \label{eq: Bh stability UNS}
     \norm{\Bhg}_{L^{\infty}(\Gah)} &\leq c, \qquad\qquad\qquad\qquad \norm{\Bhg^{-1}}_{L^{\infty}(\Ga)} \leq 1, \\
     \label{eq: Bh estimates UNS}
     \norm{\bfPg - \Bhg}_{L^{\infty}(\Gah)} &\leq ch^{k_g} ,  \qquad\qquad\norm{\bfPg - \Bhg^t\Bhg}_{L^{\infty}(\Gah)} \leq ch^{k_g+1}.
\end{align}
\end{lemma}

\subsubsection{Norm equivalence}
We use $\sim$ to denote the norm equivalence independent of $h$  and by the estimates in \cref{lemma: Geometric errors UNS} we recall the following Lemma, involving equivalence of norms on different surfaces.
\begin{lemma}[Norm Equivalence]
    \label{lemma: norm equivalence UNS}
    Let $v_h \in H^j(T),$ $j\geq 2$ for $T\in \Th$ a face of $\Gah$ and let 
    $T^{\ell} \in \mathcal{T}_h^{\ell}$ be a face of the lifted surface $\Ga$. Then for $h$ small enough we have that following equivalences
    \begin{equation}\label{eq: norm equivalence UNS}
        \begin{aligned}
            \norm{v_h^{\ell}}_{L^2(T^{\ell})}&\sim \norm{v_h}_{L^2(T)}\sim \norm{\tilde{v}_h}_{L^2(\Tilde{T})}, \\
            \norm{\nbg v_h^{\ell}}_{L^2(T^{\ell})}&\sim \norm{\nbgh v_h}_{L^2(T)}\sim \norm{\nbglin\tilde{v}_h}_{L^2(\Tilde{T})},\\
            |\nbgh^j v_h|_{L^2(T)} \leq c \sum_{k=1}^j & |\nbg v_h^{\ell}|_{L^2(T^\ell)},  \qquad 
             |\nbg^j v_h^{\ell}|_{L^2(T^{\ell})} \leq c \sum_{k=1}^j |\nbgh v_h|_{L^2(T^\ell)}.
        \end{aligned}
    \end{equation}
\end{lemma}
\begin{proof}
    The first two set of equivalences can be found in several literature, e.g. \cite{DziukElliott_acta,Demlow2009}. Regarding the last two inequalities, see \cite[Lemma 3.3]{LarssonLarsonContDiscont}.
\end{proof}

\subsection{Geometric perturbations}\label{sec: Geometric perturbations UNS}

Analogously to the continuous case of Definition \ref{def: Covariant derivatives UNS}  we define similarly for the discrete case the following operators for $\vh \in \bfH^1(\Gah)$,
\begin{equation}
    \begin{aligned} \label{eq: Eh forms UNS}
        \nbgcovh \vh &:= \bfPh \nbgh \vh,  \\ 
        E_h(\vh) &:= \half (\nbgcovh \vh + \nb_{\Gah}^{cov,t} \vh). 
     \end{aligned}
\end{equation}

\begin{lemma}\label{lemma: errors of geometric pert UNS}
For $h$ small enough, $\wh,\vh  \in \bfH^1(\Gah)$, $\bff \in (L^2(\Ga))^3$,  we have the following perturbation bounds
\begin{align}\label{eq: errors of domain of integration data UNS}
        |(\bff,\vhl)_{L^2(\Ga)}- (\bff^{-\ell},\vh)_{L^2(\Gah)}| &\leq  ch^{k_g+1}\norm{\bff}_{L^2(\Ga)}\norm{\vh}_{L^2(\Gah)},\\[5pt]
         \label{eq: Geometric perturbations nbg UNS}
        \norm{(\nbg\whl)^{-\ell} - \nbgh\wh}_{L^2(\Gah)} &\leq ch^{k_g}\norm{\wh}_{H^1(\Gah)},\\[5pt]
        \label{eq: Geometric perturbations a 1 UNS}
        |(E(\bfw_h^{\ell}),E(\vhl))_{L^2(\Ga)} - (E_h(\bfw_h),E_h(\vh))_{L^2(\Gah)} | &\leq ch^{k_g} \norm{\bfw_h}_{H^1(\Gah)}\norm{\vh}_{H^1(\Gah)},\\[5pt]
        \norm{E(\bfw_h^{\ell})^{-\ell} - E_h(\bfw_h)}_{L^2(\Gah)} &\leq ch^{k_g}\norm{\bfw_h}_{H^1(\Gah)}. \label{eq: Geometric perturbations a 2 UNS}
    \end{align}
\end{lemma}
\begin{proof}
Refer to \cite{highorderESFEM,DziukElliott_acta} for the initial two estimates, which can be easily generalized to vector-valued functions, and \cite[Lemma 4.6]{elliott2024sfem} for the remaining two.
\end{proof}

\subsection{Interpolation}
\subsubsection{Lagrange Interpolation}\label{sec: Lagrangian interpolant UNS}
We first introduce the standard Lagrange-type interpolant $\Ih : (C^0(\Gah))^3 \to (S^{k}_{h,k_g})^3$, which approximates sufficiently smooth \emph{vector-valued} functions and its lift $\Ihl \bfv = (\Ih \bfv^{-\ell} )^{\ell}$ on the smooth surface $\Ga$.
\begin{proposition}\label{prop: lagrange interpolant UNS}
  Given $\bfv \in \bfH^{k+1}(\Ga)$ the Lagrangian interpolant satisfies the following bounds 
  \begin{equation}\label{eq: lagrange interpolant UNS}
      \norm{\bfv^{-\ell} - \Ih \bfv^{-\ell}}_{L^2(\Gah)} + h\norm{\bfv^{-\ell} - \Ih \bfv^{-\ell}}_{H^1(\Gah)} \leq ch^{k+1}\norm{\bfv}_{H^{k+1}(\Ga)}.
  \end{equation}
for some integer $k$, with lift $\Ihl(\bfv) = \Ih(\bfv^{-\ell})^{\ell}$.
Furthermore, if $\bfv \in (W^{k+1}_{\infty}(\Ga))^3$ then it also holds that 
\begin{equation}\label{eq: lagrange interpolant infty UNS}
    \norm{\bfv - \Ihl \bfv}_{L^{\infty}(\Ga)}+ h\norm{\nbg(\bfv - \Ihl \bfv)}_{L^{\infty}(\Ga)} \leq ch^{k+1} \norm{\bfv}_{ W^{k+1}_{\infty}(\Ga)}.
\end{equation}
\end{proposition}
\begin{proof}
These bounds are just extensions of the scalar bounds. The first estimate \eqref{eq: lagrange interpolant UNS} can be derived from standard Lagrange interpolation theory \cite{DziukElliott_acta,Demlow2009}. The last estimate \eqref{eq: lagrange interpolant infty UNS}
can be obtained as in \cite[Lemma 6.24]{EllRan21} or \cite[Theorem 4.5.11]{BrennerScott}.
\end{proof}
\subsubsection{Scott-Zhang Interpolation}\label{sec: SZ interpolant UNS}
For the stability analysis we will need to make use of the Scott-Zhang interpolant, and specifically its projection and super-approximation properties. 

Let us define the Scott-Zhang interpolant $\Ihz : (L^2(\Gah))^n \to (S_{h,k_g}^{k})^n$ for $n=1 ,3$. Then, for each element $T \in \Th$ and $\bfv \in (L^2(\Ga))^n$ the following estimate holds
\begin{equation}
    \begin{aligned}\label{eq: bounds of Scott-Zhang interpolant UNS}
        \norm{\bfv^{-\ell} - \Ihz\bfv^{-\ell}}_{H^l(T)} \leq ch^{l-k+1}\norm{\bfv^{-\ell}}_{H^{k+1}(\omega_{\Ttilde})}
    \end{aligned}
\end{equation}
where $\omega_{\Ttilde}$ is a union of neighbouring elements of the element $T$. The following stability bound also holds
\begin{equation}\label{eq: stability of Scott-Zhang interpolant UNS}
    \norm{\Ihz \bfv^{-\ell}}_{H^1(T)} \leq c\norm{\bfv^{-\ell}}_{H^1(\omega_{\Ttilde})} \leq c\norm{\bfv}_{H^1(\omega_{\Ttilde}^\ell)}.
\end{equation}
We omit further details, but if one requires more information about the construction of the Scott-Zhang interpolant on high-order surfaces, we refer to \cite{hansbo2020analysis,elliott2024sfem}.

Finally, we present a super-approximation type of estimate and stability for the Scott-Zhang interpolant. This type of estimate has been proven in \cite{hansbo2020analysis} due to the projection property of $\Ihz$ onto the finite element space $(S_{h,k_g}^{k})^3$.
We present the lemma for completeness.
\begin{lemma}
    For discrete functions $\vh\in (S_{h,k_g}^{k})^3$ and $\chi \in [W^{k+1,\infty}(\Ga)]^3$  it holds
    \begin{equation}
        \begin{aligned}
        \label{eq: super-approximation estimate 1 UNS}
            \sum_{T\in \Th}\norm{\nbgh(I - \Ihz)(\chi^{-\ell} \cdot \vh)}_{L^2(T)} \leq c h\norm{\chi}_{W^{k+1,\infty}(\Ga)} \norm{\vh}_{H^1(\Gah)},
        \end{aligned}
    \end{equation}
     \begin{equation}
        \begin{aligned}
        \label{eq: super-approximation estimate 2 UNS}
            \sum_{T\in \Th}\norm{(I - \Ihz)(\chi^{-\ell} \cdot \vh)}_{L^2(T)} \leq ch \norm{\chi}_{W^{k+1,\infty}(\Ga)} \norm{\vh}_{L^2(\Gah)}.
        \end{aligned}
    \end{equation}
We also have the $L^2$ stability estimate
    \begin{equation}
        \begin{aligned}
        \label{eq: super-approximation stability UNS}
            \norm{\Ihz(\chi^{-\ell} \cdot \vh)}_{L^2(\Gah)} \leq c\norm{\chi^{-\ell} \cdot \vh}_{L^2(\Gah)}.
        \end{aligned}
    \end{equation}
\end{lemma}

\section{The fully Discrete scheme}\label{sec: fully discrete method UNS}

As in the Stokes case \cite{elliott2024sfem}, we consider the \emph{Taylor-Hood} surface finite elements $\mathrm{\mathbf{P}}_{k_u}$-- $\mathrm{P}_{k_{pr}}$-- $\mathrm{P}_{k_{\lambda}}$
 for a $k_g$-order approximation of the surface, (c.f.  Section \ref{subsection: Space discretization UNS}), and  set $$k_u\geq 2,~k_{pr}=k_u-1~~ \mbox{and} ~~k_{\lambda}\geq 1.$$ We maintain the notation  $k_g,k_u,k_{pr}$ and $k_\lambda$ in order to track how   error dependencies arise and define the following finite element spaces
\begin{equation}\label{eq: finite element spaces approximation UNS}
     V_h = S_{h,k_g}^{k_u}, ~~ \bfV_h = (S_{h,k_g}^{k_u})^3, ~~
        Q_h = S_{h,k_g}^{k_{pr}} \cap L^2_0(\Gah), ~~
        \Lambda_h = S_{h,k_g}^{k_{\lambda}}.
   \end{equation}
   
In particular, the eventual choice of $k_{\lambda}$ plays an important role in the stability and convergence of the two discrete pressures $\ph,\lh$. We will see, in the error analysis \cref{sec: error analysis UNS} that for $k_\lambda = k_u$ we obtain optimal convergence rates using \emph{iso-parametric surface finite elements}, while choosing $k_\lambda = k_u-1$ we lose an order of convergence with respect to the geometric approximation $k_g$ of the surface, and therefore to obtain the same convergence one has to use \emph{super-parametric surface finite elements}. This choice, as we notice in the numerical results section \Cref{sec: Numerical results}, is also important for the underlined conditioning of the system.


We define the space of \emph{discrete weakly tangential divergence-free functions}
\begin{equation}\label{eq: discrete weakly divfree space UNS}
    \bfV_h^{div} := \{\wh \in \bfV_h : \bhtil(\wh,\{\qh,\xi_h\}) =0, ~~\forall \, \{\qh,\xi_h\} \in Q_h\times\Lambda_h\}.
\end{equation}

Finally, we will need the dual space of the finite element space $\Lambda_h$, which we denote as 
$H_h^{-1}(\Gah)$. For $\ell_h \in \Lambda_h$  we define its dual norm as 
\begin{equation}\label{eq: H^-1h definition UNS}
    \norm{\ell_h}_{H_h^{-1}(\Gah)} = \sup_{\xi_h\in \Lambda_h\backslash\{0\}}\frac{<\ell_h,\xi_h>_{H^{-1},H^1}}{\norm{\xi_h}_{H^1(\Gah)}} = \sup_{\xi_h\in \Lambda_h\backslash\{0\}}\frac{(\ell_h,\xi_h)_{L^2(\Gah)}}{\norm{\xi_h}_{H^1(\Gah)}}.
\end{equation}

The layout of the rest of the section  is as follows. We define our bilinear forms in \Cref{sec: The discrete bilinear forms UNS} and prove some important properties in  \Cref{sec: Properties of the discrete bilinear forms UNS}. In \Cref{sec: Perturbation bounds and interpolation bounds in energy norm UNS} we recall some perturbation bounds and \emph{Scoot-Zhang} surface interpolation estimates. In \Cref{sec: Ritz-Stokes projection UNS,sec: Leray projection UNS} we define new surface Ritz-Stokes and Leray projections. Finally, in \Cref{sec: var stability UNS}, we define our fully discrete scheme.

\subsection{The discrete bilinear forms}\label{sec: The discrete bilinear forms UNS}
The finite element method relies on the  following discrete bilinear forms and inertia terms: For  $\bfz_h,\wh,\vh \in \bfV_h \subset \bfH^1(\Gah)$, $\{\qh,\xi_h\} \in Q_h\times \Lambda_h$ we define
\begin{align}
   \label{eq: lagrange discrete bilinear forms UNS}
    \ah(\wh,\vh) &:= \int_{\Gah} E_h(\wh):E_h(\vh) \dsh + \int_{\Gah} \wh \cdot \vh \dsh, \\
      \label{eq: skew-symmetrized ch real Ph}
        \chtil(\bfz_h; \wh,\vh) &:=\frac{1}{2} \Big(\int_{\Gah } ((\bfz_h \cdot \nbgcovh)\wh ) \cdot\vh \; \dsh - \int_{\Gah } ((\bfz_h \cdot \nbgcovh)\vh ) \cdot \wh \; \dsh \Big) \nonumber\\
      &\quad - \frac{1}{2}\Big(\int_{\Gah} (\wh\cdot \nh) \zh \cdot \bfH_h \vh \; \dsh - \int_{\Gah} (\vh\cdot \nh) \zh \cdot \bfH_h \wh \; \dsh \Big),\\
      \bhtil(\wh,\{\qh,\xi_h\}) &:= \bh(\wh,\qh) + \int_{\Gah} \xi_h \wh\cdot \nh \, \dsh = \int_{\Gah} \wh \cdot \nbgh \qh  \, \dsh + \int_{\Gah} \xi_h \wh\cdot \nh \, \dsh.\label{eq: lagrange discrete bilinear forms end UNS} 
    \end{align}
\begin{remark}\label{remark: inertia term Ph}
The main reason for using  \eqref{eq: skew-symmetrized ch real Ph} to approximate  the trilinear inertia term is to provide  a skew-symmetrized discrete analogue to \eqref{eq: continuous c error formula init UNS} since we use $H^1$-conforming Taylor-Hood elements; as seen in the literature e.g. \cite{Temam1979NavierStokesET}. In the discrete setting, recalling \eqref{eq: split cov UNS}, one can replace $\nbgcovh\bfPh\wh$ by $ \nbgcovh\wh - (\wh\cdot\nh)\bfH_h$, giving rise to \eqref{eq: skew-symmetrized ch real Ph}.
\end{remark}


\noindent In our stability analysis we use the energy norm 
\begin{equation}\label{eq: energy norm UNS}
\norm{\wh}_{\ah} = \norm{E_h(\wh)}_{L^2(\Gah)} + \norm{\wh}_{L^2(\Gah)},
\end{equation}
for a velocity finite element vector $\wh \in \bfH^1(\Gah)$. We also readily see that
\begin{equation}\label{eq: energy norm H1 bound UNS}
    \norm{\wh}_{\ah} \leq \norm{\wh}_{H^{1}(\Gah)}.
\end{equation}
\subsection{Properties of the discrete bilinear forms}\label{sec: Properties of the discrete bilinear forms UNS}
We present some known uniform interpolation and Korn-type inequalities that can be found in \cite{elliott2024sfem,jankuhn2021trace}. We also extend them using Broken Sobolev spaces.

\begin{lemma}[Interpolation and Korn-type inequalities]\label{lemma: Korn's type inequalities Lagrange UNS}
The following important uniform inequalities hold
\begin{align}
        \label{discrete Korn's inequality T nh UNS}
        \norm{\wh}_{H^1(\Gah)} &\leq c(\norm{E_h(\wh)}_{L^2(\Gah)}  +\norm{\wh}_{L^2(\Gah)} + h^{-1}\norm{\wh \cdot \nh}_{L^2(\Gah)})  \ \ \textnormal{ for all } \wh \in \bfV_h,
        \\
        \label{eq: coercivity and Korn's inequality Lagrange UNS}
        \norm{\wh}_{H^1(\Gah)}^2 &\leq c h^{-2} \norm{\wh}_{\ah}^2  \qquad \qquad  \qquad \ \textnormal{ for all } \wh \in \bfV_h,\\
        \label{eq: interpolation inequality Gaglia UNS}
        \norm{\wh}_{L^4(\Gah)} &\leq c\norm{\wh}_{L^2(\Gah)}^{1/2}\norm{\wh}_{H^1(\Gah)}^{1/2} \quad  \quad  \textnormal{ for all } \wh \in \bfH^1(\Gah).
    \end{align}
\end{lemma}
\begin{proof}
   Estimates \eqref{discrete Korn's inequality T nh UNS}, \eqref{eq: coercivity and Korn's inequality Lagrange UNS} follow from \cite[Lemma 4.7]{elliott2024sfem}, while \eqref{eq: interpolation inequality Gaglia UNS} holds due to \cite[Lemma 4.2]{olshanskii2023eulerian}.    
\end{proof}

\noindent Now we prove a special case regarding the discrete Korn-type inequality, with the help of the broken Sobolev spaces. The arguments follow similar to \cite{Hardering2022}.
\begin{lemma}\label{lemma: Korn's inequality Ph UNS}
    For $\wh \in \bfV_h$ we have
    \begin{align}
            \norm{\bfPh\wh}_{H^1(\Th)} &\leq c \, \ah(\wh,\wh), \label{eq: korn inequality Ph UNS}\\
             \norm{\bfPh\wh}_{L^4(\Th)} &\leq c\norm{\wh}^{1/2}_{L^2(\Gah)}\norm{\wh}_{\ah}^{1/2} \leq c\norm{\wh}_{\ah} \label{eq: L4 inequality Ph UNS}.
    \end{align}
\end{lemma}
\begin{proof}
For the first estimate we consider two cases. First for $k_g = 1$ we have that $\nh$, and therefore $\bfPh$, is constant on each triangle giving rise to the expression $\nbgh(\bfPh\wh) = \bfPh \nbgh \wh = \nbgcovh \wh$, which leads to our desired result,
\begin{equation*}
        \norm{\bfPh\wh}_{H^1(\Th)} \leq  \norm{\nbgcovh \wh}_{L^2(\Gah)} + \norm{\wh}_{L^2(\Gah)} \leq  \norm{\wh}_{\ah}.
\end{equation*}
Now for $k_g \geq 2$ we split our expression as followed
\begin{equation*}
    \begin{aligned}
        \norm{\bfPh\wh}_{H^1(\Th)} \leq \norm{(\bfPh-\bfP)\wh}_{H^1(\Th)} + \norm{\bfP\wh}_{H^1(\Th)}.
    \end{aligned}
\end{equation*}
For the first term we use \cite[Lemma 4.4]{Hardering2022} to get
\begin{equation*}
    \norm{(\bfPh-\bfP)\wh}_{H^1(\Th)} \leq h^{k_g-1}\norm{\wh}_{L^2(\Gah)},
\end{equation*}
while the second term following \cite[Lemma 4.7]{elliott2024sfem}, applying  Korn's inequality \eqref{eq: Korn inequality UNS}, and \Cref{lemma: errors of geometric pert UNS} gives 
\begin{equation*}
    \begin{aligned}
        \norm{\bfP\wh}_{H^1(\Th)} &\leq c \norm{\bfPg\whl}_{H^1(\Ga)} \leq c\norm{E(\bfPg\whl)}_{L^2(\Ga)} + c\norm{\bfPg \whl}_{L^2(\Ga)}\\ 
    &\leq c\big(\norm{E(\whl)}_{L^2(\Ga)} +  \norm{(\whl \cdot \bfng)\bfH}_{L^2(\Ga)}  +\norm{\bfPg \whl}_{L^2(\Gah)}\big) \\ 
     &\leq c\norm{E_h(\wh)}_{L^2(\Gah)} + c\norm{(E(\whl))^{-\ell} - E_h(\wh)}_{L^2(\Gah)} + c\norm{\wh}_{L^2(\Gah)}\\
    &\leq c\norm{E_h(\vh)}_{L^2(\Gah)} + c(1 + h^{k_g-1})\norm{\wh}_{L^2(\Gah)} \qquad\qquad\qquad\qquad \text{(by inv. ineq.)}.
    \end{aligned}
\end{equation*}
Combining these results \eqref{eq: korn inequality Ph UNS} follows. As for \eqref{eq: L4 inequality Ph UNS} applying \eqref{eq: korn inequality Ph UNS} on \eqref{eq: interpolation inequality Gaglia UNS} yields the required estimate.
\end{proof}

\noindent Now, the following continuity and coercivity estimates for the bilinear and trilinear forms hold.
\begin{lemma}[Bounds and coercivity]
\label{Lemma: discrete bounds and coercivity results}
    The following bounds and coercivity estimates hold:
    \begin{align}
        \label{ah boundedness UNS}
        \ah(\wh,\vh) &\leq \norm{\wh}_{\ah}\norm{\vh}_{\ah},\\
          \label{bhtilde boundedness UNS}
       \bhtil(\wh,\{\qh,\xi_h\}) &\leq c \norm{\wh}_{\ah}\norm{\qh}_{L^2(\Gah)} + \norm{\wh}_{L^2(\Gah)}\norm{\xi_h}_{L^2(\Gah)}\leq \norm{\wh}_{\ah}\norm{\{\qh,\xi_h\}}_{L^2(\Gah)}, \\
       \label{ch boundedness UNS}
        \chtil(\zh; \wh,\vh) &\leq c\norm{\zh}_{L^2(\Gah)}^{1/2}\norm{\zh}^{1/2}_{\ah}\norm{\wh}_{\ah}\norm{\vh}^{1/2}_{L^2(\Gah)}\norm{\vh}^{1/2}_{\ah} \\
        &\quad + c\norm{\zh}_{L^2(\Gah)}^{1/2}\norm{\zh}^{1/2}_{\ah}\norm{\vh}_{\ah}  \norm{\wh}_{L^2(\Gah)}^{1/2}\norm{\wh}^{1/2}_{\ah}, \nonumber
    \end{align}
$ \textnormal{ for all } \zh,\wh,\vh \in \bfV_h$, and $(\qh,\xi_h) \in Q_h\times \Lambda_h$.
\end{lemma}

\begin{proof}
The first estimate can be readily seen by a simple use of Schwarz's inequality. To see \eqref{bhtilde boundedness UNS} we calculate with the help of the integration by parts formula \eqref{eq: integration by parts UNS} that,
\begin{equation*}
    \begin{aligned}
        \int_{\Gah} \wh \cdot \nbgh \qh \ \dsh &\leq c \norm{\qh}_{L^2(\Gah)}\norm{\wh}_{\ah} + ch^{k_g-1}\norm{\qh}_{L^2(\Gah)}\norm{\wh}_{L^2(\Gah)}\\
       & \leq c \norm{\qh}_{L^2(\Gah)}\norm{\wh}_{\ah},
    \end{aligned}
\end{equation*}
where in the first inequality we used the $|[\mh]| \leq ch^{k_g}$ bound for the co-normal \eqref{eq: geometric errors 1 UNS} (cf. \cite{EllRan21}) and the trace inequalities $\norm{\qh}_{L^2(\mathcal{E}_h)}\leq ch^{-1/2}\norm{\qh}_{L^2(\Gah)}$, and $\norm{\wh}_{L^2(\mathcal{E}_h)}\leq ch^{-1/2}\norm{\wh}_{L^2(\Gah)}$.  Regarding the last inequality, from \eqref{eq: skew-symmetrized ch real Ph} and simple Cauchy-Schwarz bounds, we readily see that 
\begin{equation*}
    \begin{aligned}
       \chtil(\zh; \wh,\vh) &\leq c\norm{\bfPh\zh}_{L^4(\Th)}\norm{\wh}_{\ah}\norm{\bfPh\vh}_{L^4(\Th)}  + c\norm{\bfPh\zh}_{L^4(\Th)}\norm{\vh}_{\ah}\norm{\bfPh\wh}_{L^4(\Th)}.
    \end{aligned}
\end{equation*}
The final result is readily seen with the help of the previously proven results in \cref{lemma: Korn's inequality Ph UNS}.
\end{proof}

Now, let us prove some important properties that involve \emph{discrete weakly tangential divergence-free functions} $\wh \in \bfV_h^{div}$ \eqref{eq: discrete weakly divfree space UNS}. These properties will be used throughout the rest of text. 

\begin{lemma}\label{lemma: divfree L2 inner normal UNS}
Let $\wh \in \bfV_h^{div}$ and $\vh \in \bfV_h$, then we have the following inequalities
\begin{align}
\label{eq: divfree L2 inner normal kl=ku UNS}
    (\wh\cdot\nh,\vh\cdot\nh)_{L^2(\Gah)} &\leq ch\norm{\wh\cdot\nh}_{L^2(\Gah)}\norm{\vh}_{L^2(\Gah)}, \  \text{ for } k_\lambda = k_u.\\[5pt]
    \label{eq: divfree L2 inner normal kl=ku-1 UNS}
    (\wh\cdot\nh,\vh\cdot\nh)_{L^2(\Gah)} &\leq ch\norm{\wh\cdot\nh}_{L^2(\Gah)}\norm{\vh}_{H^1(\Gah)}, \  \text{ for } k_\lambda = k_u-1.
\end{align}
\end{lemma}
\begin{proof}
In the case where $\underline{k_\lambda = k_u}$, i.e. when the finite element spaces $\Lambda_h=S^{k_{\lambda}}_{h,k_g} = S^{k_u}_{h,k_g} = V_h$, for $\wh \in \bfV_h^{div}$, cf. \eqref{eq: discrete weakly divfree space UNS}, we have that  $\int_{\Gah} \xi_h(\bfw_h\cdot\nh) =0$ for any $\xi_h \in \Lambda_h=V_h$.  So, using  the estimate $\norm{\bfn-\nh}_{L^{\infty}} \leq ch^{k_g}$ \eqref{eq: geometric errors 1 UNS}, the following computations follow
        \begin{align}\label{eq: normal inner product estimate inside 1 UNS}
            (\wh\cdot\nh,\vh\cdot\nh)_{L^2(\Gah)} &= (\wh\cdot\nh,\vh\cdot\bfn)_{L^2(\Gah)} + (\wh\cdot\nh,\vh\cdot(\nh - \bfn))_{L^2(\Gah)}\nonumber\\
            &\leq \underbrace{(\wh\cdot\nh,\vh\cdot\bfn - \Ihz(\vh\cdot\bfn))_{L^2(\Gah)}}_{(\wh\cdot\nh,\Ihz(\vh\cdot\bfn))_{L^2(\Gah)} =0, \ \Ihz(\cdot)\in V_h = \Lambda_h} + ch^{k_g} \norm{\wh\cdot\nh}_{L^2(\Gah)} \norm{\vh}_{L^2(\Gah)}\nonumber\\
            &\leq ch\norm{\wh\cdot\nh}_{L^2(\Gah)}\norm{\vh}_{L^2(\Gah)},
        \end{align}
where in the last inequality we used the super-approximation estimate for the  Scott-Zhang interpolant \eqref{eq: super-approximation estimate 2 UNS}, which holds since $\vh \in \bfV_h$ and $\Ihz(\cdot) : L^2(\Gah) \to V_h$.

On the other hand, in the case where $\underline{k_\lambda = k_u-1}$ we see that $\Lambda_h \neq V_h$, i.e. now $\Lambda_h$ is not rich enough, and therefore $\int_{\Gah} \xi_h(\bfw_h\cdot\nh) =0$ only for $\xi_h \in \Lambda_h$. For that reason, in the above eq. \eqref{eq: normal inner product estimate inside 1 UNS} we, instead, restrict ourselves to Scott-Zhang interpolants projected to $\Lambda_h$, that is, $(\Ihz)^{\lambda}(\cdot) : L^2(\Gah) \to \Lambda_h$. That means the super-approximation estimate \eqref{eq: super-approximation estimate 2 UNS} does not hold (since $\vh \in \bfV_h$), and therefore we use the standard interpolation estimate \eqref{eq: bounds of Scott-Zhang interpolant UNS} to obtain
        \begin{align*}
            (\wh\cdot\nh,\vh\cdot\nh)_{L^2(\Gah)} 
            &\leq \underbrace{(\wh\cdot\nh,\vh\cdot\bfn - (\Ihz)^{\lambda}(\vh\cdot\bfn))_{L^2(\Gah)}}_{(\wh\cdot\nh,(\Ihz)^{\lambda}(\vh\cdot\bfn))_{L^2(\Gah)} =0, \ \Ihz(\cdot)\in \Lambda_h} + ch^{k_g} \norm{\wh\cdot\nh}_{L^2(\Gah)} \norm{\vh}_{L^2(\Gah)} \\
            &\leq ch\norm{\wh\cdot\nh}_{L^2(\Gah)}\norm{\vh}_{H^1(\Gah)}.
        \end{align*}
\end{proof}
 The following results are an immediate consequence of the above \cref{lemma: divfree L2 inner normal UNS}.
\begin{corollary}\label{corollary: divfree L2 norm UNS}
Let $\wh \in \bfV_h^{div}$, then if $ \underline{k_\lambda = k_u-1}$ we have the following 
\begin{align}
    \label{eq: divfree L2 norm kl=ku-1 UNS}
    \norm{\wh\cdot\nh}_{L^2(\Gah)} &\leq ch\norm{\wh}_{H^1(\Gah)}.
\end{align}
Furthermore, if $\underline{k_\lambda = k_u}$ we have, for $h\leq h_0$ with sufficiently small $h_0$
\begin{align}
    \label{eq: divfree L2 norm kl=ku UNS}
    \norm{\wh\cdot\nh}_{L^2(\Gah)} &\leq ch\norm{\wh}_{L^2(\Gah)}, \\
    \label{eq: divfree L2 norm kl=ku 2 UNS}
    \norm{\wh\cdot\nh}_{L^2(\Gah)} &\leq ch\norm{\bfP\wh}_{L^2(\Gah)} + ch^{k_g+1}\norm{\wh}_{L^2(\Gah)}.
\end{align}
\end{corollary}

\begin{lemma}[Improved $H^1$ coercivity bound]\label{lemma: improved h1-ah bound UNS}
     If $\underline{k_{\lambda} = k_u}$ and $\bfw_h \in \bfV_h^{div}$ then the following  holds 
     \begin{align}\label{eq: improved h1-ah bound UNS}
         \norm{\bfw_h}_{H^1(\Gah)}^2 &\leq c\norm{\wh}_{\ah}^2.
     \end{align}
\end{lemma}
\begin{proof}
This can be immediately proven by combining the discrete Korn's inequality \eqref{discrete Korn's inequality T nh UNS} and \eqref{eq: divfree L2 norm kl=ku UNS}.
\end{proof}

\noindent We also have the following integration by parts formula.
\begin{lemma}
 On the discrete surface $\Gah$ the following integration by parts formula holds:
 \begin{equation}
     \begin{aligned}
     \label{eq: integration by parts UNS}
         \int_{\Gah} \vh \cdot \nbgh \qh \ \dsh &= - \int_{\Gah} \qh \divgh \vh \ \dsh + \sum_{T\in \mathcal{T}_h} \int_T (\vh \cdot \nh)\qh \divgh \nh \ \dsh \\
         &+ \sum_{E \in \mathcal{E}_h} \int_E [\mh \cdot \vh] \qh \, d\ell,
     \end{aligned}
 \end{equation}
 for $\vh \in \bfH^1(\Gah)$ and $\qh \in H^1(\Gah)$, for any $k,l \geq \in \mathbb{N}$.
\end{lemma}

\subsection{Perturbation bounds and interpolation bounds in energy norm}\label{sec: Perturbation bounds and interpolation bounds in energy norm UNS}

Recalling \Cref{sec: Geometric perturbations UNS} we present geometric perturbations for the bilinear forms; see \cite{elliott2024sfem} for more details.
\begin{lemma}\label{lemma: Geometric perturbations b lagrange UNS}
Let $\bfw_h,\vh \in \bfH^1(\Gah)$ and $\{\qh,\xi_h \}\in H^1(\Gah)\times H^1(\Gah)$. Then we have
\begin{align}
    \label{eq: Geometric perturbations a UNS}
    |a(\bfw_h^{\ell},\vhl) - \ah(\bfw_h,\vh)| &\leq ch^{k_g} \norm{\bfw_h}_{H^1(\Gah)} \norm{\vh}_{H^1(\Gah)},\\
    \label{eq: Geometric perturbations btilde UNS}
     |\bhtil(\bfw_h,\{\qh,\xi_h\}) - b^L(\bfw_h^{\ell},\{\qhl,\xi_h^{\ell}\})| &\leq ch^{k_g} \norm{\bfw_h}_{L^2(\Gah)}(\norm{\qh}_{H^1(\Gah)} + \norm{\xi_h}_{L^2(\Gah)}) \nonumber\\
     &\leq ch^{k_g-1} \norm{\bfw_h}_{L^2(\Gah)}(\norm{\qh}_{L^2(\Gah)} + h\norm{\xi_h}_{L^2(\Gah)}).
\end{align}
If  $\bfw \in \bfH_T^1$, then we get the following higher order bound
\begin{equation}
    \begin{aligned}\label{eq: Geometric perturbations btilde tangent UNS}
        |\bhtil(\bfw^{-\ell},\{\qh,\xi_h\}) - b^L(\bfw,\{\qhl,\xi_h^{\ell}\})|
        \leq ch^{k_g} \norm{\bfw}_{L^2(\Ga)}\norm{\{\qh,\xi_h\}}_{L^2(\Gah)}.
    \end{aligned}
\end{equation}
Furthermore if  $\bfw,\bfv \in \bfH_T^1\cap \bfH^2(\Ga)$ then we have
\begin{align}\label{eq: Geometric perturbations btilde tangent extra regularity UNS}
    |\bhtil(\bfw^{-\ell},\{\qh,\xi_h\}) - b^L(\bfw,\{\qhl,\xi_h^{\ell}\})|&\leq ch^{k_g+1} \norm{\bfw}_{H^2(\Ga)}\norm{\{\qh,\xi_h\}}_{H^1(\Gah)}, \\
    \label{eq: Geometric perturbations a tangent extra regularity UNS}
    |a(\bfw,\bfv) - \ah(\bfw^{-\ell},\bfv^{-\ell})| &\leq ch^{k_g+1} \norm{\bfw}_{H^2(\Ga)} \norm{\bfv}_{H^2(\Gah)}.
\end{align}
\end{lemma}
\begin{proof}
The first estimate can be easily proved with the help of \eqref{eq: Geometric perturbations a 1 UNS}. For the rest \eqref{eq: Geometric perturbations btilde UNS} - \eqref{eq: Geometric perturbations btilde tangent extra regularity UNS}, see \cite[Lemma 6.1]{elliott2024sfem}. The last estimate can be proved as in \cite[Lemma 5.4]{hansbo2020analysis} with suitable alterations.
\end{proof}

\begin{lemma}[Weingarten map consistency bound]\label{lemma: weingarten map improved UNS}
The following bound holds for all $\vh \in \bfH^1(\Gah)$,
    \begin{equation}
        \begin{aligned}\label{eq: weingarten map improved UNS}
             \Big|\int_{\Gah} \beta^{-\ell}\,\vh\cdot\bfH_h\bfu^{-\ell} \, \dsh - \int_{\Ga} \beta\,\vhl\cdot\bfH\bfu \, \ds\Big| \leq ch^{k_g}\norm{\beta}_{H^1(\Ga)}\norm{\bfu}_{H^1(\Ga)}\norm{\vh}_{H^1(\Gah)}.
        \end{aligned}
    \end{equation}
\end{lemma}
\begin{proof}
This proof follows similarly to \cite[Appendix C]{olshanskii2023eulerian}. For reasons of clarity we present the proof. First, recall that $\bfH_h = \nbgcovh \nh  = \bfPh\nbgh\nh$ and by \eqref{eq: geometric errors 1 UNS} we know that  $\norm{\bfH- \bfH_h}_{L^\infty(\Gah)} \leq ch^{k_g -1}$. The idea is that with the help of the product rule $\nbg(\bfw\cdot\bfn)^t= \bfw\cdot\nbg \bfn + \bfn\cdot(\nbg\bfw)$ and integration by parts \eqref{eq: integration by parts cont UNS} we could gain one approximation order by shifting the derivative to the test function instead.

First, let us add and subtract an appropriate term, the reason for which we will discuss just right after: \begin{equation}
    \begin{aligned}\label{eq: inside weingarten map 1 UNS}
        &\Big|\int_{\Gah} \beta^{-\ell}\,\vh\cdot\bfH_h\bfu^{-\ell} \, \dsh - \int_{\Ga} \beta\,\vhl\cdot\bfH\bfu \, \ds\Big|\\
    &\leq \Big|\int_{\Gah} \beta^{-\ell}\,\vh\cdot\bfH_h\bfu^{-\ell} \,\dsh -\int_{\Ga} \beta\,\vhl\cdot\nbgcov\nhl\bfu \, \ds \Big| + \Big|\int_{\Ga} \beta\,\vhl\cdot\nbgcov\nhl\bfu \, \ds -\int_{\Ga} \beta\,\vhl\cdot\bfH\bfu \, \ds  \Big|.
    \end{aligned}
\end{equation}

\noindent Recall that $\bfH=\nbgcov\bfn = \bfPg\nbg\bfn = \bfPg\nb\bfn\bfPg$. By lifting; see \cref{sec: surface lifting UNS}, we get $\bfH_h = \bfPh\nbg\nhl\bfB_h^t$. With that in mind, along with the estimates $\norm{\bfP- \bfPh}_{L^\infty(\Gah)} \leq ch^{k_g}$, cf. \eqref{eq: geometric errors 2 UNS}, $\norm{1-\muh}_{L^{\infty}(\Gah)} \leq ch^{k_g+1}$, cf. \eqref{eq: muhkg estimate UNS}, and $ \norm{\bfPg - \Bhg}_{L^{\infty}(\Ga)} \leq ch^{k_g}$ cf. \eqref{eq: Bh estimates UNS}, and the embedding $L^4 \hookrightarrow H^1$, we clearly see that
\begin{equation}
    \begin{aligned}\label{eq: inside weingarten map 2 UNS}
        \Big|\int_{\Gah} \beta^{-\ell}\,\vh\cdot\bfH_h\bfu^{-\ell} \,\dsh -\int_{\Ga} \beta\,\vhl\cdot\nbgcov\nhl\bfu \, \ds \Big| &= \Big|\int_{\Ga} \frac{1}{\mu_h} \beta\,\vhl\cdot\bfPh\nbg\nhl\bfB_h^t\bfu \,\ds -\int_{\Ga} \beta\,\vhl\cdot\bfPg\nbg\nhl\bfu \, \ds \Big|\\
        &\leq ch^{k_g}\norm{\bfu}_{H^1(\Ga)}\norm{\beta}_{H^1(\Ga)}\norm{\vh}_{H^1(\Gah)},
    \end{aligned}
\end{equation}
where we also also used the norm equivalence \eqref{eq: norm equivalence UNS}. For the second absolute value term in \eqref{eq: inside weingarten map 1 UNS}, using the product rule we see that
\begin{equation}
    \begin{aligned}\label{eq: inside weingarten map 3 UNS}
        \int_{\Ga} \beta\,\vhl\cdot\bfH\bfu \, \ds= \int_{\Ga} \beta\, (\vhl \cdot \bfPg\nbg \bfn)\bfu \, \ds = \underbrace{\int_{\Ga} \beta\, \nbg(\bfPg\vhl\cdot\bfn)^t \bfu \, \ds}_{=0} - \int_{\Ga} \beta\, (\bfn \cdot \nbg \bfPg\vhl)\bfu \, \ds.
    \end{aligned}
\end{equation}
Furthermore, using partial integration $\int_{\Ga}\bfu\cdot\nbg q = -\int_{\Ga} \divg \bfu\,q  + \int_{\Ga}tr(\bfH)\,q \, (\bfu\cdot\bfng) $, we obtain the following, due to the fact that $\bfu\cdot\bfng=0$,
\begin{equation}
    \begin{aligned}\label{eq: inside weingarten map 4 UNS}
        \int_{\Ga} \beta\,\vhl\cdot\nbgcov\nhl\bfu \, \ds &= \int_{\Ga} \nbg(\bfPg\vhl\cdot \nhl)^t \beta\bfu \, \ds - \int_{\Ga} \beta\, (\nhl \cdot \nbg \bfPg\vhl)\bfu \\
        &=- \int_{\Ga} \divg(\beta\bfu)\bfPg\vhl\cdot \nhl \, \ds - \int_{\Ga} \beta\, (\nhl \cdot \nbg \bfPg\vhl)\bfu.
    \end{aligned}
\end{equation}
Combining now \eqref{eq: inside weingarten map 3 UNS}, \eqref{eq: inside weingarten map 4 UNS}, along with the fact that $\norm{\bfn- \nh}_{L^\infty(\Gah)} \leq ch^{k_g}$, and $\norm{\bfPg\nhl}_{L^\infty(\Ga)} \leq ch^{k_g}$; see \eqref{eq: geometric errors 1 UNS}, \eqref{eq: geometric errors 2 UNS}, it follows that
\begin{equation}
    \begin{aligned}\label{eq: inside weingarten map 5 UNS}
        \Big|\int_{\Ga} \beta\,\vhl\cdot\nbgcov\nhl\bfu \, \ds -\int_{\Ga} \beta\,\vhl\cdot\bfH\bfu \, \ds  \Big|  \leq ch^{k_g}\norm{\beta}_{H^1(\Ga)}\norm{\bfu}_{H^1(\Ga)}\norm{\vh}_{H^1(\Gah)},
    \end{aligned}
\end{equation}
where in \eqref{eq: inside weingarten map 4 UNS} we also used the fact that $ \divg(\beta\bfu) = \beta \,\divg\bfu + \bfu\cdot\nbg\beta$ and the inclusion $L^4 \hookrightarrow H^1$. Inserting the results \eqref{eq: inside weingarten map 2 UNS}, \eqref{eq: inside weingarten map 5 UNS} into \eqref{eq: inside weingarten map 1 UNS} gives our desired estimates.
\end{proof}


\subsection{Surface Ritz-Stokes projection}\label{sec: Ritz-Stokes projection UNS}

Ritz maps have been studied in the case of elliptic operators on stationary and evolving surfaces \cite{DziukElliott_L2,kovacs2016error,li2023optimal}, while Ritz-Stokes maps have been studied for stationary domains in the case of no geometric variational crimes, that is, $\Omega_h = \Omega$ \cite{AyusoPostNS2005,FrutosGradDivOseen2016,burman2009galerkin} and evolving domains \cite{rao2025optimal} using ALE FEM. In this work, we extend these results to the case of discrete surfaces, where $\Gah \neq \Ga$. We will consider a standard Ritz-Stokes projection and a modified Ritz-Stokes projection similar to the one introduced for the stationary domain case in \cite{AyusoPostNS2005,FrutosGradDivOseen2016,JohnBook2016}.

For convenience, from now on we shall sometimes omit the inverse lift extension $(\cdot)^{-\ell}$ notation, unless specified, since it should be clear from the context if e.g. a function $\bfu$ is defined on $\Ga$ or $\Gah$.

\begin{definition}[Modified Ritz-Stokes projection]\label{def: surface Ritz-Stokes projection UNS} 
\noindent Given $(\bfu,\{p,\lambda\}) \in \bfH^1(\Ga)\times(L^2(\Ga)\times L^2(\Ga))$, define by $\mathcal{R}_h(\bfu) \in \bfV_h^{div}$, $\{\mathcal{P}_h(\bfu),\mathcal{L}_h(\bfu)\} \in Q_h \times \Lambda_h$ the unique projection
\begin{align}
    \begin{cases}\label{eq: surface Ritz-Stokes projection UNS} 
     a_h(\mathcal{R}_h(\bfu),\vh) + \bhtil (\vh,\{\mathcal{P}_h(\bfu),\mathcal{L}_h(\bfu)\}) = a(\bfu,\vhl), \\
     \bhtil(\mathcal{R}_h(\bfu),\{\qh,\xi_h\}) = 0,
    \end{cases}
\end{align}
 for all $(\vh,\{q_h ,\xi_h \}) \in \bfV_h\times (Q_h \times \Lambda_h)$. The lift is then defined typically $\big(\mathcal{R}_h^{\ell}(\bfu)$, $\{\mathcal{P}_h^{\ell}(\bfu), \,  \mathcal{L}_h^{\ell}(\bfu)\} \big) \in \bfV_h^{\ell}\times (Q_h^{\ell} \times \Lambda_h^{\ell})$.
\end{definition}

\begin{definition}[Ritz-Stokes projection]\label{def: surface Ritz-Stokes projection std UNS} 
\noindent Given $(\bfu,\{p,\lambda\}) \in \bfH^1(\Ga)\times(L^2(\Ga)\times L^2(\Ga))$, define by $\mathcal{R}_h^b(\bfu) = \mathcal{R}_h(\bfu,\{p,\lambda\})\in \bfV_h^{div}$, $\mathcal{P}^b_h(\bfu) = \mathcal{P}_h(\bfu,\{p,\lambda\}) \in Q_h , \, \mathcal{L}^b_h(\bfu)  = \mathcal{L}_h(\bfu,\{p,\lambda\}) \in \Lambda_h$ the unique projection
\begin{align}
    \begin{cases}\label{eq: surface Ritz-Stokes projection std UNS} 
     a_h(\mathcal{R}^b_h(\bfu),\vh) + \bhtil (\vh,\{\mathcal{P}^b_h(\bfu),\mathcal{L}^b_h(\bfu)\}) = a(\bfu,\vhl) + b^L(\vhl,\{p,\lambda\}), \\
     \bhtil(\mathcal{R}^b_h(\bfu),\{\qh,\xi_h\}) = 0,
    \end{cases}
\end{align}
 for all $(\vh,\{q_h ,\xi_h \}) \in \bfV_h\times (Q_h \times \Lambda_h)$. The lift is defined in the standard way.
\end{definition}

\noindent We notice due to the well-posedness and stability of the surface Stokes equations \cite[Theorem 5.4]{elliott2024sfem}, we have the following \emph{a-priori} stability estimates for the Ritz-Stokes projections $\mathcal{R}_h^{(\cdot)}(\bfu)$ $\big(:= \mathcal{R}_h(\bfu)$ or $\mathcal{R}_h^{b}(\bfu)\big)$
\begin{equation}\label{eq: Stab estimates for Ritz-Stokes projection UNS} 
\begin{aligned}
    \norm{\mathcal{R}_h(\bfu)}_{\ah}^2 + \norm{\{\mathcal{P}_h(\bfu),\mathcal{L}_h(\bfu)\}}_{L^2(\Gah)}^2 \leq c\norm{\bfu}_{H^1(\Ga)}^2 , \\
    \norm{\mathcal{R}^b_h(\bfu)}_{\ah}^2 + \norm{\{\mathcal{P}_h^b(\bfu),\mathcal{L}^b_h(\bfu)\}}_{L^2(\Gah)}^2 \leq c\big( \norm{\bfu}_{H^1(\Ga)}^2 + \norm{\{p,\lambda\}}_{L^2(\Ga)}^2 \big).
    \end{aligned}
\end{equation}

\begin{lemma}[pseudo Galerkin orthogonality]\label{lemma: pseudo Galerkin orthogonality UNS}
For every $\bfu\in \bfH^1(\Ga)$ it holds that
\begin{align}
\label{eq: pseudo Galerkin orthogonality Ritz 1 UNS}
      \ah((\bfu -\mathcal{R}_h(\bfu), \vh) \leq Ch^{k_g}\norm{\vh}_{H^1(\Gah)} \quad  \text{ for all } \vh\in\bfV_h^{div},
    \end{align}
\begin{align}
   \label{eq: pseudo Galerkin orthogonality Ritz b UNS}
      \ah((\bfu -\mathcal{R}_h^b(\bfu), \vh)  + \bhtil(\vh,\{p,\lambda\} - \{\mathcal{P}^b_h(\bfu),\mathcal{L}^b_h(\bfu)\})\leq C^bh^{k_g-1}\norm{\vh}_{\ah} \quad  \text{ for all } \vh\in\bfV_h, 
\end{align}
where the constant $C = c\norm{\bfu}_{H^1(\Ga)}$ and $C^b = c(\norm{\bfu}_{H^1(\Ga)} + \norm{p}_{H^1} + \norm{\lambda}_{L^2(\Ga)})$.
\end{lemma}
\begin{proof}
Let us consider  \eqref{eq: surface Ritz-Stokes projection std UNS} since the Galerkin orthogonality for $\mathcal{R}_h(\bfu)$ follows similarly. Using the perturbation estimates \eqref{eq: Geometric perturbations a UNS}, \eqref{eq: Geometric perturbations btilde tangent UNS} and the norm equivalence \eqref{eq: norm equivalence UNS} we get for any $\vh\in\bfV_h^{div}$ that
\begin{equation*}
    \begin{aligned}
         \ah((\bfu -\mathcal{R}_h^b(\bfu), \vh) &+ \bhtil(\vh,\{p,\lambda\} - \{\mathcal{P}^b_h(\bfu),\mathcal{L}^b_h(\bfu)\}) \\
        &= \ah(\bfu,\vh) - a(\bfu,\vhl) +\bhtil(\vh,\{p,\lambda\})- b^L(\vhl,\{p,\lambda\})\\
        &\leq ch^{k_g}\norm{\bfu}_{H^1(\Ga)}\norm{\vh}_{H^1(\Gah)} + ch^{k_g}(\norm{p}_{H^1} + \norm{\lambda}_{L^2(\Ga)})\norm{\vh}_{\ah},
    \end{aligned}
\end{equation*}
where \eqref{eq: coercivity and Korn's inequality Lagrange UNS} gives the desired result.
\end{proof}

\begin{lemma}[Error Bounds modified Ritz-Stokes]\label{lemma: Error Bounds Ritz-Stokes UNS}
 Let  $(\bfu,\{p,\lambda\}) \in \bfH^1(\Ga)\times(H^1(\Ga)\times L^2(\Ga))$ the solution of \eqref{weak lagrange hom NV}. Then, we have the following error bounds for the Ritz-Stokes projection for $a=0,1$
 \begin{align}
     \label{eq: Error Bounds Ritz-Stokes UNS}
         \norm{\partial_t^{a}(\bfu - \mathcal{R}_h(\bfu))}_{\ah}  + \norm{\partial_t^{a}\{\mathcal{P}_h(\bfu),\mathcal{L}_h(\bfu)\}}_{L^2(\Gah)}&\leq ch^{r_u}\norm{\partial_t^{a}\bfu}_{H^{k_u+1}(\Ga)}\\
         \label{eq: Error Bounds Ritz-Stokes L2 UNS}
         \norm{\partial_t^{a}\bfPg(\bfu - \mathcal{R}_h^{\ell}(\bfu))}_{L^2(\Ga)} &\leq ch^{r_u+1}\norm{\partial_t^{a}\bfu}_{H^{k_u+1}(\Ga)},
 \end{align}
 with $r_u = min\{k_u,k_g-1\}$ and $k_g \geq 2$, for $h \leq h_0$ with sufficiently small $h_0$ and with $c$ independent of $h$.

If furthermore we assume that $\underline{k_{\lambda} = k_u}$ (that is $V_h= \Lambda_h$) we obtain the following improved estimates 
\begin{align}
    \label{eq: Error Bounds Ritz-Stokes improved UNS}
         \norm{\partial_t^{a}(\bfu - \mathcal{R}_h(\bfu))}_{\ah}  + \norm{\partial_t^{a}\{\mathcal{P}_h(\bfu),\mathcal{L}_h(\bfu)\}}_{L^2(\Gah)\times H_h^{-1}(\Gah)} &\leq ch^{\widehat{r}_u}\norm{\partial_t^{a}\bfu}_{H^{k_u+1}(\Ga)}\\
         \label{eq: Error Bounds Ritz-Stokes improved L2 Lh UNS}
         \norm{\partial_t^{a}\mathcal{L}_h(\bfu)}_{L^2(\Gah)} \leq c(h^{\widehat{r}_u}+ h^{k_g-1})\norm{\partial_t^{a}\bfu}_{H^{k_u+1}(\Ga)} &\leq ch^{r_u}\norm{\partial_t^{a}\bfu}_{H^{k_u+1}(\Ga)}\\
          \label{eq: Error Bounds Ritz-Stokes improved H1 UNS}
         \norm{\partial_t^{a}(\bfu - \mathcal{R}_h(\bfu))}_{H^1(\Gah)} &\leq ch^{\widehat{r}_u}\norm{\partial_t^{a}\bfu}_{H^{k_u+1}(\Ga)}
         \\
         \label{eq: Error Bounds Ritz-Stokes L2 improved UNS}
        \norm{\partial_t^{a}\bfPg(\bfu - \mathcal{R}_h^{\ell}(\bfu))}_{L^2(\Ga)} + h^{1/2} \norm{\partial_t^{a}\big((\bfu - \mathcal{R}_h(\bfu))\cdot\bfng\big)}_{L^2(\Ga)}&\leq ch^{\widehat{r}_u+1}\norm{\partial_t^{a}\bfu}_{H^{k_u+1}(\Ga)}.
\end{align}
with $\widehat{r}_u = min\{k_u,k_g\}$, for $h \leq h_0$ with sufficiently small $h_0$ and with constant $c$ independent of $h$. 
\end{lemma} 
\begin{proof}
We focus on the case where $\underline{k_\lambda = k_u}$, since we will see that the general case follows in a similar way. Let us start with ($a=0$). First, consider $\uh \in \bfV_h$ to be the finite element solution of the following surface Stokes problem:
\begin{equation}
    \begin{aligned}\label{eq: Ritz-Stokes proof inside 1 UNS}
        a(\bfu,\bfv) \ + b^L(\bfv,\{p,\lambda\}) &= (\bfg,\bfv)_{L^2(\Ga)} \ \ \ \ \ \text{for all } \bfv\in \bfH^1(\Ga),\\
        b^L(\bfu,\{q,\xi\})&=0 \ \ \  \text{ for all } \{q,\xi\}\in (L^2_0(\Ga)\times L^2(\Ga)),
    \end{aligned}
\end{equation}
with $\bfg = \bff - \partial_t\bfu - \bfu\cdot \nbgcov\bfu  - \nbg p - \lambda \bfng$  in the sense of \cite{elliott2024sfem}, where we consider approximation of the Rhs as $\bfg_h = \bfg^{-\ell}$.
Since, by assumptions, we consider $(\bfu,\{p,\lambda\})$ to be the solution of \eqref{weak lagrange hom NV}, we can readily see that then $(\bfu,\{0,0\})$ is the solution of \eqref{eq: Ritz-Stokes proof inside 1 UNS}. 

To apply, now, the finite element approximations in \cite[Theorem 6.13, Theorem 6.14, Corollary 6.1]{elliott2024sfem}, we need $\bfg \in L^2(\Ga)$, therefore, we require sufficiently regularity, i.e. $\partial_t\bfu, \bff,\lambda 
\in L^2(\Ga)$, $\bfu,\,p \in H^1(\Ga)$. Then, we readily see, with the help of known consistency errors, that the following error estimates hold, 
\begin{align}
\label{eq: Ritz-Stokes proof inside UNS}
    \norm{\bfu - \uh}_{\ah} +  \norm{\bfu - \uh}_{H^1(\Gah)} &\leq ch^{\widehat{r}_u}\norm{\bfu}_{H^{k_u+1}(\Ga)},\\
         \label{eq: Ritz-Stokes proof inside 3 UNS}
         \norm{\bfP(\bfu - \uh)}_{L^2(\Gah)} &\leq ch^{\widehat{r}_u+1}\norm{\bfu}_{H^{k_u+1}(\Ga)},
\end{align}
with $\widehat{r}_u = min\{k_u,k_g\}$ and $c = c_1(\norm{\partial_t\bfu}_{L^2(\Ga)} +\norm{\bfu}_{H^1(\Ga)}\norm{\bfu}_{H^1(\Ga)} +\norm{p}_{H^1(\Ga)} + \norm{\lambda}_{L^2(\Ga)})$. We are now able to prove the required estimates.

Let us start with the \emph{energy estimate} of \eqref{eq: Error Bounds Ritz-Stokes improved UNS}, where combining the above bounds \eqref{eq: Ritz-Stokes proof inside UNS}, the pseudo Galerkin orthogonality \eqref{eq: pseudo Galerkin orthogonality Ritz 1 UNS} (recall $\uh - \mathcal{R}_h(\bfu) \in \bfV_h^{div}$) and the improved $H^1$ coercivity estimate \eqref{eq: improved h1-ah bound UNS} we obtain,
\begin{align}
   \norm{\bfu - \mathcal{R}_h(\bfu)}_{\ah}^2 &= \ah(\bfu - \mathcal{R}_h(\bfu),\bfu-\uh) + \ah(\bfu - \mathcal{R}_h(\bfu),\uh - \mathcal{R}_h(\bfu)) \nonumber\\
        &\leq \norm{\bfu - \mathcal{R}_h(\bfu)}_{\ah}\norm{\bfu - \uh}_{\ah} + ch^{k_g}\norm{\bfu - \mathcal{R}_h(\bfu)}_{\ah}\big(\norm{\uh- \bfu}_{\ah} + \norm{\bfu - \mathcal{R}_h(\bfu)}_{\ah}\big) \nonumber\\
        \label{eq: inside vdiv UNS}
        &\leq ch^{\widehat{r}_u}\norm{\bfu}_{H^{k_u+1}(\Ga)}\norm{\bfu - \mathcal{R}_h(\bfu)}_{\ah} +ch^{k_g}\norm{\bfu - \mathcal{R}_h(\bfu)}_{\ah}^2. 
\end{align}
The energy bound in \eqref{eq: Error Bounds Ritz-Stokes improved UNS} is then derived after a simple kickback argument. 

The \emph{pressure bound} in \eqref{eq: Error Bounds Ritz-Stokes improved UNS} can be derived by considering the $L^2\times H_h^{-1}$ \textsc{inf-sup} condition \eqref{lemma: L^2 H^{-1} discrete inf-sup condition Gah Lagrange UNS}, the Ritz-Stokes projection definition \eqref{eq: surface Ritz-Stokes projection UNS} and perturbation estimate \eqref{eq: Geometric perturbations a UNS}, to see that
\begin{equation*}
    \begin{aligned}
        \norm{\{\mathcal{P}_h(\bfu),\mathcal{L}_h(\bfu)\}}_{L^2(\Gah)\times H_h^{-1}(\Gah)} &\leq  \sup_{\vh\in\bfV_h} \frac{a(\bfu,\vhl)-a_h(\bfu,\vh) + \ah(\bfu-\mathcal{R}_h(\bfu),\vh)}{\norm{\vh}_{H^1(\Gah)}} \\
        &\leq ch^{k_g}\norm{\bfu}_{H^1(\Ga)} + \norm{\bfu-\mathcal{R}_h(\bfu)}_{\ah} \leq ch^{\widehat{r}_u}\norm{\bfu}_{H^{k_u+1}(\Ga)}.
    \end{aligned}
\end{equation*}
Using the $L^2 \times L^2$ \textsc{inf-sup} condition \eqref{Lemma: Discrete inf-sup condition Gah Lagrange UNS} instead, and the worse $H^1$ coercivity bound $\norm{\vh}_{H^1(\Gah)}\leq ch^{-1}\norm{\vh}_{\ah}$ \eqref{eq: coercivity and Korn's inequality Lagrange UNS}, we obtain, similarly, the following $L^2$ bound for $\mathcal{L}_h(\bfu)$:
\begin{equation*}
    \begin{aligned}
        \norm{\mathcal{L}_h(\bfu)}_{L^2(\Gah)} \leq ch^{k_g-1}\norm{\bfu}_{H^1(\Ga)} + \norm{\bfu-\mathcal{R}_h(\bfu)}_{\ah} \leq c(h^{\widehat{r}_u}+ h^{k_g-1})\norm{\bfu}_{H^{k_u+1}(\Ga)}.
    \end{aligned}
\end{equation*}

Regarding the $H^1$ \emph{bound} \eqref{eq: Error Bounds Ritz-Stokes improved H1 UNS}, this follows from \eqref{eq: Ritz-Stokes proof inside UNS} and the $H^1$ coercivity bound \eqref{eq: improved h1-ah bound UNS} 
\begin{equation*}
\begin{aligned}
    \norm{\bfu - \mathcal{R}_h(\bfu)}_{H^1(\Gah)} &\leq \norm{\bfu - \uh}_{H^1(\Gah)} + \norm{\uh - \mathcal{R}_h(\bfu)}_{H^1(\Gah)} \leq \norm{\bfu - \uh}_{H^1(\Gah)} + \norm{\uh - \mathcal{R}_h(\bfu)}_{\ah}\\
    &\leq \norm{\bfu - \uh}_{H^1(\Gah)} +  \norm{\bfu - \uh}_{\ah} + \norm{\bfu - \mathcal{R}_h(\bfu)}_{\ah} \leq ch^{\widehat{r}_u}\norm{\bfu}_{H^{k_u+1}(\Ga)}.
\end{aligned}
\end{equation*}
Collecting the above results concludes the estimates in \eqref{eq: Error Bounds Ritz-Stokes improved UNS}, \eqref{eq: Error Bounds Ritz-Stokes improved L2 Lh UNS} and \eqref{eq: Error Bounds Ritz-Stokes improved H1 UNS}.

To bound the \emph{tangential} $L^2$\emph{-norm} \emph{velocity} \eqref{eq: Error Bounds Ritz-Stokes L2 improved UNS}, we need to use a duality argument. We follow \cite[Theorem 6.14]{elliott2024sfem}, so let us consider 
$(\bfw, \pi,\mu) \in \bfH^1(\Ga) \times L^2_0(\Ga)\times L^2(\Ga)$ that satisfy 
\begin{align}
\begin{cases}\label{eq: dual Ritz inside begin UNS}
        a(\bfw,\bfv) \ + \!\!\!\!& b^L(\bfv,\{\pi,\mu\}) = (\bfPg(\bfu - \mathcal{R}_h^{\ell}(\bfu)),\bfv) \ \ \ \ \ \text{for all } \bfv\in \bfH^1(\Ga),\\
        &b^L(\bfw,\{\sigma,\xi\})=0 \ \ \  \text{ for all } \{\sigma,\xi\} \in L_0^2(\Ga)\times L^2(\Ga),
    \end{cases}
\end{align}
where we see that $\bfw \cdot \bfng =0$. Now, the surface $\Ga$ is sufficiently smooth and thus due to \cite[Lemma 2.1]{olshanskii2021inf} the solution also satisfies the following higher regularity estimate
\begin{equation}\label{eq: regularity estimate lagrange UNS}
    \norm{\bfw}_{H^2(\Ga)} + \norm{\{\pi,\mu\}}_{H^1(\Ga)} \leq \norm{\bfPg(\bfu - \mathcal{R}_h^{\ell}(\bfu))}_{L^2(\Ga)}.
\end{equation} 
Set $\bfe_h = \bfu - \mathcal{R}_h\bfu$. Then, testing with $\bfe_h^{\ell} = \bfu - \mathcal{R}_h^{\ell}\bfu$, we clearly see that 
\begin{equation}
    \begin{aligned}\label{eq: dual Ritz inside main UNS}
        \boxed{\norm{\bfPg\bfe_h^{\ell}}_{L^2(\Ga)}^2 =  a(\bfe_h^{\ell},\bfw) + b^L(\bfe_h^{\ell},\{\pi,\mu\}).}
    \end{aligned}
\end{equation}
We need to bound the two terms on the right-hand side of \eqref{eq: dual Ritz inside main UNS} appropriately. Let us start with the first term. 
We see that 
\begin{equation}
    \begin{aligned}\label{eq: dual Ritz inside a begin UNS}
         a(\bfe_h^{\ell},\bfw) = a(\bfe_h^{\ell},\bfw - \Ihl(\bfw)) + a(\bfe_h^{\ell},\Ihl(\bfw)) - a_h(\bfe_h,\Ih(\bfw)) + a_h(\bfe_h,\Ih(\bfw)).
    \end{aligned}
\end{equation}
Let us bound each term of \eqref{eq: dual Ritz inside a begin UNS} separately. As mentioned before, we have $k_\lambda = k_u$, therefore we are going to make use of the previously proven estimates \eqref{eq: Error Bounds Ritz-Stokes improved UNS}-\eqref{eq: Error Bounds Ritz-Stokes improved H1 UNS}, that is, $\norm{\bfe_h}_{\ah} \leq \norm{\bfe_h}_{H^1(\Gah)} \leq ch^{\widehat{r}_u}\norm{\bfu}_{H^{k_u+1}}$.

\begin{itemize}
\item By standard interpolation estimates \eqref{eq: lagrange interpolant UNS} and norm equivalence \eqref{eq: norm equivalence UNS} we first see that
\begin{equation}
    \begin{aligned}\label{eq: dual Ritz inside a 1 UNS}
        a(\bfe_h^{\ell},\bfw - \Ihl(\bfw)) &\leq c\norm{\bfw - \Ihl(\bfw)}_{H^1(\Ga)}\norm{\bfe_h^{\ell}}_{H^1(\Ga)} \\
        &\leq ch^{\widehat{r}_u+1}\norm{\bfw}_{H^2(\Ga)}\norm{\bfu}_{H^{k_u+1}(\Ga)}.
    \end{aligned}
\end{equation}
\item  We also see from the perturbation bound \eqref{eq: Geometric perturbations a UNS} that
\begin{equation}
    \begin{aligned}\label{eq: dual Ritz inside a 2 UNS}
        a(\bfe_h^{\ell},\Ihl(\bfw)) - a_h(\bfe_h,\Ih(\bfw)) &\leq ch^{k_g}\norm{\Ih(\bfw)}_{H^1(\Gah)}\norm{\bfe_h}_{H^1(\Gah)}\\
        &\leq ch^{k_g + \widehat{r}_u}\norm{\bfw}_{H^2(\Ga)}\norm{\bfu}_{H^{k_u+1}(\Ga)}.
    \end{aligned}
\end{equation}
\item It remains to  bound the last term in \eqref{eq: dual Ritz inside a begin UNS}. Using the definition of the Ritz-Stokes projection \eqref{eq: surface Ritz-Stokes projection UNS} we have the following
\begin{equation}
    \begin{aligned}\label{eq: dual Ritz inside a 3 UNS}
        a_h(\bfe_h,\Ih(\bfw)) &= a_h(\bfu,\Ih(\bfw)) - a(\bfu,\Ihl(\bfw)) + \bhtil(\Ih(\bfw),\{\mathcal{P}_h(\bfu),\mathcal{L}_h(\bfu)\}) \\
        &= \underbrace{a_h(\bfu,\Ih(\bfw)-\bfw) - a(\bfu,\Ihl(\bfw)-\bfw)}_{\eqref{eq: Geometric perturbations a UNS} \leq ch^{k_g+1}\norm{\bfw}_{H^2(\Ga)}\norm{\bfu}_{H^2(\Ga)}} + \!\!\!\!\!\!\!\!\!\!\underbrace{a_h(\bfu,\bfw) -a(\bfu,\bfw)}_{\eqref{eq: Geometric perturbations a tangent extra regularity UNS}\leq ch^{k_g+1}\norm{\bfw}_{H^2(\Ga)}\norm{\bfu}_{H^2(\Ga)}}\\
        &\ + \bhtil(\Ih(\bfw)-\bfw,\{\mathcal{P}_h(\bfu),\mathcal{L}_h(\bfu)\}) + \bhtil(\bfw,\{\mathcal{P}_h(\bfu),\mathcal{L}_h(\bfu)\}).
    \end{aligned}
\end{equation}
Now using once again the standard interpolation estimates, the bound of $\bhtil$ \eqref{bhtilde boundedness UNS} and the previously proven estimates \eqref{eq: Error Bounds Ritz-Stokes improved UNS}, \eqref{eq: Error Bounds Ritz-Stokes improved L2 Lh UNS} we obtain
\begin{equation}\label{eq: dual Ritz inside a 4 UNS}
    \bhtil(\Ih(\bfw)-\bfw,\{\mathcal{P}_h(\bfu),\mathcal{L}_h(\bfu)\}) \leq c(h^{\widehat{r}_u+1}+h^{r_u+2})\norm{\bfu}_{H^{k_u+1}(\Ga)}.
\end{equation}
Furthermore, since $b^L(\bfw,\{\mathcal{P}^{\ell}_h(\bfu),\mathcal{L}^{\ell}_h(\bfu)\}) =0$ by \eqref{eq: dual Ritz inside begin UNS}, using the estimates \eqref{eq: Error Bounds Ritz-Stokes improved UNS}, \eqref{eq: Error Bounds Ritz-Stokes improved L2 Lh UNS}, the perturbation bounds \eqref{eq: Geometric perturbations btilde tangent UNS}, and considering $\widetilde{\bfw}_h \in \bfV_h^{div}$ and $\widetilde{\bfr}_h = \Ih(\bfw) -  \widetilde{\bfw}_h \in \bfV_h$ we have
\begin{equation*}
    \begin{aligned}
        \hspace{8mm}&\bhtil(\bfw,\{\mathcal{P}_h(\bfu),\mathcal{L}_h(\bfu)\}) = b_h(\bfw,\mathcal{P}_h(\bfu)) - b(\bfw,\mathcal{P}^{\ell}_h(\bfu)) + (\bfw\cdot\nh,\mathcal{L}_h(\bfu))_{L^2(\Gah)}\\
        &\leq ch^{k_g}\norm{\bfw}_{L^2(\Ga)}\norm{\mathcal{P}_h(\bfu)}_{L^2(\Gah)} + ((\bfw-\Ih(\bfw))\cdot\nh,\mathcal{L}_h(\bfu))_{L^2(\Gah)}  \\
        &\quad + (\widetilde{\bfr}_h\cdot(\nh-\bfn),\mathcal{L}_h(\bfu))_{L^2(\Gah)} +  (\widetilde{\bfr}_h\cdot\bfn,\mathcal{L}_h(\bfu))_{L^2(\Gah)} \\
        &\leq c(h^{k_g+\widehat{r}_u} + h^{k_g+ 1 + r_u})\norm{\bfw}_{H^2(\Ga)}\norm{\bfu}_{H^{k_u+1}(\Ga)} + ch^{k_g}\norm{\widetilde{\bfr}_h}_{L^2(\Gah)}\norm{\mathcal{L}_h(\bfu)}_{L^2(\Gah)} \\
        &\quad + (\underbrace{\Ihz(\widetilde{\bfr}_h\cdot\bfn)}_{\in V_h = \Lambda_h },\mathcal{L}_h(\bfu))_{L^2(\Gah)} + \underbrace{(\widetilde{\bfr}_h\cdot\bfn - \Ihz(\widetilde{\bfr}_h\cdot\bfn)}_{\leq ch\norm{\widetilde{\bfr}_h}_{L^2(\Gah)} \text{ by } \eqref{eq: super-approximation estimate 2 UNS}} ,\mathcal{L}_h(\bfu))_{L^2(\Gah)}\\[-5pt]
       &\hspace{-3.5mm}\overset{\text{Def.} \eqref{eq: H^-1h definition UNS}}{\leq} \!\!\!\!\!ch^{\widehat{r}_u+1}\norm{\bfw}_{H^2(\Ga)}\norm{\bfu}_{H^{k_u+1}(\Ga)} 
        +ch\norm{\widetilde{\bfr}_h}_{L^2(\Gah)}\norm{\mathcal{L}_h(\bfu)}_{L^2(\Gah)}
        +c\norm{\widetilde{\bfr}_h}_{H^1(\Gah)}\norm{\mathcal{L}_h(\bfu)}_{H_h^{-1}(\Gah)}.
    \end{aligned}
\end{equation*}

\noindent Now taking the infimum over all $\widetilde{\bfw}_h \in \bfV_h^{div}$ and since $\bfw$ the solution to \eqref{eq: dual Ritz inside begin UNS}, from \cite[Lemma 6.12]{elliott2024sfem} we have that  $\inf_{\wh\in\bfV_h^{div}}\norm{\bfw^{-\ell} - \wh}_{H^1(\Gah)} \leq \inf_{\wh\in\bfV_h}\norm{\bfw^{-\ell} - \wh}_{H^1(\Gah)} + ch^{k_g}\norm{\bfw}_{H^1(\Ga)} \leq ch\norm{\bfw}_{H^2(\Ga)}$. Therefore it is easy to check, by the Scott-Zhang interpolant estimate \eqref{eq: bounds of Scott-Zhang interpolant UNS} and a simple triangle inequality, that 
$\inf_{\wh\in\bfV_h^{div}}\norm{\widetilde{\bfr}_h}_{H^1(\Gah)} \leq ch\norm{\bfw}_{H^2(\Ga)}$. Plugging this result and the estimates \eqref{eq: Error Bounds Ritz-Stokes improved UNS}, \eqref{eq: Error Bounds Ritz-Stokes improved L2 Lh UNS} to the previous bound, we arrive at
\begin{equation}
    \bhtil(\bfw,\{\mathcal{P}_h(\bfu),\mathcal{L}_h(\bfu)\}) \leq ch^{\widehat{r}_u+1}\norm{\bfw}_{H^2(\Ga)}\norm{\bfu}_{H^{k_u+1}(\Ga)},
\end{equation}
and therefore going back to \eqref{eq: dual Ritz inside a 3 UNS} we obtain
\begin{equation}
    \begin{aligned}\label{eq: dual Ritz inside a 4 UNS}
        \boxed{a_h(\bfe_h,\Ih(\bfw)) \leq ch^{\widehat{r}_u+1}\norm{\bfw}_{H^2(\Ga)}\norm{\bfu}_{H^{k_u+1}(\Ga)}.}
    \end{aligned}
\end{equation}

\end{itemize}

\noindent Inserting \eqref{eq: dual Ritz inside a 1 UNS},  \eqref{eq: dual Ritz inside a 2 UNS} and \eqref{eq: dual Ritz inside a 4 UNS} into \eqref{eq: dual Ritz inside a begin UNS} we finally obtain the bound for the first term:
\begin{equation}
    \begin{aligned}\label{eq: a(w,v) main UNS}
         a(\bfe_h^{\ell},\bfw) \leq ch^{\widehat{r}_u+1}\norm{\bfw}_{H^2(\Ga)}\norm{\bfu}_{H^{k_u+1}(\Ga)}.
    \end{aligned}
\end{equation}

\noindent We still need to bound $b^L(\bfe_h^{\ell},\{\pi,\mu\})$ in \eqref{eq: dual Ritz inside main UNS}. Adding and subtracting suitable terms we get
\begin{equation}
    \begin{aligned}\label{eq: dual Ritz inside a 5 UNS}
         b^L(\bfe_h^{\ell},\{\pi,\mu\}) &= b^L(\bfe_h^{\ell},\{\pi-\Ihl(\pi),\mu-\Ihl(\mu)\})  +b^L(\bfe_h^{\ell},\{\Ihl(\pi),\Ihl(\mu)\}) -\bhtil(\bfe_h,\{\Ih(\pi),\Ih(\mu)\})\\
        &\  + \bhtil(\bfe_h,\{\Ih(\pi),\Ih(\mu)\}).
    \end{aligned}
\end{equation}
Let us, once again, bound its term in \eqref{eq: dual Ritz inside a 5 UNS} separately:
\begin{itemize}
    \item By the boundedness of the bilinear form $b^{L}$, standard interpolation estimates, and the norm equivalence \eqref{eq: norm equivalence UNS} we get
    \begin{equation}\label{eq: dual Ritz inside a 6 UNS}
        \begin{aligned}
            b^L(\bfe_h^{\ell},\{\pi-\Ihl(\pi),\mu-\Ihl(\mu)\}) &\leq c\norm{\bfe_h}_{H^1(\Gah)}\norm{\pi-\Ihl(\pi)}_{L^2(\Ga)} + c\norm{\bfe_h}_{L^2(\Gah)}\norm{\mu-\Ihl(\mu)}_{L^2(\Ga)}\\
            &\leq ch^{\widehat{r}_u+1}\norm{\bfu}_{H^{k_u+1}(\Ga)}(\norm{\pi}_{H^1(\Ga)} + \norm{\mu}_{H^1(\Ga)}).
        \end{aligned}
    \end{equation}
    \item Using the perturbation estimate \eqref{eq: Geometric perturbations btilde tangent UNS} we obtain
    \begin{equation}\label{eq: dual Ritz inside a 7 UNS}
        \begin{aligned}
            b^L(\bfe_h^{\ell},\{\Ihl(\pi),\Ihl(\mu)\}) -\bhtil(\bfe_h,\{\Ih(\pi),\Ih(\mu)\}) &\leq ch^{k_g}\norm{\bfe_h}_{L^2(\Gah)}(\norm{\pi}_{H^1(\Ga)} + \norm{\mu}_{H^1(\Ga)}) \\
            &\leq ch^{k_g + \widehat{r}_u}\norm{\bfu}_{H^{k_u+1}(\Ga)}(\norm{\pi}_{H^1(\Ga)} + \norm{\mu}_{H^1(\Ga)}).
        \end{aligned}
    \end{equation}
\item For the last term in \eqref{eq: dual Ritz inside a 5 UNS}, first remember that $\bfu$ satisfies \eqref{eq: unstead NV Lagrange}. Then, using the definition of the Ritz-Stokes projection \eqref{eq: surface Ritz-Stokes projection UNS}, the perturbation estimate \eqref{eq: Geometric perturbations btilde tangent extra regularity UNS} and standard interpolation estimates we obtain\vspace{-4mm}
\begin{equation}
    \begin{aligned}\label{eq: dual Ritz inside a 8 UNS}
        \bhtil(\bfe_h,\{\Ih(\pi),\Ih(\mu)\}) &= \bhtil(\bfu,\{\Ih(\pi),\Ih(\mu)\}) - \overbrace{b^L(\bfu,\{\Ihl(\pi),\Ihl(\mu)\})}^{=0}\\
        &\leq ch^{k_g+1}\norm{\bfu}_{H^2(\Ga)}(\norm{\pi}_{H^1(\Ga)} + \norm{\mu}_{H^1(\Ga)}).
    \end{aligned}
\end{equation}
\end{itemize}
Combining \eqref{eq: dual Ritz inside a 6 UNS},\eqref{eq: dual Ritz inside a 7 UNS},\eqref{eq: dual Ritz inside a 8 UNS} into \eqref{eq: dual Ritz inside a 5 UNS} and remembering that $k_g \geq 1$ we obtain
\begin{equation}
    \begin{aligned}\label{eq: bL(u,Ih) UNS}
         \boxed{b^L(\bfe_h^{\ell},\{\pi,\mu\}) \leq  ch^{\widehat{r}_u+1}\norm{\bfu}_{H^{k_u+1}(\Ga)}(\norm{\pi}_{H^1(\Ga)} + \norm{\mu}_{H^1(\Ga)}).}
         \end{aligned}
\end{equation}

So, finally, inserting \eqref{eq: a(w,v) main UNS}, \eqref{eq: bL(u,Ih) UNS} into \eqref{eq: dual Ritz inside main UNS} and remembering the regularity \eqref{eq: regularity estimate lagrange UNS} the tangential $L^2$ estimate \eqref{eq: Error Bounds Ritz-Stokes L2 UNS} follows. The normal $L^2$ estimate in \eqref{eq: Error Bounds Ritz-Stokes L2 UNS} is calculated similarly to \cite[Lemma 6.16]{elliott2024sfem}; we omit further details for the sake of brevity.

With similar arguments, one can prove the bounds in \eqref{eq: Error Bounds Ritz-Stokes UNS} and \eqref{eq: Error Bounds Ritz-Stokes L2 UNS} for an arbitrary choice of $k_\lambda$ (i.e. $k_\lambda=k_u-1$) with $k_g \geq 2$. More precisely,  instead of \eqref{eq: Ritz-Stokes proof inside UNS}, and \eqref{eq: Ritz-Stokes proof inside 3 UNS} we consider the error estimates in \cite[Theorem 6.6, Theorem 6.7]{elliott2024sfem} and the worse $H^1$ coercivity bound \eqref{eq: coercivity and Korn's inequality Lagrange UNS}, therefore, we replace $\widehat{r}_u$ by $r_u =  min\{k_u,k_g-1\}$. Using the $L^2 \times L^2$ discrete Lagrange \textsc{inf-sup} condition \eqref{Lemma: Discrete inf-sup condition Gah Lagrange UNS} instead, we get the pressure estimate \eqref{eq: Error Bounds Ritz-Stokes UNS}. For the tangential $L^2$ velocity estimate \eqref{eq: Error Bounds Ritz-Stokes L2 UNS} we, again, need to apply a duality argument as before; it follows similarly to \cite[Theorem 6.7]{elliott2024sfem}. For the sake of brevity, we do not provide any further details since we would just repeat the calculations.

Finally, the convergence for the time derivative $(\alpha=1)$ is obtained in a similar way after taking the time derivative of the definition of the Ritz-Stokes projection \eqref{eq: surface Ritz-Stokes projection UNS}, and considering the derivative of \eqref{eq: Ritz-Stokes proof inside 1 UNS} with new $\bfg = \partial_t( \bff - \partial_t\bfu - \bfu\cdot \nbgcov\bfu  - \nbg p - \lambda \bfng)$.
\end{proof}

\begin{lemma}[Error Bounds Ritz-Stokes]\label{lemma: Error Bounds Ritz-Stokes std UNS}
 Let  $(\bfu,\{p,\lambda\}) \in \bfH^1(\Ga)\times(H^1(\Ga)\times L^2(\Ga))$ the solution of \eqref{weak lagrange hom NV}. Then, we have the following error bounds for the  Ritz-Stokes projection for $a=0,1$
 \begin{align}
     \label{eq: Error Bounds Ritz-Stokes std UNS}
         &\norm{\partial_t^{a}(\bfu - \mathcal{R}^b_h(\bfu))}_{\ah}  + \norm{\partial_t^{a}\{p-\mathcal{P}_h^b(\bfu),\lambda-\mathcal{L}_h^b(\bfu)\}}_{L^2(\Gah)}\leq c h^m \big(\norm{\partial_t^{a}\bfu}_{H^{k_u+1}(\Ga)} \nonumber\\
         &\qquad\qquad\qquad\qquad\qquad\qquad\qquad\qquad\qquad\qquad\quad \ \ \ \, +\norm{\partial_t^{a}p}_{H^{k_{pr}+1}(\Ga)} + \norm{\partial_t^{a}\lambda}_{H^{k_{\lambda}+1}(\Ga)} \big)\\
         \label{eq: Error Bounds Ritz-Stokes L2 std UNS}
         &\norm{\partial_t^{a}\bfPg(\bfu - \mathcal{R}_h^{b,\ell}(\bfu))}_{L^2(\Ga)} \leq ch^{m+1}\big(\norm{\partial_t^{a}\bfu}_{H^{k_u+1}(\Ga)}+\norm{\partial_t^{a}p}_{H^{k_{pr}+1}(\Ga)} + \norm{\partial_t^{a}\lambda}_{H^{k_{\lambda}+1}(\Ga)} \big).
 \end{align}
 with $m=min\{r_u,k_{pr}+1,k_{\lambda}+1\}$ and $r_u = min\{k_u,k_g-1\}$, $k_g \geq 2$, for $h \leq h_0$ with sufficiently small $h_0$ and with $c$ independent of $h$. 
 \end{lemma}

\begin{proof}
The proof follows similar to the proof of \cref{lemma: Error Bounds Ritz-Stokes UNS}, where one considers $\uh \in \bfV_h$, $\{p_h,\lambda_h\}\in Q_h\times \Lambda_h$ the Galerkin approximation of \eqref{eq: Ritz-Stokes proof inside 1 UNS} with $\bfg = \bff - \partial_t \bfu - \bfu\cdot\nbgcov\bfu$ instead. Therefore since $(\bfu,\{p,\lambda\})$ the solution to \eqref{weak lagrange hom NV}, $(\bfu,\{p,\lambda\})$ is also the solution of \eqref{eq: Ritz-Stokes proof inside 1 UNS}. Then, by \cite[Theorem 6.6, Theorem 6.7]{elliott2024sfem} we readily see, with the help of known consistency errors \cref{lemma: Geometric perturbations b lagrange UNS}, that the following error estimates hold, 
\begin{align}
\label{eq: Ritz-Stokes proof inside std UNS}
    \norm{\bfu - \uh}_{\ah} +  \norm{\{p - \ph,\lambda-\lh\}}_{L^2(\Gah)} &\leq ch^{m}\big(\norm{\bfu}_{H^{k_u+1}(\Ga)}+\norm{p}_{H^{k_{pr}+1}(\Ga)} + \norm{\lambda}_{H^{k_{\lambda}+1}(\Ga)} \big),\\
         \label{eq: Ritz-Stokes proof inside 3 std UNS}
         \norm{\bfP(\bfu - \uh)}_{L^2(\Gah)} &\leq ch^{m+1}\big(\norm{\bfu}_{H^{k_u+1}(\Ga)}+\norm{p}_{H^{k_{pr}+1}(\Ga)} + \norm{\lambda}_{H^{k_{\lambda}+1}(\Ga)} \big),
\end{align}
where $m=min\{r_u,k_{pr}+1,k_{\lambda}+1\}$ and $r_u = min\{k_u,k_g-1\}$ with  $k_g \geq 2$.
Furthermore, following \cref{lemma: Error Bounds Ritz-Stokes UNS} in \eqref{eq: inside vdiv UNS} we use \eqref{eq: pseudo Galerkin orthogonality Ritz b UNS} instead of  \eqref{eq: pseudo Galerkin orthogonality Ritz 1 UNS} along with the fact that $\bhtil(\vh,\{p,\lambda\}) = \bhtil(\vh,\{p,\lambda\} - \{\Ihz(p),\Ihz(\lambda)\}) \leq c(h^{k_{pr}+1} + h^{k_{\lambda}+1})$. Now the results \eqref{eq: Error Bounds Ritz-Stokes std UNS} and \eqref{eq: Error Bounds Ritz-Stokes L2 std UNS} can be obtained analogously to \cref{lemma: Error Bounds Ritz-Stokes UNS} and \cite[Section 6.3]{elliott2024sfem}. For the sake of brevity we skip the calculations. 
\end{proof}

\begin{lemma}[$L^{\infty}$ bound Ritz-Stokes projection]\label{lemma: Linfty estimate Ritz-Stokes UNS}
    Given $\bfu \in (W^{2,\infty}(\Ga))^3$ and $\underline{k_\lambda = k_u}$ or $\underline{k_\lambda = k_u-1}$ with $k_g\geq 2$, there exists a positive constant $c$ independent of $h$ such that the following estimate holds
    \begin{equation}\label{eq: W1infty estimate Ritz-Stokes UNS}
        \norm{\mathcal{R}_h^{(\cdot)} \bfu}_{L^{\infty}(\Gah)} +  \norm{\nbgcovh\mathcal{R}_h^{(\cdot)} \bfu}_{L^{\infty}(\Gah)} \leq c\norm{\bfu}_{W^{2,\infty}(\Ga)}.
    \end{equation}
\end{lemma}
\begin{proof}
Considering the $L^{\infty}$-bound \eqref{eq: lagrange interpolant infty UNS} and the Ritz-Stokes estimates in \cref{lemma: Error Bounds Ritz-Stokes UNS,lemma: Error Bounds Ritz-Stokes std UNS} we see that for sufficiently small mesh-parameter $h$,
\begin{equation}
    \begin{aligned}\label{eq: Linfty ritz stokes UNS}
        \norm{\mathcal{R}_h^{(\cdot)} \bfu^n}_{L^{\infty}(\Gah)} &\leq \norm{\mathcal{R}_h^{(\cdot)} \bfu^n - \Ih(\bfu^n)}_{L^{\infty}(\Gah)} + \norm{ \Ih(\bfu^n)}_{L^{\infty}(\Gah)} \\
         &\leq ch^{-1} \big( \norm{\mathcal{R}_h^{(\cdot)} \bfu^n - \bfu^n}_{L^{2}(\Gah)} + \norm{\Ih(\bfu^n) - \bfu^n}_{L^{2}(\Gah)} \big) + \norm{\Ih(\bfu^n)}_{L^{\infty}(\Gah)}\\ 
         &\leq c\norm{\bfu^n}_{W^{2,\infty}(\Ga)}.
    \end{aligned}
\end{equation}
 Similarly, a bound for the gradient of the Ritz-Stokes projection can be established; see also \cite{kovacs2016error}. 
\end{proof}

\subsection{Discrete surface Leray projection}\label{sec: Leray projection UNS}
For $\underline{k_\lambda = k_u}$ we define the discrete surface Leray projection $\Pi^{div}_h : L^2(\Ga_h)^3 \to \bfV_h^{div}$ to be the $L^2-$projection to the space of discrete weakly tangential  divergence-free tangential fields $\bfV_h^{div}$ \eqref{eq: discrete weakly divfree space UNS}, that is 
\begin{equation}
    \begin{aligned}\label{eq: discrete Leray UNS}
        (\Pi^{div}_h\bfu,\vh)_{L^2(\Gah)} = (\bfu^{-\ell},\vh)_{L^2(\Gah)} \qquad \forall \vh \in \bfV_h^{div}.
    \end{aligned}
\end{equation}
By definition, the projection is $L^2$-norm stable. If we want to achieve $H^1$-stability we need sufficiently smooth $\Ga$ and $\bfu \in \bfV^{div}$. 
Results of this type were shown in \cite{ChrysaWalk2010} in the case of quasi-uniform exact domain meshes. We extend them in the case of discrete surface $\Gah$.


\begin{lemma}[Leray Projection]\label{lemma: bounds discrete Leray projection UNS}
Let $\bfu \in \bfV^{div}$, and assume that $\underline{k_{\lambda} = k_u}$, then there exists a constant $c>0$ such that
\begin{equation}
    \begin{aligned}\label{eq: Leray estimate improved UNS}
        \norm{\bfP(\Pi^{div}_h\bfu- \bfu^{-\ell})}_{L^2(\Gah)}  + h\norm{\Pi^{div}_h\bfu - \bfu^{-\ell}}_{\ah} \leq ch^{\widehat{r}_u+1}\norm{\bfu}_{H^{k_u+1}(\Ga)},
    \end{aligned}
\end{equation}
with $\widehat{r}_u = min\{k_u,k_g\}$. The following stability result then also holds
    \begin{equation}\label{eq: Leray stability bound improved UNS}
        \norm{\Pi^{div}_h\bfu}_{a_h} \leq  \norm{\bfu^{-\ell}}_{\ah} \qquad \forall \bfu \in \bfV^{div}.     
     \end{equation}
\end{lemma}
\begin{proof}
Considering the modified Ritz-Stokes projection \eqref{eq: surface Ritz-Stokes projection UNS} we see that $\Pi^{div}_h\bfu-\mathcal{R}_h\bfu \in \bfV_h^{div}$ and therefore
\begin{equation*}
   \norm{\Pi^{div}_h\bfu- \bfu^{-\ell}}_{L^2(\Gah)}^2 = \underbrace{(\Pi^{div}_h\bfu- \bfu^{-\ell},\Pi^{div}_h\bfu-\mathcal{R}_h\bfu)_{L^2(\Gah)}}_{= 0 \ (\text{by} \eqref{eq: discrete Leray UNS})} + (\Pi^{div}_h\bfu - \bfu^{-\ell},\mathcal{R}_h\bfu - \bfu^{-\ell})_{L^2(\Gah)},
\end{equation*}
thus, $\norm{\Pi^{div,\ell}_h\bfu- \bfu^{-\ell}}_{L^2(\Gah)} \leq  \norm{\mathcal{R}_h \bfu - \bfu^{-\ell}}_{L^2(\Gah)} $. Using the Ritz-Stokes projection estimate \eqref{eq: Error Bounds Ritz-Stokes improved UNS} and the fact that $(\Pi^{div}_h\bfu- \bfu^{-\ell}) = \bfP(\Pi^{div}_h\bfu- \bfu^{-\ell}) + \bfn\otimes\bfn(\Pi^{div}_h\bfu- \bfu^{-\ell})$ we obtain, for the tangential $L^2$ norm, 
\begin{equation}
    \begin{aligned}\label{eq: leray inside 1 UNS}
        &\norm{\bfP(\Pi^{div}_h\bfu- \bfu^{-\ell})}_{L^2(\Gah)}^2 = (\bfP(\Pi^{div}_h\bfu- \bfu^{-\ell}),\Pi^{div}_h\bfu-\mathcal{R}_h(\bfu))_{L^2(\Gah)}\\
        &\quad + (\bfP(\Pi^{div}_h\bfu- \bfu^{-\ell}),\bfP(\mathcal{R}_h(\bfu)- \bfu^{-\ell}))_{L^2(\Gah)}\\
        &\leq \underbrace{((\Pi^{div}_h\bfu- \bfu^{-\ell})\cdot\bfn,(\Pi^{div}_h\bfu-\mathcal{R}_h(\bfu))\cdot\bfn)_{L^2(\Gah)}}_{\mathbf{(I)}}+ch^{\widehat{r}_u+1}\norm{\bfP(\Pi^{div}_h\bfu- \bfu^{-\ell})}_{L^2(\Gah)}.
    \end{aligned}
\end{equation}
We need to bound the first term above. Setting $\wh =\Pi^{div}_h\bfu-\mathcal{R}_h(\bfu) \in \bfV_h^{div} $, using \eqref{eq: divfree L2 norm kl=ku UNS}, \eqref{eq: divfree L2 norm kl=ku 2 UNS} when appropriate and recalling the geometric estimate  $\norm{\bfn-\nh}_{L^{\infty}} \leq ch^{k_g}$; see \eqref{eq: geometric errors 1 UNS}, we obtain

\begin{equation}
    \begin{aligned}\label{eq: leray inside 2 UNS}
        \mathbf{(I)} &\leq ch^{k_g}\norm{\Pi^{div}_h\bfu - \bfu^{-\ell}}_{L^2(\Gah)}\norm{\wh}_{L^2(\Gah)} + \norm{\Pi^{div}_h\bfu - \bfu^{-\ell}}_{L^2(\Gah)}\norm{\wh\cdot\nh}_{L^2(\Gah)}\\
       &\leq ch^{k_g}\norm{\mathcal{R}_h \bfu - \bfu^{-\ell}}_{L^2(\Gah)}\big( \norm{\bfP\wh}_{L^2(\Gah)} + h^{k_g}\norm{\wh}_{L^2(\Gah)} + \norm{\wh\cdot\nh}_{L^2(\Gah)}\big)\\
       &\ +  \norm{\mathcal{R}_h \bfu - \bfu^{-\ell}}_{L^2(\Gah)}\norm{\wh\cdot\nh}_{L^2(\Gah)}\\
       &\leq ch^{k_g}\norm{\mathcal{R}_h \bfu - \bfu^{-\ell}}_{L^2(\Gah)}\norm{\bfP\wh}_{L^2(\Gah)} + ch^{2k_g}\norm{\mathcal{R}_h \bfu - \bfu^{-\ell}}_{L^2(\Gah)}\norm{\wh}_{L^2(\Gah)}\\
       &\ + c\norm{\mathcal{R}_h \bfu - \bfu^{-\ell}}_{L^2(\Gah)}\big(h\norm{\bfP\wh}_{L^2(\Gah)} + h^{k_g+1}\norm{\wh}_{L^2(\Gah)}\big) \qquad \text{ (by \eqref{eq: divfree L2 norm kl=ku 2 UNS})}\\
       &\leq ch^{2k_g}\norm{\mathcal{R}_h \bfu - \bfu^{-\ell}}_{L^2(\Gah)}^2 + ch^{k_g+1}\norm{\mathcal{R}_h \bfu - \bfu^{-\ell}}_{L^2(\Gah)}^2 + \\
       &\ + ch\norm{\mathcal{R}_h \bfu - \bfu^{-\ell}}_{L^2(\Gah)}\big(\norm{\bfP(\Pi^{div}_h\bfu - \bfu^{-\ell})}_{L^2(\Gah)} +\norm{\bfP(\mathcal{R}_h \bfu - \bfu^{-\ell})}_{L^2(\Gah)}\big),
       \end{aligned}
\end{equation}
where we also used the fact that $\norm{\wh}_{L^2(\Gah)} \leq \norm{\Pi^{div}_h\bfu^{-\ell} - \bfu^{-\ell}}_{L^2(\Gah)} +\norm{\mathcal{R}_h \bfu - \bfu^{-\ell}}_{L^2(\Gah)} \leq c\norm{\mathcal{R}_h \bfu - \bfu^{-\ell}}_{L^2(\Gah)}$. Now substituting \eqref{eq: leray inside 2 UNS} into \eqref{eq: leray inside 1 UNS} and applying the Ritz-Stokes estimates \eqref{eq: Error Bounds Ritz-Stokes improved UNS}, \eqref{eq: Error Bounds Ritz-Stokes L2 improved UNS} along with the fact that $2k_g \geq k_g+1\geq 2$ for $k_g \geq 1$ we get
\begin{equation}
    \begin{aligned}
        \norm{\bfP(\Pi^{div}_h\bfu- \bfu^{-\ell})}_{L^2(\Gah)}^2 \leq ch^{2\widehat{r}_u +2}\norm{\bfu}_{H^{k_u+1}(\Ga)}^2 +ch^{\widehat{r}_u + 1}\norm{\bfP(\Pi^{div}_h\bfu - \bfu^{-\ell})}_{L^2(\Gah)}.
    \end{aligned}
\end{equation}
Then a simple application of Young's inequality and a kickback argument gives our tangential $L^2$ norm estimate in \eqref{eq: Leray estimate improved UNS}. Now, for the energy estimate in \eqref{eq: Leray estimate improved UNS} we use the velocity normal approximation \eqref{eq: divfree L2 norm kl=ku UNS}, and the previous tangential $L^2$ error estimate to obtain the following
    \begin{align}
        \norm{\Pi_h^{div}\bfu - \bfu^{-\ell}}_{\ah}
        & \leq \norm{\mathcal{R}_h \bfu - \bfu^{-\ell}}_{\ah} + ch^{-1}\norm{\Pi_h^{div}\bfu - \mathcal{R}_h \bfu}_{L^2(\Gah)} \nonumber\\
        &\leq  \norm{\mathcal{R}_h \bfu - \bfu^{-\ell}}_{\ah} + ch^{-1}\big( \norm{\bfP(\Pi_h^{div}\bfu - \mathcal{R}_h \bfu)}_{L^2(\Gah)}  + ch^{k_g}\norm{\Pi_h^{div}\bfu - \mathcal{R}_h \bfu}_{L^2(\Gah)}\nonumber\\
        &\ + \norm{(\Pi_h^{div}\bfu - \mathcal{R}_h \bfu)\cdot\nh}_{L^2(\Gah)}\big) \nonumber\\
        &\leq c(h^{\widehat{r}_u} + h^{k_g-1+\widehat{r}_u})\norm{\bfu}_{H^{k_u+1}(\Ga)},
    \end{align}

which gives us our desired result. Then the stability bound \eqref{eq: Leray stability bound improved UNS} follows from the previous result, and the Ritz-Stokes stability \eqref{eq: Stab estimates for Ritz-Stokes projection UNS}, since we can write the Leray projection as $\Pi_h^{div}\bfu = (\Pi_h^{div}\bfu-\mathcal{R}_h \bfu) + \mathcal{R}_h \bfu$. 
\end{proof}

\subsection{The variational formulation}\label{sec: var stability UNS}
With the help of the bilinear forms and the skew-symmetrized trilinear form introduced in \eqref{eq: lagrange discrete bilinear forms UNS}-\eqref{eq: lagrange discrete bilinear forms end UNS} we define the discrete variational formulation:

The linearized \emph{unsteady fully discrete Lagrangian surface Navier-Stokes} finite element method, which we analyze, is the following: 

\noindent {\bf(UWL):} Given appropriate approximation of the initial condition $\uh^0 \in V_h$, data $\bff_h$, find $\uh^n \in \bfV_h$, $\{\ph^n,\lh^n\} \in Q_h\times \Lambda_h$ for $1 \leq n \leq N$, such that 
\begin{align}
\begin{cases}
\tag{UWL}
    \label{eq: weak lagrange fully discrete UNS}
        (\frac{\uh^n-\uh^{n-1}}{\Delta t},\vh)_{L^2(\Gah)}+ \ah(\uh^n,\vh) +\chtil(\uh^{n-1};\uh^n,\vh) + \bhtil(\vh,\{\ph^n,\lh^n\}) \!\!&= (\bff_h^n,\vh) \\
    \phantom{aaaaaaaaaaaaaaaaaaaaaaaaaaaaaaaaaaaaaaaaaaaaaa}\bhtil(\uh^n,\{\qh,\xi_h\})\!\!&=0 ,
    \end{cases}
\end{align}
for all $(\vh,\{\qh,\xi_h\}) \in \bfV_h \times \{Q_h \times \Lambda_h\}$, where we have linearized the skew-symmetrized inertia term, according to \eqref{eq: skew-symmetrized ch real Ph} and \cref{sec: Time-stepping scheme UNS}. For this linearized formulation it is easy to see that a solution exists and that it is also unique. More precisely, in \Cref{Sec: Stability analysis UNS} we stablish stability estimates for the \emph{velocity} $\uh^n$ in appropriate norms. However, it is not possible to find stability estimates for the two \emph{pressures} $\{\ph^n,\lh^n\}$. To do so, we need $\uh^n$ to be uniformly bounded for each $n=1,...,N$ in the energy $(\ah)$-norm; see \eqref{eq: uniform bound b stability UNS}.

For that reason, we consider the following \emph{auxiliary unsteady fully discrete Lagrangian surface Navier-Stokes} method:

\noindent {\bf(UWL{\tiny b}):} Given appropriate approximation of the initial condition $\uh^0 \in \bfV_h$, data $\bff_h$, find $\uh^n \in V_h$, $\{\ph^n,\lh^n\} \in Q_h\times \Lambda_h$ for $1 \leq n \leq N$, such that,
\begin{align}
\begin{cases}
\tag{UWL{\tiny b}}
    \label{eq: weak lagrange fully discrete}
        (\frac{\uh^n-\uh^{n-1}}{\Delta t},\vh)_{L^2(\Gah)}+ \ah(\uh^n,\vh) +\chtil(\bfb_h^{n-1};\uh^n,\vh) + \bhtil(\vh,\{\ph^n,\lh^n\}) \!\!&= (\bff_h^n,\vh) \\
\phantom{aaaaaaaaaaaaaaaaaaaaaaaaaaaaaaaaaaaaaaaaaaaaaa}\bhtil(\uh^n,\{\qh,\xi_h\})\!\!&=0 ,
    \end{cases}
\end{align}
for all $(\vh,\{\qh,\xi_h\}) \in \bfV_h \times \{Q_h \times \Lambda_h\}$.

Notice that the problem \eqref{eq: weak lagrange fully discrete}, is the same as \eqref{eq: weak lagrange fully discrete UNS} when $\bfb_h^{n-1}=\uh^{n-1}$ in the inertia term $\chtil(\bullet;\bullet,\bullet)$. Here $\bfb_h$ is a vector-valued function which we, as mentioned above, assume is uniformly bounded in the following sense:
\begin{equation}\label{eq: uniform bound b stability UNS}
\sup_{n} \norm{\bfb_h^n}_{\ah} \leq C_b.
\end{equation}

This \emph{assumption} helps us prove stability for the two pressures in the following \Cref{Sec: Stability analysis UNS} (in reality just when $k_\lambda=k_u$; for arbitrary $k_\lambda$ see \cref{remark: About the pressure stability for arbitrary}), which otherwise would not be possible, due to the skew-symmetrized form of $\chtil(\bullet;\bullet,\bullet)$ and the use of $H^1-$conforming Taylor-Hood finite elements. 

This \emph{assumption}, however, is not all that unreasonable or strong. In particular, we are able to prove that solutions of the initial approximation of the surface Navier-Stokes problem \eqref{eq: weak lagrange fully discrete UNS}, i.e. when $\bfb_h^{n-1} = \uh^{n-1}$, actually possess this property, in a later Section cf. \Cref{remark: About pressure stability UNS}, once we prove velocity error estimates, assuming some low regularity assumptions are made for the solutions of the continuous problem \eqref{eq: unstead NV Lagrange}.

More specifically, with the help of the Ritz-Stokes projection \cref{def: surface Ritz-Stokes projection UNS}, the convergence analysis, and low regularity assumptions for the solutions of the continuous surface Navier-Stokes system \eqref{eq: unstead NV Lagrange}, we prove that \eqref{eq: uniform bound b stability UNS} actually also holds for our discrete solution $\bfu_h^{n}$ of \eqref{eq: weak lagrange fully discrete UNS}. Nonetheless, we prove the ensuing \emph{pressure  stability} results with the help of eq. \eqref{eq: weak lagrange fully discrete} assuming that \eqref{eq: uniform bound b stability UNS} holds.



Regarding the approximation of the initial data, we see that we have chosen that $\uh^0 \in \bfV_h$. Eventually, we will see that we will want to choose $\uh^0 \in \bfV_h^{div}$ e.g. $\uh^0 := \mathcal{R}_h\bfu^0$ or $\uh^0 := \mathcal{R}_h^b\bfu^0 = \mathcal{R}_h(\bfu^0, \{p^0,\lambda^0\})$. This means that the initial condition is then still \emph{weakly tangential divergence-free} at the time step $t_0=0$, thus the stability of the velocity is unconditionally stable, see also \Cref{remark: about initial conditions UNS}.

\section{Stability analysis}\label{Sec: Stability analysis UNS}

In this section, we derive stability results for our finite element schemes \eqref{eq: weak lagrange fully discrete UNS} and \eqref{eq: weak lagrange fully discrete}. Specifically for the pressure stability we only consider \eqref{eq: weak lagrange fully discrete} assuming \eqref{eq: uniform bound b stability UNS} holds. In this situation we also focus on the case where $\underline{k_\lambda = k_u}$ in which just low regularity conditions for continuous velocity are needed; see \cref{remark: About pressure stability UNS}. In \cref{remark: About the pressure stability for arbitrary} we show how the pressure stability can also be established for $\underline{k_\lambda = k_u-1}$. Nevertheless, let us start with an energy estimate involving the discrete velocity $\uh$.
\begin{lemma}[Velocity Estimate]\label{Lemma: velocity stab estimate UNS}
    For the velocity solution $\uh^k \in \bfV_h$, $k=0,1,...,n$ of \eqref{eq: weak lagrange fully discrete UNS} or \eqref{eq: weak lagrange fully discrete} we have the following energy estimates, which holds for $n \geq 1$
    \begin{equation}
        \begin{aligned}\label{eq: velocity stab estimate UNS}
            \norm{\uh^n}_{L^2(\Gah)}^2 &+ \sum_{k=1}^{n}\norm{\uh^k - \uh^{k-1}}_{L^2(\Gah)}^2 + \Delta t \sum_{k=1}^{n} \norm{\uh^k}_{\ah}^2 \\
            &\qquad \qquad \qquad \leq exp(ct_n)\big(\norm{\uh^0}_{L^2(\Gah)}^2 + \Delta t \sum_{k=0}^{n}\norm{\bff_h^{k}}_{L^2(\Gah)}^2 \big),
        \end{aligned}
    \end{equation}
    where $\uh^0 \in \bfV_h^{div}$ an appropriate approximation $\bfu^0 \in \bfH^1(\Ga)$, i.e. $\uh^0 := \mathcal{R}_h(\bfu^0)$.
\end{lemma}
\begin{proof}
    Testing the first equation of \eqref{eq: weak lagrange fully discrete UNS} or \eqref{eq: weak lagrange fully discrete} with $\vh = \uh^n$ and the second one with $\{\qh,\xi_h\} = \{\ph^n,\lh^n\}$, adding the two statements, remembering the skew-symmetry of $\chtil(\bullet \, ;\bullet,\bullet)$ and using the algebraic identity $(a-b)a = \frac{1}{2}(a^2 - b^2) + \frac{1}{2}(a-b)^2$ we get 
    \begin{equation}
        \begin{aligned}
            \frac{1}{2}(\norm{\uh^n}_{L^2(\Gah)}^2 - \norm{\uh^{n-1}}_{L^2(\Gah)}^2 + \norm{\uh^n - \uh^{n-1}}_{L^2(\Gah)}^2) + \Delta t \norm{\uh^n}_{\ah}^2 \leq \Delta t \norm{\bff_n}_{L^2(\Gah)}\norm{\uh^n}_{L^2(\Ga)}.
        \end{aligned}
    \end{equation}
Now, summing the inequalities for $k=1,2,...,n$, applying  general Young's inequality and a standard kickback argument for small enough time step $\Delta t$ we obtain 
\begin{equation}
    \begin{aligned}
      \norm{\uh^n}_{L^2(\Gah)}^2 + \sum_{k=1}^{n}\norm{\uh^k - \uh^{k-1}}_{L^2(\Gah)}^2 + \Delta t \sum_{k=1}^{n} \norm{\uh^k}_{\ah}^2 &\leq \norm{\uh^0}_{L^2(\Gah)}^2 + \Delta t \varepsilon^{-1}\sum_{k=0}^{n}\norm{\bff_h^k}_{L^2(\Gah)}^2 \\
      & \ + \Delta t \sum_{k=0}^{n-1}\norm{\uh^{n-1}}_{L^2(\Gah)}^2,  
    \end{aligned}
\end{equation}
where a standard discrete Gronwall's inequality gives the desired result.
\end{proof}

We now present the two discrete inf-sup conditions, which play an important role in the pressure stability of our numerical scheme. These were proved in the case of the Lagrangian formulation of the surface Stokes problem in \cite[Lemma 5.6, Lemma 5.8]{elliott2024sfem}. 

\begin{lemma}[$L^2\times L^2$  Discrete Lagrange \textsc{inf-sup} condition]
\label{Lemma: Discrete inf-sup condition Gah Lagrange UNS}
  Assume a quasi-uniform triangulation of $\Gah$, then there exists a constant $c_{b}>0$ independent of the mesh parameter $h$ such that 
    \begin{equation}\label{eq: discrete inf-sup condition Gah Lagrange UNS}
    c_{b}\norm{\{\qh,\xi_h\}}_{L^2(\Gah)}\leq \sup_{\vh \in \bfV_h} \frac{\bhtil(\vh,\{\qh,\xi_h\})}{\norm{\vh}_{\ah}} \ \ \ \forall \{\qh,\xi_h\} \in Q_h \times \Lambda_h.
    \end{equation}
\end{lemma}

    \begin{lemma}[$L^2\times H_h^{-1}$ Discrete Lagrange \textsc{inf-sup} condition]\label{lemma: L^2 H^{-1} discrete inf-sup condition Gah Lagrange UNS}
Assume a quasi-uniform triangulation of $\Gah$, then there exists a constant $c_{b} >0$ independent of the mesh parameter $h$ such that 
\begin{equation}\label{eq: L^2 H^{-1} discrete inf-sup condition Gah Lagrange UNS}
    c_{b}\norm{\{\qh,\xi_h\}}_{L^2(\Gah)\times H_h^{-1}(\Gah)}\leq \sup_{\vh \in \bfV_h} \frac{\bhtil(\vh,\{\qh,\xi_h\})}{\norm{\vh}_{H^1(\Gah)}} \ \ \ \forall \{\qh,\xi_h\} \in Q_h \times \Lambda_h.
\end{equation}
\end{lemma}

\noindent Now, as we mentioned in \Cref{sec: var stability UNS}, let us consider \eqref{eq: weak lagrange fully discrete} assuming \eqref{eq: uniform bound b stability UNS} holds.
To complete the stability results, we need to estimate the pressures $\{\ph,\,\lh\}$. But before we do so, we require to define an inverse operator $\mathcal{A}_h: \, \bfV_h^{div} \to \bfV_h^{div}$. Given $\wh \in \bfV_h^{div}$
we may define the discrete inverse Stokes operator $\mathcal{A}_h$ as the unique solution to the discrete Stokes problem
\begin{equation}
    \begin{aligned}\label{eq: Discrete inverse Stokes UNS}
        \ah(\mathcal{A}_h \wh, \vh)  &= (\wh,\vh)_{L^2(\Gah)}\\
        \bhtil(\mathcal{A}_h \wh,\{\qh,\xi_h\}) &= 0
    \end{aligned}
\end{equation}
for every $\vh \in \bfV_h^{div}$ and $\{\qh,\xi_h\} \in Q_h\times\Lambda_h$. It is clear that the operator is coercive and also bounded, thus, it is well-defined. We can also see by testing appropriately that 
\begin{equation*}
\begin{aligned}
   \norm{\wh}_{\mathcal{A}_h}^2 &=\norm{\mathcal{A}_h \wh}_{\ah}^2  = (\mathcal{A}_h\wh,\wh)_{L^2(\Gah)}, \\
   \norm{\wh}_{\mathcal{A}_h}  &=  \norm{\mathcal{A}_h \wh}_{\ah} \leq \norm{\wh}_{L^2(\Gah)}.
   \end{aligned}
\end{equation*}

With the help of the above inverse Stokes operator we are able to show the auxiliary result:


\begin{lemma}[Auxiliary bound]\label{Lemma: auxilary stab bounds UNS}
Let assumption \eqref{eq: uniform bound b stability UNS} hold. Then for the solution $\uh^k \in \bfV_h$ of \eqref{eq: weak lagrange fully discrete} we have the following energy estimates, which holds for $n \geq 1$
    \begin{equation}\label{eq: auxilary stab bounds UNS}
        \sum_{k=1}^n \Delta t \norm{\frac{\mathcal{A}_h(\uh^k - \uh^{k-1})}{\Delta t}}_{\ah}^2 \leq  C(\bfb_h)\bfA(\uh^0,\bff_h),
    \end{equation}
where $\uh^0 \in \bfV_h^{div}$ an appropriate approximation, $\uh^0 := \mathcal{R}_h(\bfu^0)$, $\bfA(\uh^0,\bff_h)$ the bound in \eqref{eq: velocity stab estimate UNS} and $C(\bullet)$ a constant depending on the bound of $\bfb_h$ \eqref{eq: uniform bound b stability UNS}.
\end{lemma}
\begin{proof}
Testing the first equation of \eqref{eq: weak lagrange fully discrete} with $\mathcal{A}_h(\uh^n - \uh^{n-1}) \in \bfV_h^{div}$, see \eqref{eq: Discrete inverse Stokes UNS} where $\uh^n - \uh^{n-1}\in \bfV_h^{div}$, and considering our choice of the initial value $\uh^0\in \bfV_h^{div}$ we get for $n \geq 1$:
\begin{equation}
    \begin{aligned}\label{eq: Auxiliary stab bound inside 1}
        (\frac{\uh^n-\uh^{n-1}}{\Delta t},\mathcal{A}_h(\uh^n - \uh^{n-1}))_{L^2(\Gah)}+ \ah(\uh^n,\mathcal{A}_h(\uh^n - \uh^{n-1})) &+\chtil(\bfb_h^{n-1};\uh^n,\mathcal{A}_h(\uh^n - \uh^{n-1})) \\
        &= (\bff_h^n,\mathcal{A}_h(\uh^n - \uh^{n-1}))_{L^2(\Gah)}.
    \end{aligned}
\end{equation}
Then we have the following estimates,
\begin{itemize}
     \item By the definition of $\mathcal{A}_h$ \eqref{eq: Discrete inverse Stokes UNS} we see that $$(\frac{\uh^n-\uh^{n-1}}{\Delta t},\mathcal{A}_h(\uh^n - \uh^{n-1}))_{L^2(\Gah)} = \frac{1}{\Delta t } \norm{\mathcal{A}_h(\uh^n - \uh^{n-1})}_{\ah}^2.$$
      \item Young's inequality provides us with the estimate
      \begin{equation*}
          \begin{aligned}
              \ah(\uh^n,\mathcal{A}_h(\uh^n - \uh^{n-1})) &\leq \norm{\uh^n}_{\ah}\norm{\mathcal{A}_h(\uh^n - \uh^{n-1})}_{\ah}\\
              &\leq c\Delta t \norm{\uh^n}_{\ah}^2 + \frac{1}{6\Delta t }\norm{\mathcal{A}_h(\uh^n - \uh^{n-1})}_{\ah}^2.
          \end{aligned}
      \end{equation*}
\end{itemize}
\begin{itemize}
      \item Recalling the assumption \eqref{eq: uniform bound b stability UNS}, the bound \eqref{ch boundedness UNS}, and using Young's inequality once again
      \begin{equation*}
          \begin{aligned}
              \chtil(\bfb_h^{n-1};\uh^n,\mathcal{A}_h(\uh^n - \uh^{n-1})) &\leq \norm{\bfb_h^{n-1}}_{\ah}\norm{\uh^n}_{\ah}\norm{\mathcal{A}_h(\uh^n - \uh^{n-1})}_{\ah}^2\\
              &\leq c(\bfb_h)\Delta t \norm{\uh^n}_{\ah}^2 +  \frac{1}{6\Delta t }\norm{\mathcal{A}_h(\uh^n - \uh^{n-1})}_{\ah}^2.
          \end{aligned}
      \end{equation*}
      \item Similar calculations also give
      \begin{equation*}
          (\bff_h^n,\mathcal{A}_h(\uh^n - \uh^{n-1}))_{L^2(\Gah)} \leq c\Delta t \norm{\bff_h^n}_{L^2(\Gah)}^2 +\frac{1}{6\Delta t }\norm{\mathcal{A}_h(\uh^n - \uh^{n-1})}_{\ah}^2.
      \end{equation*}
\end{itemize}
Combining the above results to \eqref{eq: Auxiliary stab bound inside 1} yields
\begin{equation*}
    \begin{aligned}
        \frac{1}{2\Delta t }\norm{\mathcal{A}_h(\uh^n - \uh^{n-1})}_{\ah}^2 \leq c\Delta t \norm{\bff_h^n}_{L^2(\Gah)}^2 + c(\bfb_h)\Delta t \norm{\uh^n}_{\ah}^2,
    \end{aligned}
\end{equation*}
where summing the inequality for $n=1,2,...,\, k$ and \Cref{Lemma: velocity stab estimate UNS} completes the proof.
\end{proof}

Before proving the pressure stability we need to bound some variant of the dual energy norm. Estimates for similar quantities have been proven in \cite{FrutosGradDivOseen2016} for the Oseen problem, \cite{AyusoPostNS2005,JohnBook2016} for the Navier-Stokes problem in the $H^{-1}$-norm on fixed domains without variational crimes involving the domain approximation. 

\begin{lemma}\label{lemma: dual estimate UNS}
Assume that $\underline{k_{\lambda} = k_u}$, then for a vector-valued function $\wh \in \bfV_h^{div}$ the following estimate holds for constant $c>0$ independent of time $t$, mesh-parameter $h$
\begin{equation}
    \begin{aligned}\label{eq: dual estimate UNS}
        \sup_{\vh \in \bfV_h}\frac{(\wh,\vh)_{L^2(\Gah)}}{\norm{\vh}_{H^1(\Gah)}} \leq  \sup_{\vh \in \bfV_h}\frac{(\wh,\vh)_{L^2(\Gah)}}{\norm{\vh}_{\ah}}\leq ch\norm{\wh}_{L^2(\Gah)} + c\norm{\mathcal{A}_h\wh}_{\ah}.
    \end{aligned}
\end{equation}
\end{lemma}
\noindent Proof of \cref{lemma: dual estimate UNS} is given in \Cref{appendix: proof of dual estimate UNS}. We can now finalize the stability results concerning the two pressures. We prove $L^2_{L^2}\times L^2_{L^2}$ pressure bounds, from which we can easily see that the $L^2_{L^2}\times L^2_{H_h^{-1}}$ stability estimate also follows.

\begin{lemma}[$L^2_{L^2}\times L^2_{L^2}$ Pressure Bounds]\label{Lemma: Pressure stab Estimate UNS}
      Let \Cref{Lemma: Discrete inf-sup condition Gah Lagrange UNS} and the assumptions in \Cref{Lemma: auxilary stab bounds UNS} and \Cref{lemma: dual estimate UNS} hold.  Then, for the pressure solutions $\{\ph^k,\lh^k\} \in Q_h\times L_h$, $k=1,...,n$ of \eqref{eq: weak lagrange fully discrete} the following stability estimate holds true
     \begin{equation}\label{eq: Pressure stab Estimate UNS}
            \Delta t \sum_{k=1}^{n} \norm{\{\ph^k,\lh^k\}}_{L^2(\Gah)}^2 \leq C(\bfb_h)\bfA(\uh^0,\bff_h),
     \end{equation}
where $\bfA = \bfA(\uh^0,\bff_h)$ the bound in \eqref{eq: velocity stab estimate UNS} and $C(\bullet)$ a constant depending on the bound of $\bfb_h$ \eqref{eq: uniform bound b stability UNS}.
 \end{lemma}
\begin{proof}
 Recall the discrete $L^2\times L^2$ \textsc{inf-sup} condition in \Cref{Lemma: Discrete inf-sup condition Gah Lagrange UNS}, we have 
    \begin{equation}
        \begin{aligned}\label{eq: Pressure stab Estimate UNS inside 2 inf-sup}
            \norm{\{\ph^n,\lh^n\}}_{L^2(\Gah)} \leq \sup_{\vh^n \in \bfV_h} \frac{\bhtil(\vh^n,\{\ph^n,\lh^n\})}{\norm{\vh^n}_{\ah}},
        \end{aligned}
    \end{equation}
where due to \eqref{eq: weak lagrange fully discrete}, and the bounds of our bilinear and trilinear forms in \Cref{Lemma: discrete bounds and coercivity results} we get
\begin{equation*}
\begin{aligned}
    \bhtil(\vh^n,\{\ph^n,\lh^n\}) &= -(\frac{\uh^n-\uh^{n-1}}{\Delta t},\vh^n)_{L^2(\Gah)} - \ah(\uh^n,\vh^n) -\chtil(\bfb_h^{n-1},\uh^n,\vh^n) + (\bff_h^n,\vh^n) \\
    &\leq |(\frac{\uh^n-\uh^{n-1}}{\Delta t},\vh^n)_{L^2(\Gah)}| + c\norm{\uh^n}_{\ah}\norm{\vh^n}_{\ah} + \norm{\bfb_h^{n-1}}_{\ah}\norm{\uh^n}_{\ah}\norm{\vh^n}_{\ah} \\
    &\  + \norm{\bff_h^{n}}_{L^2(\Gah)}\norm{\vh^n}_{L^2(\Gah)}.
    \end{aligned}
\end{equation*}   
Therefore, going back to \eqref{eq: Pressure stab Estimate UNS inside 2 inf-sup},  by \Cref{lemma: dual estimate UNS} we have
\begin{equation}
        \begin{aligned}\label{eq: Pressure stab Estimate UNS inside inf-sup}
            \norm{\{\ph^n,\lh^n\}}_{L^2(\Gah)} &\leq ch\norm{\frac{\uh^n-\uh^{n-1}}{\Delta t}}_{L^2(\Gah)} + c\norm{\frac{\mathcal{A}_h(\uh^n-\uh^{n-1})}{\Delta t}}_{\ah} + c(1 + \norm{\bfb_h^{n-1}}_{\ah})\norm{\uh^n}_{\ah} + \norm{\bff_h^{n}}_{L^2(\Gah)} \\
        &\leq c\norm{\frac{\mathcal{A}_h(\uh^n-\uh^{n-1})}{\Delta t}}_{\ah} + c\norm{\uh^n}_{\ah}+ \norm{\bfb_h^{n-1}}_{\ah}\norm{\uh^n}_{\ah} + \norm{\bff_h^{n}}_{L^2(\Gah)},
        \end{aligned}
    \end{equation}
where we used the definition of the discrete inverse Stokes operator $\mathcal{A}_h$ \eqref{eq: Discrete inverse Stokes UNS} to see that 
\begin{equation*}
    \begin{aligned}
        \norm{\frac{\uh^n-\uh^{n-1}}{\Delta t}}_{L^2(\Gah)}^2 &\leq \norm{\frac{\mathcal{A}_h(\uh^n-\uh^{n-1})}{\Delta t}}_{\ah} \norm{\frac{\uh^n-\uh^{n-1}}{\Delta t}}_{\ah}\\
        &\leq ch^{-1}\norm{\frac{\mathcal{A}_h(\uh^n-\uh^{n-1})}{\Delta t}}_{\ah} \norm{\frac{\uh^n-\uh^{n-1}}{\Delta t}}_{L^2(\Gah)},
    \end{aligned}
\end{equation*}
hence $\norm{\frac{\uh^n-\uh^{n-1}}{\Delta t}}_{L^2(\Gah)} \leq ch^{-1}\norm{\frac{\mathcal{A}_h(\uh^n-\uh^{n-1})}{\Delta t}}_{\ah}$. Squaring and multiplying \eqref{eq: Pressure stab Estimate UNS inside inf-sup} by $\Delta t$  we obtain 
\begin{equation*}
    \begin{aligned}
        \Delta t \norm{\{\ph^k,\lh^k\}}_{L^2(\Gah)}^2 \leq c\Delta t \norm{\frac{\mathcal{A}_h(\uh^n-\uh^{n-1})}{\Delta t}}_{\ah}^2 + \Delta t(1 + \norm{\bfb_h^{n-1}}_{\ah}^2)\norm{\uh^n}_{\ah}^2 + \Delta t \norm{\bff_h^{n}}_{L^2(\Gah)}^2.
    \end{aligned}
\end{equation*}
Finally, summation for $k=1,...,n$ and using the bounds \eqref{eq: velocity stab estimate UNS}, \eqref{eq: auxilary stab bounds UNS} along with the assumption \eqref{eq: uniform bound b stability UNS} regarding the uniform bound of $\bfb_h^{n-1}$ results in our desired estimate.
\end{proof}

\begin{remark}[About $L^2_{L^2}\times L^2_{L^2}$ pressure stability for arbitrary $k_{\lambda}$]\label{remark: About the pressure stability for arbitrary}
    Notice in \Cref{Lemma: Pressure stab Estimate UNS} we had to assume $k_\lambda=k_u$, since \cref{lemma: dual estimate UNS} does not hold 
    for $k_\lambda=k_u-1$; see \cref{remark: Lower Regularity convergence estimate UNS} for further details. To prove pressure stability in the $k_\lambda=k_u-1$ case we need further assumptions for  $\bfb_h$, mainly that
    \begin{equation}\label{eq: b extra stability UNS}
\sup_{n} \norm{\bfb_h^n}_{L^{\infty}(\Gah)} \leq C_b.
\end{equation}\vspace{-4mm}

\noindent Then we can establish $L^2_{L^2}\times L^2_{L^2}$ pressure stability bounds, by testing the main equation with $\vh = \frac{\uh^n-\uh^{n-1}}{\Delta t}$ and finding a stability estimate for the stronger expression
\begin{equation*}
    \Delta t\sum_{n=1}^n\norm{\frac{\uh^n-\uh^{n-1}}{\Delta t}}_{L^2(\Gah)}^2,
\end{equation*}
instead. We omit further details, as the calculations are straightforward, but the procedure is similar to how we handle the error estimates in \cref{sec: Pressure a-priori estimates for kl= ku-1 UNS} for $k_\lambda = k_u-1$. Lastly, we notice that this new assumption \eqref{eq: b extra stability UNS} holds only if further regularity assumptions are imposed; see \cref{remark: About the pressure stability for arbitrary convergence UNS}.
\end{remark}

\section{Error Analysis}\label{sec: error analysis UNS}
For the error analysis, we consider the linearized \emph{unsteady fully discrete Lagrangian surface Navier-Stokes} finite element method \eqref{eq: weak lagrange fully discrete UNS}, that is, we choose $\bfb_n^{n-1} = \uh^{n-1}$ for \eqref{eq: weak lagrange fully discrete}. Let us introduce the problem \eqref{eq: weak lagrange fully discrete UNS} once more for clarity reasons:
Find $\uh^n \in \bfV_h$, $\{\ph^n,\lh^n\} \in Q_h\times \Lambda_h$ such that,
\begin{align}
\begin{cases}\tag{UWL}
        (\frac{\uh^n-\uh^{n-1}}{\Delta t},\vh)_{L^2(\Gah)}+ \ah(\uh^n,\vh) +\chtil(\uh^{n-1};\uh^n,\vh) + \bhtil(\vh,\{\ph^n,\lh^n\}) \!\!&= (\bff_h^n,\vh), \\
    \phantom{aaaaaaaaaaaaaaaaaaaaaaaaaaaaaaaaaaaaaaaaaaaaaa}\bhtil(\uh^n,\{\qh,\xi_h\})\!\!&=0 ,
    \end{cases}
\end{align}
for all $(\vh,\{\qh,\xi_h\}) \in \bfV_h \times \{Q_h \times \Lambda_h\}$, with the approximation of the source term $\bff_h^{\ell} = \bff$, and where the initial condition $\uh^0$ is an appropriate approximation of $\bfu^0=\bfu(t_0)=\bfu(0)$. In our case, as we mentioned before, we suppose that $\uh^0\in \bfV_h^{div}$. This initial condition means that even at time-step $t_0=0$ our solution is still \emph{discrete weakly tangential divergence-free}. In \cref{sec: velocity a-priori estimates UNS,sec: pressure a-priori kl=ku UNS} we consider $\uh^0 = \mathcal{R}_h\bfu^0$, the modified Ritz-Stokes projection \eqref{eq: surface Ritz-Stokes projection UNS}. In \cref{sec: Pressure a-priori estimates for kl= ku-1 UNS} we choose the initial condition as $\uh^0 = \mathcal{R}_h^b\bfu^0 = \mathcal{R}_h(\bfu^0,\{p^0,\lambda^0\})$, for convenience reasons.
\begin{remark}\label{remark: about initial conditions UNS}
    Different approximations could have been selected for the initial condition, yet each choice would present different difficulties. For example
    \begin{itemize}
        \item[-] $\uh^0 = \Ih(\bfu^0)$ : In this case $\uh^0$ is not weakly divergence-free at time-step 0, then we can only have condition stability, and an inverse type CFL condition is necessary, see also \cite{burman2009galerkin}.
        \item[-] $\uh^0 = \mathcal{R}_h(\bfu^0,\{p^0,\lambda^0\})$ : The main issue is that in practice the initial pressures $\{p^0,\lambda^0\}$ will be unknown.
    \end{itemize}
\end{remark}
For convenience we, once again, omit writing the inverse lift extension $(\cdot)^{-\ell}$ since it should be clear from the context whether a quantity is defined on $\Ga$ or $\Gah$, e.g. for functions $\bfu,\, p,\, \lambda$. 

The error analysis follows standard arguments. We define the errors and decompose them into an interpolation and a discrete error in the appropriate finite element spaces. For that we consider the modified Ritz-Stokes projection \eqref{eq: surface Ritz-Stokes projection UNS} for the velocity $\bfu$ and  normal Lagrange interpolant for the two pressures $\{p, \lambda\}$:
\begin{align}\label{eq: decomposition error 1 Ritz UNS}
    \eu^{n} = \bfu^n - \uh^n &= \underbrace{(\bfu^n- \mathcal{R}_h\bfu^n) }_{\text{Interpolation error}} + \underbrace{(\mathcal{R}_h\bfu^n - \uh^n)}_{\text{discrete remainder}} := \rho_{\bfu}^n + \theta_{\bfu}^n , \\
        \label{eq: decomposition error 2 lagrange UNS} 
         \ep^n = p^n - \ph &= \underbrace{(p^n- \Ih p^n)}_{\text{Interpolation error}} + \underbrace{(\Ih p^n - \ph)}_{\text{discrete remainder}}:= \rho_{p}^n + \theta_{p}^n,
        \\
        \label{eq: decomposition error 3 lagrange UNS} \el^n =  \lambda^n - \lh &= \underbrace{(\lambda^n- \Ih \lambda^n )}_{\text{Interpolation error}} + \underbrace{(\Ih \lambda^n - \lh)}_{\text{discrete remainder}}:= \rho_{\lambda}^n + \theta_{\lambda}^n.      
\end{align}
The bounds of the modified Ritz-Stokes interpolation error have already been proved; see  \cref{lemma: Error Bounds Ritz-Stokes UNS}. 
Note also that now $\thbfu^n \in \bfV_h^{div}$ for all $n=0,1,...,N$, due to the choice of interpolant and the choice of the initial condition. The interpolation errors for the pressures are well-known
\begin{equation}
    \begin{aligned}\label{eq: interpolation errors pressures UNS}
        \norm{p^n - \Ih p^n}_{L^2(\Gah)}  &\leq h^{k_{pr}+1}\norm{p^n}_{H^{k_{pr}+1}},\\
        \norm{\lambda^n - \Ih \lambda^n}_{L^2(\Gah)}  &\leq h^{k_{\lambda}+1}\norm{\lambda^n}_{H^{k_{\lambda}+1}}.
    \end{aligned}
\end{equation}

\subsection{Consistency and velocity a-priori estimates}\label{sec: velocity a-priori estimates UNS}
\begin{assumption}[Regularity assumptions I]\label{assumption: Regularity assumptions for velocity estimate}
We assume the following regularity for the solution \eqref{eq: unstead NV Lagrange}   
\begin{equation*}
    \begin{aligned}
    &\{p,\lambda\} \in L^{\infty}([0,T];H^{k_{pr}+1}(\Ga))\times L^{\infty}([0,T];H^{k_{\lambda}+1}(\Ga)),\\
        &\bfu \in H^2([0,T];L^2(\Ga))\cap W^{1,\infty}([0,T];L^2(\Ga))\cap H^1([0,T];H^{k_u+1}(\Ga))\cap L^{\infty}([0,T];H^{k_u+1}(\Ga)).
    \end{aligned}
\end{equation*}
\end{assumption}

\noindent We remind the weak formulation \eqref{weak lagrange hom NV}:
\begin{align}
\begin{cases}\label{eq: weak lagrange fully discrete 2 UNS} 
        (\partial_{t}\bfu,\vhl)_{L^2(\Ga)} + a(\bfu,\vhl) \ + c(\bfu;\bfu,\vhl)\ + \!\!\!\!&b^L(\vhl,\{p,\lambda\}) = (\bff,\vhl)_{L^2(\Ga)} \ \ \ \ \ \text{for all } \vhl\in V_h^{\ell},\\
        &b^L(\bfu,\{\qhl,\xi_h^{\ell}\})=0 \ \ \  \text{ for all } \{\qhl,\xi_h^{\ell}\}\in Q_h^{\ell} \times \Lambda_h^{\ell},
    \end{cases}
\end{align}
for a.e. $t\in [0,T]$. For the error analysis we consider the pair $(\bfu,\{p,\lambda\})$ as the strong solution of \eqref{eq: unstead NV Lagrange}  therefore we may rewrite the inertia from now on following \eqref{eq: continuous c error formula init UNS} as
\begin{equation}
    \begin{aligned}\label{eq: continuous c error formula UNS}
        c(\bfu;\bfu,\vhl) = \frac{1}{2}\Big(\int_{\Ga}((\bfu\cdot\nbgcov)\bfu) \cdot\vhl \, \ds - \int_{\Ga}((\bfu\cdot\nbgcov)\bfPg\vhl )\cdot\bfu \, \ds\Big).
    \end{aligned}
\end{equation}

We start by finding an error equation involving the discrete remainder $\thbfu^n$ with respect to some consistency, interpolations  and inertia errors. As mentioned before, by the Ritz-Stokes projection \eqref{eq: surface Ritz-Stokes projection UNS} $\thbfu^n\in \bfV_h^{div}$ and therefore we are able to \emph{decouple} the pressure discrete remainders from the error equations, which is crucial to find error bounds just for $\thbfu^n$ independent of $\theta_{p}^n,\, \theta_{\lambda}^n$. 


\begin{lemma}\label{lemma: error equation UNS}
   Let $(\vh,\{\qh,\xi_h\}) \in \bfV_h\times (Q_h\times \Lambda_h)$ be test functions, then for $\thbfu,\, \theta_{p},\theta_{\lambda}$ as in  \eqref{eq: decomposition error 1 Ritz UNS}, \eqref{eq: decomposition error 2 lagrange UNS}, \eqref{eq: decomposition error 3 lagrange UNS} the following error equation holds true for $n \geq 1$
   \begin{align}
\begin{cases}\label{eq: error equation UNS}
       (\frac{\thbfu^n -\thbfu^{n-1}}{\Delta t},\vh )_{L^2(\Gah)} + \ah(\theta_{\bfu}^n,\vh)\,+\!\!\!\!\!&\bhtil(\vh,\{\theta_{p}^n,\theta_{\lambda}^n\}) = \sum_{i=1}^5 \textsc{Err}_{i}^{C}(\vh) +  \sum_{i=1}^2 \textsc{Err}_{i}^{I}(\vh) + \mathcal{C}(\vh),\\
        & \bhtil(\theta_{\bfu}^n,\{\qh,\xi_h\})= 0,
    \end{cases}
\end{align}
for all $\vh \in \bfV_h$ and $\{\qh,\xi_h\}\in Q_h\times \Lambda_h$,
with consistency errors:

\begin{minipage}[t]{9cm}
    \begin{itemize}
        \item[\textbullet\hspace{9mm}] \hspace{-9mm}$\textsc{Err}_1^{C}(\vh) := (\frac{\bfu^n - \bfu^{n-1}}{\Delta t},\vh )_{L^2(\Gah)} - (\partial_t\bfu^n,\vhl)_{L^2(\Ga)}$
        \item[\textbullet\hspace{9mm}] \hspace{-9mm}$ \textsc{Err}_2^{C}(\vh) := \bhtil(\vh,\{p^n,\lambda^n\}) -b^L(\vhl,\{p^n,\lambda^n\})$
        \item[\textbullet\hspace{9mm}] \hspace{-9mm}$ \textsc{Err}_3^{C}(\vh) := \chtil(\bfu^{n-1};\bfu^n,\vh)  - c(\bfu^n;\bfu^n,\vhl)$
    \end{itemize}
\end{minipage}
\begin{minipage}[t]{8cm}
    \begin{itemize}
        \item[\textbullet\hspace{12mm}] \hspace{-14mm} $ \textsc{Err}_4^{C}(\vh) := (\bff^n,\vhl)_{L^2(\Ga)}-(\bff_h^n,\vh)_{L^2(\Gah)}$
        \item[\textbullet\hspace{12mm}] \hspace{-14mm} $ \textsc{Err}_5^{C}(\vh) :=\ah(\mathcal{R}_h\bfu^n,\vh) - a(\bfu^n,\vhl)$
    \end{itemize}
\end{minipage}

\vspace{2mm}
\noindent and interpolations errors:
\vspace{2mm}

\begin{minipage}[t]{9cm}
    \begin{itemize}
        \item[\textbullet\hspace{9mm}] \hspace{-9mm}$ \textsc{Err}_1^{I}(\vh) := - (\frac{\rho_{\bfu}^n - \rho_{\bfu}^{n-1}}{\Delta t},\vh )_{L^2(\Gah)}$
        \item[\textbullet\hspace{9mm}] \hspace{-9mm}$ \textsc{Err}_2^{I}(\vh) := - \bhtil(\vh,\{\rho_{p}^n,\rho_{\lambda}^n\}).$
    \end{itemize}
\end{minipage}
\begin{minipage}[t]{8cm}
    \begin{itemize}
        \item[\textbullet\hspace{19mm}] \hspace{-21mm} $ \mathcal{C}(\vh) := - \chtil(\eu^{n-1};\bfu^n,\vh) - \chtil(\uh^{n-1};\eu^{n},\vh)$
    \end{itemize}
\end{minipage}

\end{lemma}

\begin{proof}
    Recall the decompositions \eqref{eq: decomposition error 1 Ritz UNS}, \eqref{eq: decomposition error 2 lagrange UNS}, \eqref{eq: decomposition error 3 lagrange UNS}. Then, using our F.E. Scheme \eqref{eq: weak lagrange fully discrete UNS}  the following calculations hold
\begin{equation*}
    \begin{aligned}
        &D_n:=(\frac{\thbfu^n -\thbfu^{n-1}}{\Delta t},\vh )_{L^2(\Gah)} + \ah(\theta_{\bfu}^n,\vh) + \bhtil(\vh,\{\theta_{p}^n,\theta_{\lambda}^n\})  = -(\bff_h^n,\vh)_{L^2(\Gah)}\\ 
        &\ +  (\frac{\mathcal{R}_h\bfu^n - \mathcal{R}_h\bfu^{n-1}}{\Delta t},\vh )_{L^2(\Gah)} 
         + \ah(\mathcal{R}_h\bfu^n,\vh) + \bhtil(\vh,\{\Ih p^n,\Ih \lambda^n\}) + \chtil(\uh^{n-1};\uh^n,\vh) .
    \end{aligned}
\end{equation*}
Recalling the continuous equation \eqref{eq: weak lagrange fully discrete 2 UNS} at time $t=t^n$ we get
\begin{equation*}
    \begin{aligned}
        D^n &= -(\bff_h^n,\vh)_{L^2(\Gah)} + \chtil(\uh^{n-1};\uh^n,\vh) +  (\frac{\mathcal{R}_h\bfu^n - \mathcal{R}_h\bfu^{n-1}}{\Delta t},\vh )_{L^2(\Gah)} + a(\bfu^n,\vhl)\\
        &\ + \ah(\mathcal{R}_h\bfu^n,\vh) - a(\bfu^n,\vhl) + \bhtil(\vh,\{\Ih p^n,\Ih \lambda^n\}) 
         \\&= (\bff^n,\vhl)_{L^2(\Ga)}-(\bff_h^n,\vh)_{L^2(\Gah)} + (\frac{\mathcal{R}_h\bfu^n - \mathcal{R}_h\bfu^{n-1}}{\Delta t},\vh )_{L^2(\Gah)}- (\partial_t\bfu^n,\vhl)_{L^2(\Ga)}\\
        &\ + \ah(\mathcal{R}_h\bfu^n,\vh) - a(\bfu^n,\vhl) + \bhtil(\vh,\{\Ih p^n,\Ih \lambda^n\}) -b^L(\vhl,\{p^n,\lambda^n\})\\
        &\ + \chtil(\uh^{n-1};\uh^n,\vh) - c(\bfu^n;\bfu^n,\vhl) \\
        &= - (\frac{\rho_{\bfu}^n - \rho_{\bfu}^{n-1}}{\Delta t},\vh )_{L^2(\Gah)} -  \bhtil(\vh,\{\rho_{p}^n,\rho_{\lambda}^n\}) - \chtil(\eu^{n-1};\bfu^n,\vh) - \chtil(\uh^{n-1};\eu^{n},\vh) \\
        &\ + (\frac{\bfu^n - \bfu^{n-1}}{\Delta t},\vh )_{L^2(\Gah)} - (\partial_t\bfu^n,\vhl)_{L^2(\Ga)} + \bhtil(\vh,\{p^n,\lambda^n\}) -b^L(\vhl,\{p^n,\lambda^n\}) \\
        &\ + \chtil(\bfu^{n-1};\bfu^n,\vh)  - c(\bfu^n;\bfu^n,\vhl) + (\bff^n,\vhl)_{L^2(\Ga)}-(\bff_h^n,\vh)_{L^2(\Gah)} +\ah(\mathcal{R}_h\bfu^n,\vh) - a(\bfu^n,\vhl)\\
         & =: \text{Err}_1^{I}(\vh) + \text{Err}_2^{I}(\vh) + \mathcal{C}(\vh)\\
        &\ +  \text{Err}_1^{C}(\vh) +  \text{Err}_2^{C}(\vh) +  \text{Err}_3^{C}(\vh) + \text{Err}_4^{C}(\vh) + \text{Err}_5^{C}(\vh),
    \end{aligned}
\end{equation*}
where in the second to last equality we used the fact that
\begin{equation}\label{eq: discrete inertia error UNS}
    \chtil(\uh^{n-1};\uh^n;\vh) - \chtil(\bfu^{n-1};\bfu^n;\vh) = - \chtil(\eu^{n-1};\bfu^n,\vh) - \chtil(\uh^{n-1};\eu^{n},\vh).
\end{equation}
\noindent Finally, by the modified Ritz-Stokes projection \eqref{eq: surface Ritz-Stokes projection UNS}, we can clearly see that $\thbfu^n \in \bfV_h^{div}$ for all $n\geq 0$, thus the second equation in \eqref{eq: error equation UNS} follows. 
\end{proof}

Before deriving the error estimates, we first want to  establish appropriate bounds for the consistency and interpolation errors, appearing above in \cref{lemma: error equation UNS}. 
\begin{lemma}[Interpolation errors]\label{lemma: Interpolation errors UNS}
 Let $ \partial_t\bfu \in L^{2}(I_n;(H^{k_u+1}(\Ga))^3)$. The following interpolation bounds hold \vspace{-2.5mm}
 \begin{align}\label{eq: Interpolation errors UNS}
        |\textsc{Err}^{I}_1(\vh)| &\leq   \frac{c}{\sqrt{\Delta t}}\norm{\partial_t \rho_\bfu(t,\cdot)}_{L^2(I_n;L^2(\Gah))}\norm{\vh}_{L^2(\Gah)},\\ 
        \label{eq: Interpolation error 2 UNS}
        |\textsc{Err}^{I}_2(\vh)| &\leq c(h^{k_{pr}+1} + h^{k_{\lambda}+1})( \norm{p^n}_{H^{k_{pr}+1}(\Ga)} + \norm{\lambda^n}_{H^{k_{\lambda}+1}(\Ga)})\norm{\vh}_{\ah},
\end{align}
with constant $c>0$ independent of $h$. If furthermore $\underline{k_\lambda = k_u}$ and $\vh \in \bfV_h^{div}$, the following  holds
\begin{align}
    \label{eq: Interpolation errors improved UNS}
        |\textsc{Err}^{I}_1(\vh)| &\leq   \frac{ch^{\widehat{r}_u+1}}{\sqrt{\Delta t}}\norm{\partial_t \bfu}_{L^2(I_n;H^{k_u+1}(\Ga))}\norm{\vh}_{L^2(\Gah)},
\end{align}
\vspace{-2.5mm}

\noindent where $\widehat{r}_u = min\{k_u,k_g\}$ and constant $c>0$ independent of $h$.
\end{lemma}
\begin{proof}
Let us start with the interpolation error $\textsc{Err}^{I}_1(\cdot)$. We notice that the first term may be rewritten as
\begin{equation*}
    \begin{aligned}
        (\frac{\rho_{\bfu}^n-\rho_{\bfu}^{n-1}}{\Delta t},\vh)_{L^2(\Gah)} &\leq \frac{1}{\Delta t} \norm{\rho_{\bfu}^n-\rho_{\bfu}^{n-1}}_{L^2(\Gah)} \norm{\vh}_{L^2(\Gah)}\\
        &= \frac{1}{\Delta t} \norm{\int_{t_{n-1}}^{t_n} \partial_t\rho_{\bfu}(t,\cdot) \, dt}_{L^2(\Gah)}\norm{\vh}_{L^2(\Gah)}.
    \end{aligned}
\end{equation*}
Then, using the fact that
\begin{equation*}
    \int_{t_{n-1}}^{t_n} \partial_t\rho_{\bfu}(t,\cdot) dt\leq \sqrt{\Delta t}\norm{\partial_t\rho_{\bfu}}_{L^2(I_n)},
\end{equation*}
coupled with the estimate of the time derivative of the modified Ritz-Stokes interpolation operator \eqref{eq: Error Bounds Ritz-Stokes UNS}, we obtain our desired result, i.e.
\begin{equation*}
    \begin{aligned}
        (\frac{\rho_{\bfu}^n-\rho_{\bfu}^{n-1}}{\Delta t},\vh)_{L^2(\Gah)}\leq \frac{c}{\sqrt{\Delta t}}\norm{\partial_t \rho_\bfu(t,\cdot)}_{L^2(I_n;L^2(\Gah))}\norm{\vh}_{L^2(\Gah)}.
    \end{aligned}
\end{equation*}

Next, \eqref{eq: Interpolation error 2 UNS} can be calculated with the help of the continuity bound \eqref{bhtilde boundedness UNS} and interpolation errors \eqref{eq: interpolation errors pressures UNS}.

\noindent Now, regarding \eqref{eq: Interpolation errors improved UNS} we follow the same footsteps as above but this time we split it as followed
\begin{equation*}
    \begin{aligned}
        (\frac{\rho_{\bfu}^n-\rho_{\bfu}^{n-1}}{\Delta t},\vh)_{L^2(\Gah)} &= (\frac{\bfP(\rho_{\bfu}^n-\rho_{\bfu}^{n-1})}{\Delta t},\bfP\vh)_{L^2(\Gah)} + (\frac{(\rho_{\bfu}^n-\rho_{\bfu}^{n-1})\cdot\bfn}{\Delta t},\vh\cdot(\bfn-\nh))_{L^2(\Gah)} \\
        &\ + (\frac{(\rho_{\bfu}^n-\rho_{\bfu}^{n-1})\cdot\bfn}{\Delta t},\vh\cdot\nh)_{L^2(\Gah)}\\
        &\leq  \frac{c}{\sqrt{\Delta t}}\big(\norm{\partial_t \bfP\rho_\bfu(t,\cdot)}_{L^2(I_n;L^2(\Gah))} + ch\norm{\partial_t\rho_\bfu(t,\cdot)}_{L^2(I_n;L^2(\Gah))}\big)\norm{\vh}_{L^2(\Gah)},
    \end{aligned}
\end{equation*}
where we applied  the geometric error \eqref{eq: geometric errors 2 UNS}, and  the normal bound \eqref{eq: divfree L2 norm kl=ku UNS}, since $\vh \in \bfV_h^{div}$. Recalling the modified Ritz-Stokes estimates \eqref{eq: Error Bounds Ritz-Stokes L2 improved UNS} and \eqref{eq: Error Bounds Ritz-Stokes improved UNS} our desired result follows.
\end{proof}

\begin{lemma}[Inertia error]\label{lemma: Trilinear Errors UNS}
Given that $\bfu \in L^{\infty}(I_n;(H^{k_u+1}(\Ga))^3),$ the following bounds are satisfied, for all $\vh \in \bfV_h$ 
\begin{equation}
    \begin{aligned}\label{eq: Trilinear Errors general 1 UNS}
        |\mathcal{C}(\vh)| &\leq c\big(\norm{\eu^{n-1}}_{L^2(\Gah)}^{1/2}\norm{\eu^{n-1}}_{\ah}^{1/2} +  c\norm{\uh^{n-1}}_{\ah}^{1/2}\norm{\eu^n}_{\ah}\\
        &\qquad \qquad \qquad \  + \norm{\uh^{n-1}}_{\ah}^{1/2}\norm{\eu^n}_{L^2(\Gah)}^{1/2}\norm{\eu^n}_{\ah}^{1/2}\big) \norm{\vh}_{\ah},
    \end{aligned}
\end{equation}

\begin{equation}
    \begin{aligned}\label{eq: Trilinear Errors general 2 UNS}
        |\mathcal{C}(\vh)| &\leq c\big(\norm{\uh^{n-1}}_{\ah}^{1/2}\norm{\rho_{\bfu}^n}_{\ah} + \norm{\theta_{\bfu}^{n-1}}_{L^2(\Gah)}^{1/2}\norm{\theta_{\bfu}^{n-1}}_{\ah}^{1/2}\big)\norm{\vh}_{\ah}\\
        & \qquad \qquad \qquad \ \ \, + \norm{\uh^{n-1}}_{\ah}^{1/2}\norm{\thbfu^n}_{\ah}\norm{\vh}_{\ah}.
    \end{aligned}
\end{equation}

In the case where $\vh = \thbfu^n$, considering the skew-symmetric nature of the inertia term, we arrive at the following
\begin{equation}
    \begin{aligned}\label{eq: Trilinear Errors general UNS}
     |\mathcal{C}(\thbfu^n)| &\leq c(\norm{\rho_{\bfu}^{n-1}}_{L^2(\Gah)}^{1/2}\norm{\rho_{\bfu}^{n-1}}_{\ah}^{1/2} + \norm{\theta_{\bfu}^{n-1}}_{L^2(\Gah)}^{1/2}\norm{\theta_{\bfu}^{n-1}}_{\ah}^{1/2} + \norm{\uh^{n-1}}_{\ah}^{1/2}\norm{\rho_{\bfu}^n}_{\ah})\norm{\thbfu^n}_{\ah}.
    \end{aligned}
\end{equation}
\end{lemma}
\begin{proof}
Let us start with \eqref{eq: Trilinear Errors general 1 UNS}. Remember that $\mathcal{C}(\vh) = - \chtil(\eu^{n-1};\bfu^n,\vh) - \chtil(\uh^{n-1};\eu^{n},\vh)$. Considering \eqref{eq: skew-symmetrized ch real Ph}, the fact that $\norm{\bfH_h}_{L^{\infty}}\leq c$ and the continuity bound \eqref{ch boundedness UNS} of the trilinear form,  we see that
\begin{equation*}
    \begin{aligned}
       \chtil(\eu^{n-1};\bfu^n,\vh) \leq  \norm{\eu^{n-1}}_{L^2(\Gah)}^{1/2}\norm{\eu^{n-1}}_{\ah}^{1/2}\norm{\bfu^n}_{H^1(\Gah)}\norm{\vh}_{\ah}.
    \end{aligned}
\end{equation*}
Similarly, with the additional use of the stability estimates for the discrete velocity \eqref{eq: velocity stab estimate UNS} we obtain
\begin{equation*} 
    \begin{aligned}
       \chtil(\uh^{n-1};\eu^{n},\vh) \leq  c\norm{\uh^{n-1}}_{\ah}^{1/2}\norm{\eu^n}_{\ah}\norm{\vh}_{\ah} + c\norm{\uh^{n-1}}_{\ah}^{1/2}\norm{\vh}_{\ah}\norm{\eu^n}_{L^2(\Gah)}^{1/2}\norm{\eu^n}_{\ah}^{1/2},
    \end{aligned}
\end{equation*}
proving the first bound  \eqref{eq: Trilinear Errors general 1 UNS}. 

\noindent By employing the decomposition \eqref{eq: decomposition error 1 Ritz UNS}, we can also derive the following
\begin{equation*}
    \begin{aligned}
        \chtil(\eu^{n-1};\bfu^n,\vh) &= \chtil(\rho_{\bfu}^{n-1};\bfu^n,\vh) + \chtil(\theta_{\bfu}^{n-1};\bfu^n,\vh)\\
        &\leq \norm{\bfu^n}_{H^1(\Gah)}\big(\norm{\rho_{\bfu}^{n-1}}_{L^2(\Gah)}^{1/2}\norm{\rho_{\bfu}^{n-1}}_{\ah}^{1/2} + c\norm{\theta_{\bfu}^{n-1}}_{L^2(\Gah)}^{1/2}\norm{\theta_{\bfu}^{n-1}}_{\ah}^{1/2}\big)\norm{\vh}_{\ah}.
    \end{aligned}
\end{equation*}
In a similar manner, for the second term, by applying the stability estimates in \Cref{Lemma: velocity stab estimate UNS}, and once more utilizing the bound \eqref{ch boundedness UNS}, we compute that
\begin{equation*}
    \begin{aligned}
      \chtil(\uh^{n-1};\eu^{n},\vh) &=  \chtil(\uh^{n-1};\rho_{\bfu}^n,\vh)+ \chtil(\uh^{n-1}; \theta_{\bfu}^n,\vh)\\
      &\leq c\norm{\uh^{n-1}}_{\ah}^{1/2}\norm{\rho_{\bfu}^n}_{\ah}\norm{\vh}_{\ah} +  c\norm{\uh^{n-1}}_{\ah}^{1/2}\norm{\thbfu^n}_{\ah}\norm{\vh}_{\ah}.
    \end{aligned}
\end{equation*}
Combining the two results, we get the second bound \eqref{eq: Trilinear Errors general 2 UNS}. Finally, \eqref{eq: Trilinear Errors general UNS} inequality becomes apparent when one considers the skew-symmetric nature of the inertia term as expressed by \eqref{eq: skew-symmetrized ch real Ph} and \Cref{remark: inertia term Ph}. 
\end{proof}

\begin{lemma}[Consistency errors]\label{lemma: Consistency errors UNS}
 Assume 
 $$ \partial_{tt}\bfu \in L^{2}(I_n;(L^2(\Ga))^3), \quad \partial_{t}\bfu \in L^{\infty}(I_n;(L^2(\Ga))^3)\cap L^{2}(I_n;(L^4(\Ga))^3),$$
 then we have the following consistency bounds 
 \begin{align}\label{eq: Consistency error 1 UNS}
        |\textsc{Err}_1^{C}(\vh)| &\leq   c\big(\sqrt{\Delta t} \norm{\partial_{tt}\bfu}_{L^2(I_n;L^2(\Ga))}+h^{k_g+1} \norm{\partial_{t}\bfu^n}_{L^2(\Ga)}\big)\norm{\vh}_{L^2(\Gah)},\\
        \label{eq: Consistency error 2 UNS}
        | \textsc{Err}_2^{C}(\vh)| &\leq ch^{k_g}(\norm{p^n}_{H^1(\Ga)} + \norm{\lambda^n}_{L^2(\Ga)})\norm{\vh}_{L^2(\Gah)},\\
        \label{eq: Consistency error 3 UNS}
        | \textsc{Err}_3^{C}(\vh)| &\leq ch^{k_g}\norm{\bfu^n}_{H^1(\Ga)}\norm{\bfu^{n-1}}_{H^1(\Ga)}\norm{\vh}_{H^1(\Gah)} + c\sqrt{\Delta t}\norm{\partial_t\bfu}_{L^2(I_n;L^4(\Ga))}\norm{\vh}_{\ah},\\
        \label{eq: Consistency error 4 UNS}
        | \textsc{Err}_4^{C}(\vh)| &\leq ch^{k_g+1}\norm{\bff^n}_{L^2(\Ga)}\norm{\vh}_{L^2(\Gah)},
\end{align}
for all $\vh \in \bfV_h$, where constant $c>0$ independent of $h$. 
\end{lemma}
\begin{proof}
Let us begin with the first consistency error, involving the time derivative.

 \noindent\underline{$\textsc{Err}_1^{C} :$} By means of changing the domain of integration, see \eqref{eq: errors of domain of integration data UNS} for $\bff = \partial_{t}\bfu^n$, and integration by parts 
\begin{equation*}
\begin{aligned}
    &(\frac{\bfu^n-\bfu^{n-1}}{\Delta t},\vh)_{L^2(\Gah)} - (\partial_{t}\bfu^n,\vhl)_{L^2(\Ga)} = (\frac{\partial_{t}\bfu^n}{\mu_h},\vhl)_{L^2(\Ga)}- (\partial_{t}\bfu^n,\vhl)_{L^2(\Ga)} \\
     &\ +(\int_{t_{n-1}}^{t_n} \frac{\partial_{t}\bfu(t,\cdot) - \bfu_t^n}{\Delta t} \, dt ,\vh)_{L^2(\Gah)}  \\
    &\leq ch^{k_g+1}\norm{\partial_{t}\bfu^n}_{L^2(\Ga)}\norm{\vhl}_{L^2(\Ga)} +\frac{1}{\Delta t } ( \int_{t_{n-1}}^{t_n} (t - t_{n-1})\partial_{tt}\bfu(t,\cdot) \, dt,\vh)_{L^2(\Gah)}\\
    & \leq ch^{k_g+1}\norm{\partial_{t}\bfu^n}_{L^2(\Ga)}\norm{\vh}_{L^2(\Gah)} +\frac{1}{\Delta t} (\Big( \int_{t_{n-1}}^{t_n} (t - t_{n-1})^2 \, dt \Big)^{1/2}\Big( \int_{t_{n-1}}^{t_n} \partial_{tt}\bfu^n(t,\cdot)^2 \, dt \Big)^{1/2},\vh)_{L^2(\Gah)}\\
    & \leq  ch^{k_g+1}\norm{\partial_{t}\bfu^n}_{L^2(\Ga)}\norm{\vh}_{L^2(\Gah)} + (\Delta t)^{1/2} \norm{\partial_{tt}\bfu^n}_{L^2(I_n;L^2(\Ga))} \norm{\vh}_{L^2(\Gah)},
    \end{aligned}
\end{equation*}
where in the last two inequalities we used the norm equivalence \eqref{eq: norm equivalence UNS} and  Schwarz inequality.\\

\noindent \underline{$\textsc{Err}_2^{C}$}, \underline{$\textsc{Err}_4^{C} :$} These estimates are  an immediate consequence of \eqref{eq: Geometric perturbations btilde UNS} and \eqref{eq: errors of domain of integration data UNS} respectively.\\

\noindent \underline{$\textsc{Err}_3^{C}:$}
\ We split this error into two parts: one associated with geometric errors and another concerning the different time steps in the first argument. So, adding and subtracting appropriate terms yields
\begin{equation}
    \begin{aligned}\label{eq: Consistency errors inside 1 UNS}
        \chtil(\bfu^{n-1};\bfu^n,\vh)  - c(\bfu^n;\bfu^n,\vhl) &= \chtil(\bfu^{n-1};\bfu^n,\vh) - c(\bfu^{n-1};\bfu^n,\vhl)\\
        &\ + c(\bfu^{n-1};\bfu^n,\vhl) - c(\bfu^n;\bfu^n,\vhl).
    \end{aligned}
\end{equation}

\noindent For the first line on the right-hand side of \eqref{eq: Consistency errors inside 1 UNS}, by \eqref{eq: skew-symmetrized ch real Ph} (see also \Cref{remark: inertia term Ph}) and \eqref{eq: continuous c error formula UNS}, along with the fact that $\nbgcov \bfu = \nbgcov(\bfPg\bfu) + (\bfu\cdot\bfn) \bfH$, we see that it is equal to
\begin{equation*}
    \begin{aligned}
        \chtil(\bfu^{n-1};\bfu^n,\vhl) - c(\bfu^{n-1};\bfu^n,\vhl) & = \frac{1}{2}\int_{\Gah}((\bfu^{n-1}\cdot\nbgcovh)\bfu^n)\cdot\vh -  \frac{1}{2}\int_{\Ga}((\bfu^{n-1}\cdot\nbgcov)\bfu^n)\cdot\vhl\\
        &\ -  \frac{1}{2}\int_{\Gah} \bfu^n\cdot\nh \bfu^{n-1}\cdot\bfH_h\vh\\
        &\ + \frac{1}{2}\int_{\Gah}((\bfu^{n-1}\cdot\nbgcovh)\vh)\cdot\bfu^n -  \frac{1}{2}\int_{\Ga}((\bfu^{n-1}\cdot\nbgcov)\vhl)\cdot\bfu^n\\
        &\ -  \frac{1}{2}\int_{\Gah} \vh\cdot\nh \bfu^{n-1}\cdot\bfH_h\bfu^n + \frac{1}{2}\int_{\Ga} \vhl\cdot\bfng \bfu^{n-1}\cdot\bfH\bfu^n.
    \end{aligned}
\end{equation*}
Using the bounds in \Cref{lemma: errors of geometric pert UNS} and the geometric estimates $\norm{\bfn- \nh}_{L^\infty(\Gah)} \leq ch^{k_g}$, $\norm{\bfP- \bfPh}_{L^\infty(\Gah)} \leq ch^{k_g}$, $|1 - \frac{1}{\muh}| \leq ch^{k_g+1}$ as seen in \eqref{eq: geometric errors 1 UNS}, \eqref{eq: geometric errors 2 UNS}, \eqref{eq: muhkg estimate UNS} respectively, along with the fact that $\bfu^n\cdot\nh = \bfu^n\cdot(\nh-\bfn)$, we clearly see that 
\begin{equation}
    \begin{aligned}\label{eq: consistency inside UNS}
        &|\chtil(\bfu^{n-1};\bfu^n,\vhl) - c(\bfu^{n-1};\bfu^n,\vhl)|\\
        & \leq ch^{k_g}\norm{\bfu^{n-1}}_{H^1(\Ga)}\big(\norm{\bfu^n}_{H^1(\Ga)}\norm{\vh}_{L^2(\Gah)}+ \norm{\bfu^n}_{L^2(\Gah)}\norm{\vh}_{H^1(\Ga)}\big)
        \\
        & \ + \Big| \frac{1}{2}\int_{\Gah} \vh\cdot(\nh-\bfn) \bfu^{n-1}\cdot\bfH_h\bfu^n \Big| + \Big|\frac{1}{2}\int_{\Ga} (\vhl\cdot\bfng) \bfu^{n-1}\cdot\bfH\bfu^n -  \frac{1}{2}\int_{\Gah} (\vh\cdot\bfn) \bfu^{n-1}\cdot\bfH_h\bfu^n\Big|\\
        & \leq  ch^{k_g}\norm{\bfu^{n-1}}_{H^1(\Ga)}\norm{\bfu^n}_{H^1(\Ga)}\norm{\vh}_{H^1(\Gah)}
    \end{aligned}
\end{equation}
where in the last inequality we used \eqref{eq: weingarten map improved UNS} with $\beta = \vhl\cdot\bfng$, and $\vh = \bfu^{n-1}\in \bfH^1(\Gah)$. For the second and last lines of \eqref{eq: Consistency errors inside 1 UNS}, we have from \eqref{eq: continuous c error formula UNS} that 
\begin{equation*}
    \begin{aligned}
        |c(\bfu^{n-1};\bfu^n,\vhl) - c(\bfu^n;\bfu^n,\vhl)| &= \frac{1}{2} \underbrace{\int_{\Ga } ((\bfu^{n} \cdot \nbgcov)\bfu^{n} ) \cdot \vhl \; \ds - \frac{1}{2}\int_{\Ga } ((\bfu^{n-1} \cdot \nbgcov)\bfu^{n} ) \cdot \vhl \; \ds}_{\textbf{I}c} \\
        &\ -\frac{1}{2} \underbrace{\int_{\Ga } ((\bfu^{n} \cdot \nbgcov)\bfPg\vhl ) \cdot \bfu^{n} \; \ds
        + \frac{1}{2}\int_{\Ga } ((\bfu^{n-1} \cdot \nbgcov)\bfPg\vhl ) \cdot \bfu^{n} \; \ds}_{\textbf{II}c}.
    \end{aligned}
\end{equation*}
We focus on $(\textbf{II}c)$, since $(\textbf{I}c)$ will follow in a similar way. Using the fact that 
\begin{equation*}
    \int_{t_{n-1}}^{t^n} \partial_t\bfu(t,\cdot) \, \mathrm{d}t\leq \sqrt{\Delta t}\norm{\partial_t\bfu}_{L^2(I_n)},
\end{equation*}
a straightforward calculation reveals that
\begin{equation}
    \begin{aligned}\label{eq: consistency inside IIc UNS}
         \textbf{II}c = \int_{\Ga} (\Big(\int_{t_{n-1}}^{t^n} \partial_t\bfu(t,\cdot) dt\Big)\cdot \nbgcov)\bfPg\vhl ) \cdot \bfu^{n} \leq \sqrt{\Delta t}\norm{\vh}_{\ah}\norm{\partial_t\bfu}_{L^2(I_n;L^4(\Ga))}\norm{\bfu^n}_{L^4(\Ga)},
    \end{aligned}
\end{equation}
where we made use of $\norm{\nbgcov\bfPg\vhl}_{L^2(\Ga)} \leq c\norm{\bfPg\vhl}_{a} \leq c\norm{\vh}_{\ah}$ by the perturbation bound \eqref{eq: Geometric perturbations a 2 UNS} and the $H^1$ coercivity bound \eqref{eq: korn inequality Ph UNS}. Similarly, recalling the bound \eqref{ch boundedness UNS} we get
\begin{equation}\label{eq: consistency inside Ic UNS}
    \textbf{I}c \leq \sqrt{\Delta t}\norm{\vh}_{\ah}\norm{\partial_t\bfu}_{L^2(I_n;L^4(\Ga))}\norm{\bfu^n}_{L^4(\Ga)}. 
\end{equation}
To complete the proof, we insert \eqref{eq: consistency inside UNS} and \eqref{eq: consistency inside IIc UNS}, \eqref{eq: consistency inside Ic UNS} into \eqref{eq: Consistency errors inside 1 UNS} and recall the embedding $L^4 \hookrightarrow H^1$.
\end{proof}

Now we are ready to show the error estimates involving the velocity. We split it into two cases. One where we assume $k_\lambda=k_u$ and a second one where we assume  $k_\lambda=k_u-1$. Let us start with the first case.

\begin{lemma}\label{lemma: discrete remainder velocity estimates UNS}
Assume $\underline{k_\lambda = k_u}$, and that  the regularity \cref{assumption: Regularity assumptions for velocity estimate} hold. 
Let $(\bfu,\{p,\lambda\})$ be the solution of \eqref{eq: unstead NV Lagrange} and let $\uh^k$ and $\{\ph^k,\lh^k\}$, $k=1,...,n$ be the discrete solutions of \eqref{eq: weak lagrange fully discrete UNS}, with initial condition $\uh^0=\mathcal{R}_h\bfu^0$. Then the following estimate holds for $1 \leq n \leq N$:
\begin{equation}
    \begin{aligned}\label{eq: discrete remainder velocity estimates UNS}
        \norm{\thbfu^n}_{L^2(\Gah)}^2 &+ \sum_{k=1}^n\norm{\thbfu^k-\thbfu^{k-1}}_{L^2(\Gah)}^2+ \Delta t \sum_{k=1}^n \norm{\thbfu^k}_{\ah}^2 \\
        &\qquad\qquad\qquad\leq Cexp(ct_n)\big((\Delta t)^2 + h^{2\widehat{r}_u} + h^{2k_{pr}+2} +h^{2k_{\lambda}+2}\big),
    \end{aligned}
\end{equation}
where $\widehat{r}_u = min\{k_{u},k_g\}$, with $C=C(u,p,\lambda,\bff) =  \sup_{t\in[0,T]} \big(\norm{p}_{H^{k_{pr}+1}(\Ga)}^2   +\norm{\lambda}_{H^{k_{\lambda}+1}(\Ga)}^2 + \norm{\bfu}_{H^{k_u+1}(\Ga)}^2 +\norm{\partial_t\bfu}_{L^2(\Ga)}^2 + \norm{\bff}_{L^2(\Ga)}^2\big) + \norm{\partial_{tt}\bfu}_{L^2([0,T];L^2(\Ga))}^2 + \norm{\partial_t\bfu}_{L^2([0,T];H^{k_u+1}(\Ga))}^2.$

\end{lemma}

\begin{proof}
We test \eqref{eq: error equation UNS} with $\vh = \thbfu^n\in\bfV_h^{div}$ and $\{q_h,\xi_h\} = \{\theta_p^n,\theta_{\lambda}^n\}$ and subtract to obtain 
\begin{equation}
    \begin{aligned}\label{eq: discrete remainder velocity estimates inside UNS}
        \norm{\thbfu^n}_{L^2(\Gah)}^2 +\norm{\thbfu^n-\thbfu^{n-1}}_{L^2(\Gah)}^2 &+ 2\Delta t \norm{\thbfu^n}_{\ah}^2 \leq  \norm{\thbfu^{n-1}}_{L^2(\Gah)}^2 \\
        & + 2\Delta t \Big( |\sum_{i=1}^4 \textsc{Err}_{i}^{C}(\thbfu^n)| +  |\sum_{i=1}^2 \textsc{Err}_{i}^{I}(\thbfu^n)| + |\mathcal{C}(\thbfu^n)|\Big),
    \end{aligned}
\end{equation}
where $\textsc{Err}_{5}^{C}(\thbfu^n)=0$ by the Ritz-Stokes projection \eqref{def: surface Ritz-Stokes projection UNS}.
Now we calculate the terms in the last line according to the interpolation and consistency estimates found previously. From \cref{lemma: Consistency errors UNS} and $\norm{\thbfu^n}_{H^1(\Gah)} \leq \norm{\thbfu^n}_{\ah}$, cf. \eqref{eq: improved h1-ah bound UNS}, using Young's inequality appropriately we see that
\begin{equation}
    \begin{aligned}\label{eq: discrete remainder velocity estimates inside 1 UNS}
      |\sum_{i=1}^4 \textsc{Err}_{i}^{C}(\thbfu^n)| &\leq c\Delta t \big(\norm{\partial_{tt}\bfu}_{L^2(I_n;L^2(\Ga))}^2 + \norm{\partial_{t}\bfu}_{L^2(I_n;L^4(\Ga))}^2\big)  +ch^{2k_g}\big(\norm{\bfu^n}_{H^1(\Ga)}^2 \\
      &\ +\norm{\partial_t\bfu^n}_{L^2(\Ga)}^2 + \norm{p^n}_{H^1(\Ga)}^2 + \norm{\lambda^n}_{L^2(\Ga)}^2\big) + ch^{2k_g+2}\norm{\bff^n}_{L^2(\Ga)}^2  + \frac{1}{5}\norm{\thbfu^n}_{\ah}^2.
    \end{aligned}
\end{equation}
For the interpolation errors due to \cref{lemma: Interpolation errors UNS} and specifically \eqref{eq: Interpolation errors improved UNS} we, likewise, have
\begin{equation}
    \begin{aligned}\label{eq: discrete remainder velocity estimates inside 2 UNS}
        |\sum_{i=1}^2 \textsc{Err}_{i}^{I}(\thbfu^n)| \leq c\frac{h^{2\widehat{r}_u+2}}{\Delta t}\norm{\partial_t\bfu}_{L^2(I_n;H^{k_u+1}(\Ga))}^2 + c(h^{2k_{pr}+2} + h^{2k_{\lambda}+2})\big( \norm{p^n}_{H^{k_{pr}+1}(\Ga)}^2 \\ 
        + \norm{\lambda^n}_{H^{k_{\lambda}+1}(\Ga)}^2\big)
        + \frac{1}{5}\norm{\thbfu^n}_{\ah}^2.
    \end{aligned}
\end{equation}
Lastly, by employing the bound \eqref{eq: Trilinear Errors general UNS} in \cref{lemma: Trilinear Errors UNS} for the inertia term,  the Ritz-Stokes projection estimate \eqref{eq: Error Bounds Ritz-Stokes improved UNS} and two Young's inequalities we obtain 
\begin{equation}
    \begin{aligned}\label{eq: discrete remainder velocity estimates inside 3 UNS}
        |\mathcal{C}(\thbfu^n)| \leq ch^{2\widehat{r}_u}(1+\norm{\uh^{n-1}}_{\ah})\norm{\bfu^n}_{H^{k_u+1}(\Ga)}^2 + c\norm{\thbfu^{n-1}}_{L^2(\Gah)}^2  + \frac{1}{5}\norm{\thbfu^{n-1}}_{\ah}^2 + \frac{1}{5}\norm{\thbfu^{n}}_{\ah}^2.
    \end{aligned}
\end{equation}
Combining \eqref{eq: discrete remainder velocity estimates inside 1 UNS}, \eqref{eq: discrete remainder velocity estimates inside 2 UNS}, and \eqref{eq: discrete remainder velocity estimates inside 3 UNS} into \eqref{eq: discrete remainder velocity estimates inside UNS}, after a kickback argument we get
\begin{equation}
    \begin{aligned}\label{eq: discrete remainder velocity estimates inside 4 UNS}
         \norm{\thbfu^n}_{L^2(\Gah)}^2 &+\norm{\thbfu^n-\thbfu^{n-1}}_{L^2(\Gah)}^2 + \frac{4}{5}\Delta t \norm{\thbfu^n}_{\ah}^2 \leq  (1+c\Delta t)\norm{\thbfu^{n-1}}_{L^2(\Gah)}^2 + \frac{2}{5}\Delta t \norm{\thbfu^{n-1}}_{\ah}^2 \\
        & + c\Delta t \Big( \Delta t \big(\norm{\partial_{tt}\bfu}_{L^2(I_n;L^2(\Ga))}^2 + \norm{\partial_{t}\bfu}_{L^2(I_n;L^4(\Ga))}^2\big) + \frac{ch^{2\widehat{r}_u+1}}{\Delta t}\norm{\partial_t\bfu}_{L^2(I_n;H^{k_u+1}(\Ga))}^2\\
        & + c h^{2\widehat{r}_u} (1+\norm{\uh^{n-1}}_{\ah})\norm{\bfu^n}_{H^{k_u+1}(\Ga)}^2 + c(h^{2\widehat{r}_u} + h^{2k_{pr}+2} + h^{2k_{\lambda}+2})\big( \norm{p^n}_{H^{k_{pr}+1}(\Ga)}^2 \\
        &+ \norm{\lambda^n}_{H^{k_{\lambda}+1}(\Ga)}^2
       + \norm{\bfu^n}_{H^1(\Ga)}^2 +\norm{\partial_t\bfu^n}_{L^2(\Ga)}^2\big)
         + ch^{2k_g+2}\norm{\bff^n}_{L^2(\Ga)}^2
        \Big),
    \end{aligned}
\end{equation}
where $\widehat{r}_u = min\{k_{u},k_g\}$. Now, since $\thbfu^0 =0$, due to the choice of the initial condition $\uh^0 = \mathcal{R}_h\bfu^0$, and since $\thbfu^n\in\bfV_h^{div}$ by   \eqref{def: surface Ritz-Stokes projection UNS}, upon summation over $k=1,...,n$ of \eqref{eq: discrete remainder velocity estimates inside 4 UNS} and straightforward application of the discrete Gronwall's inequality, it yields
\begin{equation}
    \begin{aligned}
        \norm{\thbfu^n}_{L^2(\Gah)}^2 &+ \sum_{k=1}^n\norm{\thbfu^k-\thbfu^{k-1}}_{L^2(\Gah)}^2 + \frac{1}{5}\Delta t \sum_{k=1}^n\norm{\thbfu^k}_{\ah}^2 \leq cexp(ct_n)\Big( h^{2k_g+2}\sup_{t\in[0,T]}\norm{\bff}_{L^2(\Ga)}^2\\
        & + \Delta t^2 \big(\norm{\partial_{tt}\bfu}_{L^2([0,T];L^2(\Ga))}^2 + \norm{\partial_{t}\bfu}_{L^2([0,T];L^4(\Ga))}^2\big)\\
        & + Ch^{2\widehat{r}_u}\big(\norm{\partial_t\bfu}_{L^2([0,T];H^{k_u+1}(\Ga))}^2 +\sup_{t\in[0,T]}\big(\norm{\bfu}_{H^{k_u+1}(\Ga)}^2 +\norm{\partial_t\bfu}_{L^2(\Ga)}^2 \big)\big)\\
        & + (h^{2k_{pr}+2} + h^{2k_{\lambda}+2} )\sup_{t\in[0,T]}\big(\norm{p}_{H^{k_{pr}+1}(\Ga)}^2  + \norm{\lambda}_{H^{k_{\lambda}+1}(\Ga)}^2\big) \Big),
    \end{aligned}
\end{equation}
where the constant $C$ on the third line emerged as a result of applying the discrete stability estimate $\Delta t \sum_{k=1}^n \norm{\uh^{n-1}}_{\ah} \leq C$; see \eqref{eq: velocity stab estimate UNS}. This completes the proof.
\end{proof}

\begin{lemma}\label{lemma: discrete remainder velocity estimates kl=ku-1 UNS}
Assume $\underline{k_\lambda = k_u-1}$, $k_g \geq 2$, and that the regularity \cref{assumption: Regularity assumptions for velocity estimate} hold. 
Let $(\bfu,\{p,\lambda\})$ be the solution of \eqref{eq: unstead NV Lagrange} and let $\uh^k$ and $\{\ph^k,\lh^k\}$, $k=1,...,n$ be the discrete solutions of \eqref{eq: weak lagrange fully discrete UNS}, with initial condition $\uh^0=\mathcal{R}_h\bfu^0$. Then the following estimate holds for $1 \leq n \leq N$:
\begin{equation}
    \begin{aligned}\label{eq: discrete remainder velocity estimates kl=ku-1 UNS}
        \norm{\thbfu^n}_{L^2(\Gah)}^2 &+ \sum_{k=1}^n\norm{\thbfu^k-\thbfu^{k-1}}_{L^2(\Gah)}^2+ \Delta t \sum_{k=1}^n \norm{\thbfu^k}_{\ah}^2 \\
        &\qquad\qquad\qquad\leq Cexp(ct_n)\big((\Delta t)^2 + h^{2r_u} + h^{2k_{pr}+2} +h^{2k_{\lambda}+2}\big),
    \end{aligned}
\end{equation}
where $r_u = min\{k_{u},k_g-1\}$, with $C=C(u,p,\lambda,\bff) =  \sup_{t\in[0,T]} \big(\norm{p}_{H^{k_{pr}+1}(\Ga)}^2   +\norm{\lambda}_{H^{k_{\lambda}+1}(\Ga)}^2 + \norm{\bfu}_{H^{k_u+1}(\Ga)}^2 +\norm{\partial_t\bfu}_{L^2(\Ga)}^2 + \norm{\bff}_{L^2(\Ga)}^2\big) + \norm{\partial_{tt}\bfu}_{L^2([0,T];L^2(\Ga))}^2 + \norm{\partial_t\bfu}_{L^2([0,T];H^{k_u+1}(\Ga))}^2.$
\end{lemma}
\begin{proof}
The proof follows as in \cref{lemma: discrete remainder velocity estimates UNS}, where instead of the Ritz-Stokes estimate \eqref{eq: Error Bounds Ritz-Stokes improved UNS} (for $k_\lambda=k_u$) we use the worse estimate \eqref{eq: Error Bounds Ritz-Stokes UNS} throughout. Moreover, in \eqref{eq: discrete remainder velocity estimates inside 2 UNS} we use the bound \eqref{eq: Interpolation errors UNS} instead of \eqref{eq: Interpolation errors improved UNS}, and in \eqref{eq: discrete remainder velocity estimates inside 1 UNS} we use instead $\norm{\thbfu^n}_{H^1(\Gah)} \leq h^{-1}\norm{\thbfu^n}_{\ah}$ cf. \eqref{eq: coercivity and Korn's inequality Lagrange UNS}. We omit further details.
\end{proof}

\begin{remark}[About $L^2_{L^2}\times L^2_{L^2}$ pressure stability]\label{remark: About pressure stability UNS}
Now that we have tracked the regularity needed (\Cref{assumption: Regularity assumptions for velocity estimate}) and proved the estimates \eqref{eq: discrete remainder velocity estimates UNS} and \eqref{eq: discrete remainder velocity estimates kl=ku-1 UNS}, we are able to go back and prove a bound of the form \eqref{eq: uniform bound b stability UNS} for $\bfb_h^{n-1} = \uh^{n-1}$. Using the Ritz-Stokes projection \eqref{eq: surface Ritz-Stokes projection UNS}, we notice that
\begin{equation}
    \begin{aligned}
        \norm{\uh^n}_{\ah} \leq \norm{\mathcal{R}_h \bfu^n}_{\ah} + \norm{\thbfu^n}_{\ah}\leq  \norm{\mathcal{R}_h \bfu^n}_{\ah} + h^{-1}\norm{\thbfu^n}_{L^2(\Gah)}, 
    \end{aligned}
\end{equation}
where in the last inequality we have used the inverse inequality. Now it is easy to see, due to the stability of the Ritz-Stokes projection \eqref{eq: Stab estimates for Ritz-Stokes projection UNS} and convergence results \eqref{eq: discrete remainder velocity estimates UNS} (or \eqref{eq: discrete remainder velocity estimates kl=ku-1 UNS}), that the assumption \eqref{eq: uniform bound b stability UNS} holds, more precisely we have that
\begin{align}
\label{eq: uniform bound uha stability UNS}
    \sup_n \norm{\uh^n}_{\ah} \leq C,
\end{align}
if the assumptions in  \cref{lemma: discrete remainder velocity estimates UNS} (or \cref{lemma: discrete remainder velocity estimates kl=ku-1 UNS}) hold and $\Delta t \leq c h.$
\noindent This time-step condition is reasonable since it will be automatically satisfied if we try to balance spatial and temporal discretization errors, as this leads to $\Delta t \sim h^{k_u}$.
\end{remark}

\begin{remark}[About $L^2_{L^2}\times L^2_{L^2}$ pressure stability for $k_\lambda=k_u-1$]\label{remark: About the pressure stability for arbitrary convergence UNS}
As mentioned before in \cref{remark: About the pressure stability for arbitrary}, to prove $L^2_{L^2}\times L^2_{L^2}$ pressure stability for arbitrary choice of $k_\lambda$, we would need to prove $\sup_{n} \norm{\bfu_h^n}_{L^{\infty}(\Gah)} \leq c$, cf. \eqref{eq: b extra stability UNS}. To see this, first recall the  $L^{\infty}$-bound of the  Ritz-Stokes projection  \eqref{eq: Linfty ritz stokes UNS}, then using an inverse inequality we see that
\begin{equation}
    \begin{aligned}\label{eq: Linfty uh bound UNS}
        \norm{\uh^n}_{L^{\infty}(\Gah)} \leq \norm{\mathcal{R}_h \bfu^n}_{L^{\infty}(\Gah)} + \norm{\thbfu^n}_{L^{\infty}(\Gah)}\leq  \norm{\bfu^n}_{W^{2,\infty}(\Ga)} + h^{-1}\norm{\thbfu^n}_{L^2(\Gah)},
    \end{aligned}
\end{equation}
which yields our desired result considering the velocity error convergence estimates \eqref{eq: discrete remainder velocity estimates UNS} or \eqref{eq: discrete remainder velocity estimates kl=ku-1 UNS} and assuming $\Delta t \leq ch$, cf. \cref{remark: About pressure stability UNS}, along with further regularity for the velocity $\bfu \in L^{\infty}([0,T];W^{2,\infty}(\Ga))$.
\end{remark}




\begin{theorem}[Velocity Error Estimates]\label{theorem: Velocity Error Estimates UNS}
Under the  \cref{assumption: Regularity assumptions for velocity estimate} and \Cref{lemma: discrete remainder velocity estimates UNS} the following velocity error estimates hold for $1 \leq n \leq N$
\begin{equation}\label{eq: Velocity Error Estimates UNS}
    \begin{aligned}
       \norm{\eu^n}_{L^2(\Gah)}^2 + \sum_{k=1}^n\norm{\eu^k-\eu^{k-1}}_{L^2(\Gah)}^2 &+ \Delta t \sum_{k=1}^n \norm{\eu^k}_{\ah}^2 \\
       & \leq Cexp(ct_n)\big((\Delta t)^2 + h^{2\widehat{r}_u} + h^{2k_{pr}+2} +h^{2k_{\lambda}+2}\big), 
    \end{aligned}
\end{equation}
where $\widehat{r}_u = min\{k_{u},k_g\}$, with $C$ the constant in \Cref{lemma: discrete remainder velocity estimates UNS}.

Furthermore, under the Assumption  \ref{assumption: Regularity assumptions for velocity estimate} and \Cref{lemma: discrete remainder velocity estimates kl=ku-1 UNS} we instead have the following velocity error estimates hold for $1 \leq n \leq N$
\begin{equation}\label{eq: Velocity Error Estimates kl=ku-1 UNS}
    \begin{aligned}
       \norm{\eu^n}_{L^2(\Gah)}^2 + \sum_{k=1}^n\norm{\eu^k-\eu^{k-1}}_{L^2(\Gah)}^2 &+ \Delta t \sum_{k=1}^n \norm{\eu^k}_{\ah}^2 \\
       & \leq Cexp(ct_n)\big((\Delta t)^2 + h^{2r_u} + h^{2k_{pr}+2} +h^{2k_{\lambda}+2}\big), 
    \end{aligned}
\end{equation}
where $r_u = min\{k_{u},k_g-1\}$, with $C$ the constant in \Cref{lemma: discrete remainder velocity estimates kl=ku-1 UNS}.
\end{theorem}
\begin{proof}
    This can be readily seen from the decomposition \eqref{eq: decomposition error 1 Ritz UNS}, the triangle inequality, the interpolation estimates \eqref{eq: Error Bounds Ritz-Stokes UNS}, \eqref{eq: Error Bounds Ritz-Stokes improved UNS} and the estimates of the discrete remainder in \eqref{eq: discrete remainder velocity estimates UNS}, \eqref{eq: discrete remainder velocity estimates kl=ku-1 UNS}.
\end{proof}

\subsection{Pressure a-priori estimates for $\underline{k_\lambda = k_u}$}\label{sec: pressure a-priori kl=ku UNS}
We now consider \cref{assumption: Regularity assumptions for velocity estimate} and the time-step condition in \cref{remark: About pressure stability UNS}, that is $$\Delta t \leq ch.$$
We also recall that \eqref{eq: uniform bound uha stability UNS} holds under the assumptions in \cref{lemma: discrete remainder velocity estimates UNS}. We now establish a-priori estimates for the two pressures when $k_\lambda = k_u$. We will make use of the ideas introduced in the stability analysis section (\Cref{Sec: Stability analysis UNS}), that is, the discrete inverse Stokes operator $\mathcal{A}_h$ \eqref{eq: Discrete inverse Stokes UNS} to bound the approximation of the time derivative in an appropriate negative norm,
and the $L^2\times H_h^{-1}$ \textsc{inf-sup} condition \eqref{eq: L^2 H^{-1} discrete inf-sup condition Gah Lagrange UNS} to establish optimal pressure a-priori estimates.





\begin{lemma}[Auxiliary result]\label{lemma: Auxiliary convergence result UNS}
Assume $\underline{k_\lambda = k_u}$, and that the regularity \cref{assumption: Regularity assumptions for velocity estimate} hold, with time-step $\Delta t \leq ch$. Let $(\bfu,\{p,\lambda\})$ be the solution of \eqref{eq: unstead NV Lagrange}  and let $\uh^k$ and $\{\ph^k,\lh^k\}$, $k=1,...,n$ be the discrete solutions of \eqref{eq: weak lagrange fully discrete UNS}, with initial condition $\uh^0=\mathcal{R}_h\bfu^0$. Recalling the bound \eqref{eq: uniform bound uha stability UNS}, then the following estimate holds for $1 \leq n \leq N$
\begin{equation}
    \begin{aligned}\label{eq: Auxiliary convergence result imrpoved UNS}
        \sum_{k=1}^n\Delta t \norm{\frac{\mathcal{A}_{h}(\thbfu^k - \thbfu^{k-1})}{\Delta t}}_{\ah}^2 \leq C\big((\Delta t)^2 + h^{2\widehat{r}_u} + h^{2k_{pr}+2} +h^{2k_{\lambda}+2}\big),
    \end{aligned}
\end{equation}
where $\widehat{r}_u = min\{k_{u},k_g\}$ and $C$ the constant in \Cref{lemma: discrete remainder velocity estimates UNS}
\end{lemma}
\begin{proof}
We start by testing our error equations \eqref{eq: error equation UNS} with $\mathcal{A}_{h}(\thbfu^n - \thbfu^{n-1}) \in \bfV_h^{div}$ (which is rigorous since $\thbfu^n - \thbfu^{n-1} \in \bfV_h^{div}$ for all $n\geq 1$), then recalling \eqref{eq: Discrete inverse Stokes UNS} we derive that 
\begin{equation}
    \begin{aligned}\label{eq: Auxiliary convergence result inside UNS}
        \frac{1}{\Delta t }\norm{\mathcal{A}_{h}(\thbfu^n - \thbfu^{n-1})}_{\ah}^2 &+ \ah(\thbfu^n,\mathcal{A}_{h}(\thbfu^n - \thbfu^{n-1})) \leq 2\Big( |\sum_{i=1}^4 \textsc{Err}_{i}^{C}(\mathcal{A}_{h}(\thbfu^n - \thbfu^{n-1}))| \\[-5pt]
        &\qquad \  +  |\sum_{i=1}^2 \textsc{Err}_{i}^{I}(\mathcal{A}_{h}(\thbfu^n - \thbfu^{n-1}))| + |\mathcal{C}(\mathcal{A}_{h}(\thbfu^n - \thbfu^{n-1}))|\Big),
    \end{aligned}
\end{equation}
where we also utilized that $\textsc{Err}_{5}^{C}(\mathcal{A}_{h}(\thbfu^n - \thbfu^{n-1}))=0$ by \eqref{eq: surface Ritz-Stokes projection UNS}, since $\mathcal{A}_{h}(\thbfu^n - \thbfu^{n-1}) \in \bfV_h^{div}$. Now we want to bound the rest of the results appropriately. Again, this is similar to \cref{Lemma: auxilary stab bounds UNS}. The idea is a suitable use of Young's inequality followed by a kickback argument against the first term. We have the following:
\begin{itemize}
    \item The second can be bounded as
    \begin{equation*}
        \begin{aligned}
          \ah(\thbfu^n,\mathcal{A}_{h}(\thbfu^n - \thbfu^{n-1})) \leq   c\Delta t \norm{\thbfu^n}_{\ah}^2 +  \frac{1}{5\Delta t }\norm{\mathcal{A}_h(\thbfu^n - \thbfu^{n-1})}_{\ah}^2.
        \end{aligned}
    \end{equation*}
\item For the interpolations errors using \Cref{lemma: Interpolation errors UNS} and specifically \eqref{eq: Interpolation errors improved UNS}  we get
        \begin{equation*}
            \begin{aligned}
                &|\sum_{i=1}^2 \textsc{Err}_{i}^{I}(\mathcal{A}_{h}(\thbfu^n - \thbfu^{n-1}))| \leq \Delta t \frac{h^{2\widehat{r}_u+2}}{\Delta t}\norm{\partial_t\bfu}_{L^2(I_n;H^{k_u+1}(\Ga))}^2 \\
                &\qquad + c \Delta t(h^{2k_{pr}+2} + h^{2k_{\lambda}+2})\big( \norm{p^n}_{H^{k_{pr}+1}(\Ga)}^2  
        + \norm{\lambda^n}_{H^{k_{\lambda}+1}(\Ga)}^2\big)
        + \frac{1}{5\Delta t }\norm{\mathcal{A}_h(\thbfu^n - \thbfu^{n-1})}_{\ah}^2.
            \end{aligned}
        \end{equation*}
\item For the consistency errors using \Cref{lemma: Consistency errors UNS} and $\norm{\mathcal{A}_{h}(\thbfu^n - \thbfu^{n-1})}_{H^1(\Gah)} \leq \norm{\mathcal{A}_{h}(\thbfu^n - \thbfu^{n-1})}_{\ah}$, cf. \eqref{eq: improved h1-ah bound UNS}, we may obtain
        \begin{equation*}
            \begin{aligned}
      &|\sum_{i=1}^4 \textsc{Err}_{i}^{C}(\mathcal{A}_{h}(\thbfu^n - \thbfu^{n-1}))| \leq c(\Delta t)^2 \big(\norm{\partial_{tt}\bfu}_{L^2(I_n;L^2(\Ga))}^2 + \norm{\partial_{t}\bfu}_{L^2(I_n;L^4(\Ga))}^2\big)\\
      &\qquad +c\Delta t h^{2k_g}\big(\norm{\bfu^n}_{H^1(\Ga)}^2 
       +\norm{\partial_t\bfu^n}_{H^1(\Ga)}^2 + \norm{p^n}_{H^1(\Ga)}^2 + \norm{\lambda^n}_{L^2(\Ga)}^2\big) + c\Delta t h^{2k_g+2}\norm{\bff^n}_{L^2(\Ga)}^2  \\
       &\qquad \qquad + \frac{1}{5\Delta t }\norm{\mathcal{A}_h(\thbfu^n - \thbfu^{n-1})}_{\ah}^2.
            \end{aligned}
        \end{equation*}
\item Finally we can bound the inertia term with the help of the uniform estimate \eqref{eq: uniform bound uha stability UNS} and bound \eqref{eq: Trilinear Errors general 1 UNS} in \Cref{lemma: Trilinear Errors UNS}
\begin{equation*}
    \begin{aligned}
        |\mathcal{C}(\mathcal{A}_{h}(\thbfu^n - \thbfu^{n-1}))| \leq c\Delta t \norm{\eu^{n-1}}_{L^2(\Gah)}\norm{\eu^{n-1}}_{\ah} + C\Delta t \norm{\eu^n}_{\ah}^2 + \frac{1}{5\Delta t }\norm{\mathcal{A}_h(\thbfu^n - \thbfu^{n-1})}_{\ah}^2,
    \end{aligned}
\end{equation*}
where $C$ is the constant in \eqref{eq: uniform bound uha stability UNS}.
\end{itemize}
Now putting everything together, using a kickback argument against the term $\frac{1}{\Delta t }\norm{\mathcal{A}_h(\thbfu^n - \thbfu^{n-1})}_{\ah}^2$, summing for $k=1,...,n$ and using the convergence results involving the velocity, \Cref{lemma: discrete remainder velocity estimates UNS} and \eqref{eq: Velocity Error Estimates UNS} in \Cref{theorem: Velocity Error Estimates UNS}, we get our desired result \eqref{eq: Auxiliary convergence result imrpoved UNS}.
\end{proof}

We now want to finish the convergence results by providing error estimates for the two pressures. First, it suffices to control the discrete remainders $\{\theta_p^n,\theta_{\lambda}^n\}$. For that, we will use the same techniques as in pressure stability estimates; see \Cref{Lemma: Pressure stab Estimate UNS}. 



\begin{lemma}\label{lemma: discrete remainder pressure estimates improved UNS}
Assume $\underline{k_{\lambda}=k_u}$ and that \cref{assumption: Regularity assumptions for velocity estimate} hold, with time-step $\Delta t \leq ch$.  Let $(\bfu,\{p,\lambda\})$ be the solution of \eqref{eq: unstead NV Lagrange}  and let 
 $\uh^k$ and $\{\ph^k,\lh^k\}$, $k=1,...,n$ be the discrete solutions of \eqref{eq: weak lagrange fully discrete UNS}, with initial condition $\uh^0=\mathcal{R}_h\bfu^0$. Recalling the bound \eqref{eq: uniform bound uha stability UNS}, then the following estimate holds for $1 \leq n \leq N$
\begin{equation}\label{eq: discrete remainder pressure estimates improved UNS}
            \Delta t \sum_{k=1}^{n} \norm{\{\theta_p^k,\theta_{\lambda}^k\}}_{L^2(\Gah)\times H_h^{-1}(\Gah)}^2 \leq C\big((\Delta t)^2 + h^{2\widehat{r}_u} + h^{2k_{pr}+2} +h^{2k_{\lambda}+2}\big),
     \end{equation}
where $\widehat{r}_u = min\{k_{u},k_g\}$, 
and $C$ the constant in \Cref{lemma: discrete remainder velocity estimates UNS}.   
\end{lemma}
\begin{proof}
Recall the $L^2 \times H_h^{-1}$ \textsc{inf-sup} condition \Cref{lemma: L^2 H^{-1} discrete inf-sup condition Gah Lagrange UNS}, then since $\{\theta_p^n,\theta_{\lambda}^n\}\in Q_h\times\Lambda_h$ it holds
    \begin{equation}
        \begin{aligned}\label{eq: Pressure conv Estimate UNS inside inf-sup UNS}
            \norm{\{\theta_p^n,\theta_{\lambda}^n\}}_{L^2(\Gah)\times H_h^{-1}(\Gah)} \leq \sup_{\vh \in \bfV_h} \frac{\bhtil(\vh,\{\theta_p^n,\theta_{\lambda}^n\})}{\norm{\vh}_{H^{1}(\Gah)}}.
        \end{aligned}
    \end{equation}

Considering now the error equation \eqref{eq: error equation UNS}, solving for $\bhtil(\vh,\{\theta_p^n,\theta_{\lambda}^n\})$, using the Ritz-Stokes projection \eqref{eq: surface Ritz-Stokes projection UNS}, and the velocity error estimate \eqref{eq: Velocity Error Estimates UNS} in \Cref{theorem: Velocity Error Estimates UNS}, along with the bounds \eqref{eq: Interpolation errors UNS}--\eqref{eq: Interpolation error 2 UNS} in \Cref{lemma: Interpolation errors UNS}, \eqref{eq: Trilinear Errors general 1 UNS} in \Cref{lemma: Trilinear Errors UNS}, and  \eqref{eq: Consistency error 1 UNS}--\eqref{eq: Consistency error 4 UNS} in \Cref{lemma: Consistency errors UNS}, we have the following 
\begin{equation}
    \begin{aligned}\label{eq: disc remaninder pressure inside 1 UNS}
        &\bhtil(\vh,\{\theta_p^n,\theta_{\lambda}^n\})\\
    &\leq |(\frac{\thbfu^n - \thbfu^{n-1}}{\Delta t},\vh)_{L^2(\Gah)}|+ \norm{\thbfu^n}_{\ah}\norm{\vh}_{\ah} + C\Big( \frac{c}{\sqrt{\Delta t}}\norm{\partial_t \rho_\bfu(t,\cdot)}_{L^2(I_n;L^2(\Gah))}\norm{\vh}_{L^2(\Gah)}  \\
    &\ + c(h^{k_{pr}+1} + h^{k_{\lambda}+1})\norm{\vh}_{\ah} + ch^{k_g}\norm{\vh}_{H^1(\Gah)}\Big) + \Big(\sqrt{\Delta t} \norm{\partial_{tt}\bfu}_{L^2(I_n;L^2(\Ga))} \\
    &\ + \sqrt{\Delta t}\norm{\partial_t\bfu}_{L^2(I_n;L^4(\Ga))}\norm{\bfu^n}_{H^1(\Ga)}\Big)\norm{\vh}_{\ah} +\big|\bhtil (\vh,\{\mathcal{P}_h(\bfu),\mathcal{L}_h(\bfu)\})\big|.
    \end{aligned}
\end{equation}


We need to bound the first and last term. For the first term, similar to \Cref{Lemma: Pressure stab Estimate UNS}, since $\thbfu^n - \thbfu^{n-1}\in\bfV_h^{div}$, with the use of \eqref{eq: dual estimate UNS} and the definition of the discrete inverse Stokes operator $\mathcal{A}_h$ \eqref{eq: Discrete inverse Stokes UNS} we have that
\begin{equation}
    \begin{aligned}\label{eq: disc remaninder pressure inside 2 UNS}
        \sup_{\vh\in\bfV_h}\frac{|(\frac{\thbfu^n - \thbfu^{n-1}}{\Delta t},\vh)_{L^2(\Gah)}|}{\norm{\vh}_{H^1(\Gah)}} \leq c\norm{\frac{\mathcal{A}_{h}(\thbfu^n - \thbfu^{n-1})}{\Delta t}}_{\ah}.
    \end{aligned}
\end{equation}
For the last term, the bound \eqref{bhtilde boundedness UNS}, the super approximation estimates \eqref{eq: super-approximation estimate 2 UNS}, and \eqref{eq: super-approximation stability UNS}, the modified Ritz-Stokes interpolation estimates \eqref{eq: Error Bounds Ritz-Stokes improved UNS}, \eqref{eq: Error Bounds Ritz-Stokes improved L2 Lh UNS} and the $H_h^{-1}$-norm definition \eqref{eq: H^-1h definition UNS}, coupled with the geometric approximations \eqref{eq: geometric errors 1 UNS},
yield the following
\begin{equation}
    \begin{aligned}\label{eq: disc remaninder pressure inside 3 UNS}
        &\big|\bhtil (\vh,\{\mathcal{P}_h(\bfu),\mathcal{L}_h(\bfu)\})\big| \leq \norm{\vh}_{\ah}\norm{\mathcal{P}_h(\bfu)}_{L^2(\Gah)} + \Big|\int_{\Gah}\vh\cdot(\nh-\bfn)\mathcal{L}_h(\bfu) \Big|\\
        &\quad +  \Big|\int_{\Gah}\Ihz(\vh\cdot\bfn)\mathcal{L}_h(\bfu)  + \int_{\Gah}\big(\Ihz(\vh\cdot\bfn)-\vh\cdot\bfn\big)\mathcal{L}_h(\bfu)\Big| \\
        &\leq \norm{\vh}_{\ah}\norm{\mathcal{P}_h(\bfu)}_{L^2(\Gah)} + ch\norm{\vh}_{L^2(\Gah)}\norm{\mathcal{L}_h(\bfu)}_{L^2(\Gah)} + \norm{\vh}_{H^1(\Gah)}\norm{\mathcal{L}_h(\bfu)}_{H_h^{-1}(\Gah)}\\
        &\leq ch^{\widehat{r}_u}\norm{\bfu}_{H^{k_u+1}}\norm{\vh}_{H^1(\Gah)}.
        \end{aligned}
\end{equation}
Now we substitute \eqref{eq: disc remaninder pressure inside 1 UNS}-\eqref{eq: disc remaninder pressure inside 3 UNS} into \eqref{eq: Pressure conv Estimate UNS inside inf-sup UNS}, and factor out the term $\norm{\vh}_{H^1(\Gah)}$ (since $\norm{\vh}_{\ah} \leq \norm{\vh}_{H^1(\Gah)}$). Then, squaring and applying the operator $\Delta t \sum_{k=1}^n$, along with the velocity estimates \eqref{eq: discrete remainder velocity estimates UNS} and the auxiliary convergence result \eqref{eq: Auxiliary convergence result imrpoved UNS}, we obtain our final result. 
\end{proof}

Now, by the decompositions \eqref{eq: decomposition error 2 lagrange UNS}, \eqref{eq: decomposition error 3 lagrange UNS}, the interpolation bounds \eqref{eq: interpolation errors pressures UNS}, and a simple application of the triangles inequality one may then derive the following estimate for the pressure errors.
\begin{theorem}[Pressures Error Estimate] \label{theorem: Pressures Error Estimate UNS}
Under the Assumptions  \ref{assumption: Regularity assumptions for velocity estimate} and \Cref{lemma: discrete remainder pressure estimates improved UNS} the following pressure error estimates holds for $1 \leq n \leq N$\vspace{-1mm}
\begin{equation}\label{eq: Pressure Error Estimates improved UNS}
    \begin{aligned}
       \Delta t \sum_{k=1}^{n} \norm{\{\ep^k,\el^k\}}_{L^2(\Gah)\times H_h^{-1}(\Gah)}^2 \leq C\big((\Delta t)^2 + h^{2\widehat{r}_u} + h^{2k_{pr}+2} +h^{2k_{\lambda}+2}\big),
    \end{aligned}
\end{equation}\vspace{-1mm}
where $\widehat{r}_u = min\{k_{u},k_g\}$ with $C$ the constant in \Cref{lemma: discrete remainder velocity estimates UNS}.   
\end{theorem}

\subsection{Pressure a-priori estimates for $\underline{k_\lambda = k_u-1}$}\label{sec: Pressure a-priori estimates for kl= ku-1 UNS}
Now we want to prove a-priori estimates when $k_\lambda=k_u-1$. Let us briefly discuss the differences between this case and \cref{sec: pressure a-priori kl=ku UNS}.

First, as we saw in \cref{remark: About the pressure stability for arbitrary convergence UNS} and will also observe from the a-priori pressure estimates later on, we require \emph{further regularity assumptions}. Second, since we cannot follow the steps in \cref{sec: pressure a-priori kl=ku UNS} because \cref{lemma: dual estimate UNS} does not hold for $k_\lambda=k_u-1$ (see \cref{remark: Lower Regularity convergence estimate UNS} for further details where one may prove a similar estimate in that case) instead of calculating an estimate with the help of the \emph{inverse Stokes operator} as in \cref{lemma: Auxiliary convergence result UNS}, we prove an error bound for the approximation of the time derivative in the following norm $$ \Delta t\sum_{n=1}^n\norm{\frac{\thbfu^n-\thbfu^{n-1}}{\Delta t}}_{L^2(\Gah)}^2.$$
As we shall see, such estimate is not immediately attainable by testing the error equation \eqref{eq: error equation UNS} with $\thbfu^n-\thbfu^{n-1}$, since terms involving test function in the energy norm $\norm{\cdot}_{\ah}$, see for example \eqref{eq: Interpolation error 2 UNS}, \eqref{eq: Trilinear Errors general 2 UNS} and \eqref{eq: Consistency error 3 UNS} would give suboptimal order of convergence. 

For that reason, we seek bounds for the consistency, interpolation and inertia errors such that the test functions are controlled by the $L^2$-norm instead. To achieve this, first, we now consider the standard Ritz-Stokes projection \eqref{eq: surface Ritz-Stokes projection std UNS}, thus removing the problematic bound \eqref{eq: Interpolation error 2 UNS} involving the bilinear forms $\bhtil$. 
We also consider new initial condition for our discrete scheme \eqref{eq: weak lagrange fully discrete UNS}; $\uh^0 = \mathcal{R}_h^b\bfu^0= \mathcal{R}_h(\bfu^0,\{p^0,\lambda^0\})$. This is so that at time $t_0=0$ the \emph{discrete remainder} in \eqref{eq: decomposition error 1 Ritz 2 UNS} $\sigma_{\bfu}^0=0$; see \cref{remark: about initial conditions UNS}. Second, for the inertia terms \eqref{eq: Trilinear Errors general 2 UNS} we will have to apply integration by parts \eqref{eq: integration by parts UNS}, \eqref{eq: integration by parts cont UNS}. All these changes, compared to \cref{sec: pressure a-priori kl=ku UNS}, come at the expense of higher regularity assumptions, which are the following:
\begin{assumption}[Regularity assumptions II]\label{assumption: Regularity assumptions for velocity estimate 2 UNS}
We assume that the solution $(\bfu,\{p,\lambda\})$ has the following regularity   
\begin{equation*}
    \begin{aligned}
    &\{ p,\lambda\} \in W^{1,\infty}([0,T];H^{k_{pr}+1}(\Ga))\times W^{1,\infty}([0,T];H^{k_{\lambda}+1}(\Ga)),\\
        &\bfu \in H^2([0,T];L^2(\Ga))\cap W^{1,\infty}([0,T]; W^{1,\infty}(\Ga))\cap L^{\infty}([0,T]; W^{2,\infty}(\Ga))\cap  W^{1,\infty}([0,T];H^{k_u+1}(\Ga)),
    \end{aligned}
\end{equation*}
\end{assumption}

\noindent Considering the standard Ritz-Stokes projection \eqref{eq: Error Bounds Ritz-Stokes std UNS} we use the following new decompositions:
\begin{align}\label{eq: decomposition error 1 Ritz 2 UNS}
    \eu^{n} = \bfu^n - \uh^n &= \underbrace{(\bfu^n- \mathcal{R}_h^b(\bfu^n)) }_{\text{Interpolation error}} + \underbrace{(\mathcal{R}_h^b(\bfu^n) - \uh^n)}_{\text{discrete remainder}} := \eta_{\bfu}^n + \sigma_{\bfu}^n , \\
    \label{eq: decomposition error 2 Ritz 2 UNS}
     \{\ep^n, \el^n\} = \{p^n,\lambda^n\} - \{\ph^n,\lambda_h^n\} &= \underbrace{(\{p^n,\lambda^n\}-  \{\mathcal{P}_h^b(\bfu^n),\mathcal{L}_h^b(\bfu^n)\}}_{\text{Interpolation error}} + \underbrace{( \{\mathcal{P}_h^b(\bfu^n),\mathcal{L}_h^b(\bfu^n)\} - \{\ph,\lambda_h\})}_{\text{discrete remainder}}\nonumber \\
     &:= \eta_{\{p,\lambda\}}^n + \sigma_{\{p,\lambda\}}^n.
\end{align}
The bounds on the Ritz-Stokes interpolation error have already been proved in \cref{lemma: Error Bounds Ritz-Stokes std UNS}. Due to the new pressure decompositions \eqref{eq: decomposition error 2 Ritz 2 UNS}, and the new initial condition $\uh^0 = \mathcal{R}_h^b\bfu^0$, the new error equation is the following.
\begin{lemma}
   Let $(\vh,\{\qh,\xi_h\}) \in \bfV_h\times (Q_h\times \Lambda_h)$ be test functions, then for $\sigma_{\bfu},\, \sigma_{\{p,\lambda\}}$ as in  \eqref{eq: decomposition error 1 Ritz 2 UNS}, \eqref{eq: decomposition error 2 Ritz 2 UNS} the following error equations hold true $n \geq 1$
   \begin{align}
\begin{cases}\label{eq: error equation 2 UNS}
       (\frac{\sigma_{\bfu}^n -\sigma_{\bfu}^{n-1}}{\Delta t},\vh )_{L^2(\Gah)} + \ah(\sigma_{\bfu}^n,\vh)\,+\!\!\!\!\!\!&\bhtil(\vh,\sigma_{\{p,\lambda\}}^n) = \sum_{i=1}^3 \textsc{Err}_{i}^{C}(\vh) +  \textsc{Err}_{1}^{I}(\vh) + \mathcal{C}(\vh),\\
        & \bhtil(\sigma_{\bfu}^n,\{\qh,\xi_h\})= 0,
    \end{cases}
\end{align}
for all $\vh \in \bfV_h$ and $\{\qh,\xi_h\}\in Q_h\times \Lambda_h$, with the consistency errors:
\vspace{2mm}

\begin{minipage}[t]{9cm}
    \begin{itemize}
        \item[\textbullet\hspace{9mm}] \hspace{-9mm}$\textsc{Err}_1^{C}(\vh) := (\frac{\bfu^n - \bfu^{n-1}}{\Delta t},\vh )_{L^2(\Gah)} - (\partial_t\bfu^n,\vhl)_{L^2(\Ga)}$
        \item[\textbullet\hspace{9mm}] \hspace{-9mm}$ \textsc{Err}_2^{C}(\vh) := \chtil(\bfu^{n-1};\bfu^n,\vh)  - c(\bfu^n;\bfu^n,\vhl)$
    \end{itemize}
\end{minipage}
\begin{minipage}[t]{8cm}
    \begin{itemize}
        \item[\textbullet\hspace{12mm}] \hspace{-14mm} $ \textsc{Err}_3^{C}(\vh) := (\bff^n,\vhl)_{L^2(\Ga)}-(\bff_h^n,\vh)_{L^2(\Gah)}$
    \end{itemize}
\end{minipage}

\vspace{2mm}
\noindent and interpolations errors:
\vspace{2mm}

\begin{minipage}[t]{9cm}
    \begin{itemize}
        \item[\textbullet\hspace{9mm}] \hspace{-9mm}$ \textsc{Err}_1^{I}(\vh) := - (\frac{\eta_{\bfu}^n - \eta_{\bfu}^{n-1}}{\Delta t},\vh )_{L^2(\Gah)}$
    \end{itemize}
\end{minipage}
\begin{minipage}[t]{8cm}
    \begin{itemize}
        \item[\textbullet\hspace{19mm}] \hspace{-21mm} $ \mathcal{C}(\vh) := - \chtil(\eu^{n-1};\bfu^n,\vh) - \chtil(\uh^{n-1};\eu^{n},\vh).$
    \end{itemize}
\end{minipage}
\end{lemma}
\begin{proof}
    The proof follows exactly as in \Cref{lemma: error equation UNS}, where we utilize the Ritz-Stokes projection as described in \eqref{eq: surface Ritz-Stokes projection std UNS}. This time, due to \eqref{eq: decomposition error 2 Ritz 2 UNS}, the consistency and interpolation errors related to the bilinear form $\bhtil(\cdot,\{\cdot,\cdot\})$ cancel out, leaving us with the remaining terms unchanged.
\end{proof}
We can readily see from \cref{sec: velocity a-priori estimates UNS}, and specifically following similar calculations to \cref{lemma: discrete remainder velocity estimates kl=ku-1 UNS} and \cref{theorem: Velocity Error Estimates UNS} that the following error estimates for the velocity hold:
\begin{align}
\label{eq: discrete remainder velocity estimates kl=ku-1 HR UNS}
\norm{\sigma_{\bfu}^n}_{L^2(\Gah)}^2 + \Delta t \sum_{k=1}^n \norm{\sigma_{\bfu}^k}_{\ah}^2 &\leq Cexp(ct_n)\big((\Delta t)^2 + h^{2r_u}\big),\\
\label{eq: Velocity Error Estimates kl=ku-1 HR UNS}
     \norm{\eu^n}_{L^2(\Gah)}^2 + \Delta t \sum_{k=1}^n \norm{\eu^k}_{\ah}^2  &\leq Cexp(ct_n)\big((\Delta t)^2 + h^{2r_u} + h^{2k_{pr}+2} +h^{2k_{\lambda}+2}\big),
\end{align}
where for the last estimate we used the standard Ritz-Stokes interpolation estimates \eqref{eq: Error Bounds Ritz-Stokes std UNS}. 
\begin{remark}\label{remark: init cond and stab}
Let us discuss, once more, the initial condition  and the condition \eqref{eq: Linfty uh bound UNS}:
\begin{enumerate}
    \item[i)] Regarding the initial condition, we chose $\uh^0 = \mathcal{R}_h^b\bfu^0$ such that $\sigma_{\bfu}^0=0$ and therefore the computations become straightforward at $t=0$. For example, the previous estimate \eqref{eq: discrete remainder velocity estimates kl=ku-1 HR UNS} is immediate from \cref{sec: velocity a-priori estimates UNS} or in the later auxiliary result \cref{lemma: Auxiliary convergence result HR UNS} no $\ah$-norm of $\sigma_{\bfu}^0$ appears after a summation. In practice, though, this is not convenient as mentioned in \cref{remark: about initial conditions UNS}. Therefore a more practical choice would have been $\uh^0 = \mathcal{R}_h\bfu^0$ and since, by \cref{def: surface Ritz-Stokes projection UNS,def: surface Ritz-Stokes projection std UNS}, the Ritz-Stokes projection is linear, it would then give that $\norm{\sigma_{\bfu}^0}_{\ah} = \norm{\mathcal{R}_h^b\bfu^0 - \mathcal{R}_h\bfu^0}_{\ah} =  \norm{\mathcal{R}_h(0,\{p^0,\lambda^0\})}_{\ah} \leq c(h^{k_{pr}+1}+h^{k_{\lambda}+1} + h^{k_g})(\norm{p}_{H^{k_{pr}+1}} + \norm{\lambda}_{H^{k_{\lambda}+1}})$, and so our convergence results would still be optimal.
    \item[ii)] By the $L^{\infty}$ stability of the standard Ritz-Stokes projection \eqref{eq: Linfty ritz stokes UNS} and \eqref{eq: discrete remainder velocity estimates kl=ku-1 HR UNS}  
    \cref{remark: About the pressure stability for arbitrary convergence UNS} still stands.
\end{enumerate}
\end{remark}

Now before moving forward, as mentioned in the beginning of the subsection, we need to re-calculate our consistency, interpolation and inertia bounds. From \cref{lemma: Interpolation errors UNS} and Ritz-Stokes estimate \eqref{eq: Error Bounds Ritz-Stokes std UNS} we have the following.
\begin{lemma}[Interpolation errors II]\label{lemma: Interpolation errors HR UNS}
Let 
\begin{align*}
    \partial_t\bfu \in L^{\infty}(I_n;(H^{k_u+1}(\Ga))^3), \, \partial_t p\in L^{\infty}(I_n;H^{k_{pr}+1}(\Ga)),\, \partial_t \lambda\in L^{\infty}(I_n;H^{k_{\lambda}+1}(\Ga)),
\end{align*}
then the following interpolation bound holds
 \begin{align}\label{eq: Interpolation errors HR UNS}
        |\textsc{Err}^{I}_1(\vh)| &\leq ch^{r_u} \sup_{t \in I_n} \big(\norm{\partial_t \bfu}_{H^{k_u+1}(\Ga)} \big)\norm{\vh}_{L^2(\Gah)} \\
        &+ c(h^{k_{pr}+1} + h^{k_{\lambda}+1}) \sup_{t \in I_n} \big(\norm{\partial_t p}_{H^{k_{pr}+1}(\Ga)} + \norm{\partial_t \lambda}_{H^{k_{\lambda}+1}(\Ga)} \big)\norm{\vh}_{L^2(\Gah)},
\end{align}
where $r_u = min\{k_u,k_g-1\}$ and constant $c>0$ independent of $h$.
\end{lemma}

\begin{lemma}[Consistency errors II]\label{lemma: Consistency errors HR UNS}
 Assume 
 $$ \partial_{tt}\bfu \in L^{2}(I_n;(L^2(\Ga))^3), \quad \partial_{t}\bfu \in L^{\infty}(I_n;W^{1,{\infty}}(\Ga)^3),$$
 then we have the following consistency bounds 
 \begin{align}\label{eq: Consistency error HR 1 UNS}
        |\textsc{Err}_1^{C}(\vh)| &\leq   c\big(\sqrt{\Delta t} \norm{\partial_{tt}\bfu}_{L^2(I_n;L^2(\Ga))}+h^{k_g+1} \norm{\partial_{t}\bfu^n}_{L^2(\Ga)}\big)\norm{\vh}_{L^2(\Gah)},\\
        \label{eq: Consistency error HR 3 UNS}
        | \textsc{Err}_2^{C}(\vh)| &\leq c\big(h^{k_g-1}\norm{\bfu^{n-1}}_{H^1(\Ga)}\norm{\bfu^{n}}_{H^1(\Ga)} + \Delta t \norm{\bfu^n}_{H^1(\Ga)} \sup_{t \in I_n}\norm{\partial_{t}\bfu}_{W^{1,{\infty}}(\Ga)}\big)\norm{\vh}_{L^2(\Gah)},\\
        | \textsc{Err}_3^{C}(\vh)| &\leq ch^{k_g+1}\norm{\bff^n}_{L^2(\Ga)}\norm{\vh}_{L^2(\Gah)},
\end{align}
for $\vh \in \bfV_h$, where constant $c>0$ independent of $h$.    
\end{lemma}
\begin{proof}
The first and third consistency terms are bounded as the consistency errors in \Cref{lemma: Consistency errors UNS}. For the second estimate, as in \Cref{lemma: Consistency errors UNS}, we split the error into two parts: one associated with the integration domain errors and another concerning the different time steps of the first argument in the inertia term:
\begin{equation}
    \begin{aligned}\label{eq: consistency inside HR 1 UNS}
        \chtil(\bfu^{n-1};\bfu^n,\vh)  - c(\bfu^n;\bfu^n,\vhl) &= \chtil(\bfu^{n-1};\bfu^n,\vhl) - c(\bfu^{n-1};\bfu^n,\vhl)\\
        &\ + c(\bfu^{n-1};\bfu^n,\vhl) - c(\bfu^n;\bfu^n,\vhl),
    \end{aligned}
\end{equation}
then \eqref{eq: consistency inside UNS} and the inverse inequality yield
\begin{equation}
    \begin{aligned}\label{eq: consistency inside HR UNS}
        \chtil(\bfu^{n-1};\bfu^n,\vhl) - c(\bfu^{n-1};\bfu^n,\vhl) \leq ch^{k_g-1}\norm{\bfu^{n-1}}_{H^1(\Ga)}\norm{\bfu^n}_{H^1(\Ga)}\norm{\vh}_{L^2(\Gah)}.
    \end{aligned}
\end{equation}    
Moving onto the second line of \eqref{eq: consistency inside HR 1 UNS}, deferring from  \Cref{lemma: Consistency errors UNS} we have that
\begin{equation}
    \begin{aligned}\label{eq: consistency inside 1 HR UNS}
        |c(\bfu^{n-1};\bfu^n,\vhl) - c(\bfu^n;\bfu^n,\vhl)| &= \frac{1}{2} \underbrace{\int_{\Ga } ((\bfu^{n} \cdot \nbgcov)\bfu^{n} ) \cdot \vhl \; \ds - \int_{\Ga } ((\bfu^{n-1} \cdot \nbgcov)\bfu^{n} ) \cdot \vhl \; \ds}_{\textbf{I}c} \\
        &\ -\frac{1}{2} \underbrace{\int_{\Ga } ((\bfu^{n} \cdot \nbgcov)\bfPg\vhl )\cdot  \bfu^{n} \; \ds
        + \int_{\Ga } ((\bfu^{n-1} \cdot \nbgcov)\bfPg\vhl ) \cdot \bfu^{n} \; \ds}_{\textbf{II}c}.
    \end{aligned}
\end{equation}
We can bound the initial line on the right-hand side as:
\begin{equation}
    \begin{aligned}\label{eq: consistency inside 5 HR UNS}
         \textbf{I}c = \int_{\Ga} (\Big(\int_{t_{n-1}}^{t^n} \partial_t\bfu(t,\cdot) dt\Big)\cdot \nbgcov)\bfu^{n} ) \cdot \vhl \; \ds \leq \Delta t  \sup_{t \in I_n}\norm{\partial_{t}\bfu}_{L^{\infty}(\Ga)}\norm{\bfu^n}_{H^1(\Ga)}\norm{\vh}_{L^2(\Gah)}.
    \end{aligned}
\end{equation}
For the second estimate $\textbf{II}c$ we perform a simple integration by parts \eqref{eq: integration by parts cont UNS} so that we may bound the test function w.r.t. the $L^2$-norm. Let us consider the first integral of $\textbf{II}c$, since the second follows similarly. Recalling that $\bfu^n$ is tangent as the solution of the continuous problem \eqref{weak lagrange hom NV}, we obtain
\begin{equation}
    \begin{aligned}\label{eq: consistency inside 2 HR UNS}
        \int_{\Ga } ((\bfu^{n} \cdot \nbgcov)\bfPg\vhl )\cdot \bfu^{n} \; \ds &= \int_{\Ga} \bfu^n \cdot (\nbg(\vhl \cdot \bfu^n))  \; \ds -\int_{\Ga } ((\bfu^{n} \cdot \nbg)\bfu^{n} ) \cdot \bfPg \vhl  \; \ds\\
        & = - \int_{\Ga} \divg(\bfu^n) (\vhl \cdot \bfu^n)  \; \ds  -\int_{\Ga } ((\bfu^{n} \cdot \nbgcov)\bfu^{n} ) \cdot \vhl  \; \ds.
    \end{aligned}
\end{equation}
So considering \eqref{eq: consistency inside 2 HR UNS} it is clear that
\begin{equation}
    \begin{aligned}\label{eq: consistency inside 3 HR UNS}
       \textbf{II}c = \textbf{I}c +  \underbrace{\int_{\Ga} \divg(\bfu^n) (\vhl \cdot \bfu^n)  \; \ds - \int_{\Ga} \divg(\bfu^{n-1}) (\vhl \cdot \bfu^n)  \; \ds}_{\textbf{III}c},
    \end{aligned}
\end{equation}
where now, as before, we are able to bound the last term as followed
\begin{equation}
    \begin{aligned}\label{eq: consistency inside 4 HR UNS}
        \textbf{III}c = \int_{\Ga} (\Big(\int_{t_{n-1}}^{t^n} \partial_t\divg(\bfu(t,\cdot)) dt\Big)\cdot \nbgcov)\bfu^{n} ) \cdot \vhl \; \ds \leq \Delta t  \sup_{t \in I_n}\norm{\partial_{t}\bfu}_{W^{1,{\infty}}(\Ga)}\norm{\bfu^n}_{H^1(\Ga)}\norm{\vh}_{L^2(\Gah)}.
    \end{aligned}
\end{equation}
Combining  \eqref{eq: consistency inside 4 HR UNS} into \eqref{eq: consistency inside 3 HR UNS}, considering the bound \eqref{eq: consistency inside 5 HR UNS} and plugging everything into \eqref{eq: consistency inside 1 HR UNS} we get our desired estimate. 
\end{proof}

\begin{lemma}[Inertia error II]\label{lemma: Trilinear Errors HR UNS}
Given that $$\bfu \in L^{\infty}(I_n;(W^{2,\infty}(\Ga))^3),$$ it follows that the following bound holds true, with constant $c>0$  independent of $h$,
\begin{equation}
    \begin{aligned}\label{eq: Trilinear Errors general HR 1 UNS}
        |\mathcal{C}(\vh)| &\leq  c\norm{\bfu^n}_{W^{2,\infty}(\Ga)} \big(\norm{\eu^{n-1}}_{L^2(\Gah)} + \norm{\eu^{n}}_{L^2(\Gah)} + \norm{\eu^{n-1}}_{\ah}+  \norm{\eu^{n}}_{\ah}\big) \norm{\vh}_{L^2(\Gah)} \\
        &\ + c\norm{\sigma_{\bfu}^{n-1}}_{L^2(\Gah)}^{1/2}\norm{\sigma_{\bfu}^{n-1}}_{\ah}^{1/2}\norm{\eu^{n}}_{\ah}\norm{\vh}_{L^2(\Gah)}^{1/2}\norm{\vh}_{\ah}^{1/2}\\
    &\ +c\norm{\sigma_{\bfu}^{n-1}}_{L^2(\Gah)}^{1/2}\norm{\sigma_{\bfu}^{n-1}}_{\ah}^{1/2}\norm{\eu^{n}}_{L^2(\Gah)}^{1/2}\norm{\eu^{n}}_{\ah}^{1/2}\norm{\vh}_{\ah}.
    \end{aligned}
\end{equation}
\end{lemma}
\begin{proof}
Recall that
\begin{equation}\label{eq: Trilinear Errors HR inside start UNS}
    \mathcal{C}(\vh) = - \underbrace{\tilde{c}_h(\eu^{n-1};\bfu^n,\vh)}_{\textbf{C}_1} -  \underbrace{\tilde{c}_h(\uh^{n-1};\eu^{n},\vh)}_{\textbf{C}_2}.
\end{equation}
Starting with the first term on the right-hand side, we have 
\begin{equation}
    \begin{aligned}\label{eq: Trilinear Errors HR inside 1 UNS}
      \textbf{C}_1 &=  \frac{1}{2} \Big( \int_{\Gah} ((\eu^{n-1} \cdot \nbgcovh)\bfu^n ) \vh \; \dsh - \int_{\Gah} ((\eu^{n-1} \cdot \nbgcovh)\vh ) \bfu^n \; \dsh \Big)\\
      & - \frac{1}{2} \Big( \int_{\Gah} (\bfu^n\cdot\nh)\eu^{n-1}\cdot\bfH_h\vh\; \dsh -  \int_{\Gah} (\vh\cdot\nh)\eu^{n-1}\cdot\bfH_h\bfu^n \; \dsh  \Big).
    \end{aligned}
\end{equation}
The last two term are easily bounded by 
\begin{equation}\label{eq: Trilinear Errors HR inside H UNS}
    \begin{aligned}
   \Big| \int_{\Gah} (\bfu^n\cdot\nh)\eu^{n-1}\cdot\bfH_h\vh\; \dsh -  \int_{\Gah} (\vh\cdot\nh)\eu^{n-1}\cdot\bfH_h\bfu^n \; \dsh \Big| \\
   \leq c\norm{\bfu^n}_{L^{\infty}(\Ga)}\norm{\eu^{n-1}}_{L^2(\Gah)}\norm{\vh}_{L^2(\Gah)},
   \end{aligned}
\end{equation}
with the help of the $L^{\infty}$ bounds of the Weingarten map $\bfH_h$ i.e. $\norm{\bfH_h}_{L^{\infty}(\Gah)}\leq c$; see \eqref{eq: geometric errors 1 UNS}. Due to our assumed regularity the first integral on the first line of \eqref{eq: Trilinear Errors HR inside 1 UNS} is bounded by 
\begin{equation}
    \begin{aligned}\label{eq: Trilinear Errors HR inside 2 UNS}
         \int_{\Gah} ((\eu^{n-1} \cdot \nbgcovh)\bfu^n ) \vh \; \dsh \leq \norm{\bfu^n}_{W^{1,\infty}(\Ga)}\norm{\eu^{n-1}}_{L^2(\Gah)}\norm{\vh}_{L^2(\Gah)},
    \end{aligned}
\end{equation}
where we have also used the norm equivalence \eqref{eq: norm equivalence UNS}. For the second integral in \eqref{eq: Trilinear Errors HR inside 1 UNS}, we want our test function in the $L^2$-norm, as in \eqref{eq: Trilinear Errors HR inside 2 UNS}. For that, recall first that $\norm{\bfP- \bfPh}_{L^\infty(\Gah)} \leq ch^{k_g}$, cf. \eqref{eq: geometric errors 1 UNS}, and the inverse inequality $\norm{\vh}_{H^1(\Gah)} \leq ch^{-1}\norm{\vh}_{L^2(\Gah)}$, then we readily see that  
\begin{equation}
    \begin{aligned}\label{eq: Trilinear Errors HR inside 3 UNS}
      &\int_{\Gah} ((\eu^{n-1} \cdot \nbgcovh)\vh ) \cdot \bfu^n \; \dsh = \int_{\Gah} ((\eu^{n-1} \cdot \nbgh)\vh ) \cdot \bfPh\bfu^n \; \dsh  \\
      &\qquad \qquad\leq ch^{k_g-1}\norm{\eu^{n-1}}_{L^2(\Gah)}\norm{\vh}_{L^2(\Gah)}\norm{\bfu^n}_{L^\infty(\Gah)}+ \int_{\Gah}((\eu^{n-1} \cdot\nbgh)\vh )  \cdot \bfP\bfu^n \; \dsh.
    \end{aligned}
\end{equation}
Now the integration by parts formula \eqref{eq: integration by parts UNS} yields
\begin{equation}
    \begin{aligned}\label{eq: Trilinear Errors HR inside 4 UNS}
         &\int_{\Gah}((\eu^{n-1} \cdot\nbgh)\vh ) \cdot \bfP\bfu^n \; \dsh  = \int_{\Gah}\eu^{n-1} \cdot \nbgh(\vh\cdot \bfP\bfu^n) \; \dsh  - \int_{\Gah}((\eu^{n-1} \cdot \nbgh)\bfu^n ) \cdot \vh \; \dsh\\
         &=  -\int_{\Gah} \divgh(\eu^{n-1}) (\vh\cdot\bfu^n)\, \dsh - \int_{\Gah}((\eu^{n-1} \cdot \nbgh)\bfu^n ) \cdot \vh \; \dsh \\
         &+ \sum_{T\in\Th}\int_{T} \divgh(\nh)(\bfu^n\cdot\vh)(\eu^{n-1}\cdot\nh) \,\dsh + \sum_{E \in \mathcal{E}_h} \int_E [\mh] \cdot \eu^{n-1} (\vh\cdot\bfu^n) \, d\ell.
    \end{aligned}
\end{equation}
Therefore, using the co-normal bound \eqref{eq: geometric errors 1 UNS}, $\norm{\cdot}_{L^2(\mathcal{E}_h)} \leq h^{-1/2}\norm{\cdot}_{L^2(\Gah)}$ (see proof of \Cref{Lemma: discrete bounds and coercivity results}) and the fact that $k_g \geq 2$ we obtain
\begin{equation}
    \begin{aligned}\label{eq: Trilinear Errors HR inside 5 UNS}
        \big|\int_{\Gah}((\eu^{n-1} \cdot\nbgh)\vh )  \bfP\bfu^n \; \dsh \Big| \leq  c\norm{\bfu^n}_{W^{1,\infty}(\Ga)}\big( \norm{\eu^{n-1}}_{L^2(\Gah)}+\norm{\eu^{n-1}}_{\ah} \big)\norm{\vh}_{L^2(\Gah)}.
    \end{aligned}
\end{equation}
So combining  \eqref{eq: Trilinear Errors HR inside H UNS}-\eqref{eq: Trilinear Errors HR inside 5 UNS} yields that \eqref{eq: Trilinear Errors HR inside 1 UNS} is bounded by
\begin{equation}\label{eq: Trilinear Errors HR inside C1 UNS}
    \textbf{C}_1 \leq c\norm{\bfu^n}_{W^{1,\infty}(\Ga)} \big(\norm{\eu^{n-1}}_{L^2(\Gah)} + \norm{\eu^{n-1}}_{\ah}\big) \norm{\vh}_{L^2(\Gah)}.
\end{equation}

We still have to bound $\textbf{C}_2$. We are going to apply an analogous idea to the one used to bound the previous estimate. But first let us rewrite it in the following way
\begin{equation}
    \begin{aligned}\label{eq: Trilinear Errors HR inside 7 UNS}
        \textbf{C}_2 = \chtil(\uh^{n-1};\eu^{n},\vh) = -\chtil(\sigma_{\bfu}^{n-1};\eu^{n},\vh) - \chtil(\mathcal{R}_h^b \bfu^{n-1};\eu^{n},\vh),
    \end{aligned}
\end{equation}
where $\sigma_{\bfu}^{n-1} = \mathcal{R}_h^b\bfu^{n-1} - \uh^{n-1}$ as in the first decomposition \eqref{eq: decomposition error 1 Ritz 2 UNS} and  $\mathcal{R}_h^b\bfu^{n-1}$ the standard Ritz-Stokes projection \eqref{eq: surface Ritz-Stokes projection std UNS}. Now, the first term is bounded, thanks to \eqref{ch boundedness UNS},  as followed
\begin{equation}\label{eq: Trilinear Errors HR inside 11 UNS}
\begin{aligned}
        \chtil(\sigma_{\bfu}^{n-1};\eu^{n},\vh) &\leq c\norm{\sigma_{\bfu}^{n-1}}_{L^2(\Gah)}^{1/2}\norm{\sigma_{\bfu}^{n-1}}_{\ah}^{1/2}\norm{\eu^{n}}_{\ah}\norm{\vh}_{L^2(\Gah)}^{1/2}\norm{\vh}_{\ah}^{1/2}\\
    &\ +\norm{\sigma_{\bfu}^{n-1}}_{L^2(\Gah)}^{1/2}\norm{\sigma_{\bfu}^{n-1}}_{\ah}^{1/2}\norm{\eu^{n}}_{L^2(\Gah)}^{1/2}\norm{\eu^{n}}_{\ah}^{1/2}\norm{\vh}_{\ah}.
    \end{aligned}
\end{equation}
The second term, by \eqref{eq: skew-symmetrized ch real Ph} and \Cref{remark: inertia term Ph} is equal to 
\begin{equation}
    \begin{aligned}\label{eq: Trilinear Errors HR inside 8 UNS}
       \chtil(\mathcal{R}^b_h \bfu^{n-1};\eu^{n},\vh) &=  \frac{1}{2} \Big( \int_{\Gah} ((\mathcal{R}^b_h \bfu^{n-1} \cdot \nbgcovh)\eu^{n} ) \vh \; \dsh  - \int_{\Gah} ((\mathcal{R}^b_h \bfu^{n-1} \cdot \nbgcovh)\vh ) \eu^{n} \; \dsh \Big)\\
      & - \frac{1}{2} \Big(  \int_{\Gah} (\eu^n\cdot\nh)\mathcal{R}^b_h \bfu^{n-1}\cdot \bfH_h\vh \; \dsh -  \int_{\Gah} (\vh\cdot\nh)\mathcal{R}^b_h \bfu^{n-1}\cdot \bfH_h\eu^n \; \dsh\Big).
    \end{aligned}
\end{equation}
Due to the $L^{\infty}-$bound of the Ritz-Stokes projection \eqref{eq: W1infty estimate Ritz-Stokes UNS} we have
\begin{align}\label{eq: Trilinear Errors HR inside H 2 UNS}
        &\big| \int_{\Gah} (\eu^n\cdot\nh)\mathcal{R}^b_h \bfu^{n-1}\cdot \bfH_h\vh \; \dsh - \int_{\Gah} (\vh\cdot\nh)\mathcal{R}^b_h \bfu^{n-1}\cdot \bfH_h\eu^n \; \dsh\Big|\\
        &\qquad\qquad\qquad\qquad\qquad\qquad\qquad\qquad\qquad\qquad\quad\leq c\norm{\bfu^{n-1}}_{W^{2,\infty}(\Ga)} \norm{\eu^n}_{L^2(\Gah)}\norm{\vh}_{L^2(\Gah)}, \nonumber\\
        \label{eq: Trilinear Errors HR inside 9 UNS}
        &\int_{\Gah} ((\mathcal{R}^b_h \bfu^{n-1} \cdot \nbgcovh)\eu^{n} ) \vh \; \dsh \leq c\norm{\bfu^{n-1}}_{W^{2,\infty}(\Ga)}\norm{\eu^n}_{\ah}\norm{\vh}_{L^2(\Gah)}.
\end{align}
It remains to bound the second in \eqref{eq: Trilinear Errors HR inside 8 UNS}. This follows exactly as in the previous calculations \eqref{eq: Trilinear Errors HR inside 3 UNS} - \eqref{eq: Trilinear Errors HR inside 5 UNS}, where now using the $W^{1,\infty}-$bound of the Ritz-Stokes projection \eqref{eq: W1infty estimate Ritz-Stokes UNS} we readily see that
\begin{equation}
    \begin{aligned}\label{eq: Trilinear Errors HR inside 10 UNS}
      &\int_{\Gah} ((\mathcal{R}^b_h \bfu^{n-1} \cdot \nbgcovh)\vh ) \eu^{n} \; \dsh \leq  c\norm{\bfu^{n-1}}_{W^{2,\infty}(\Ga)}\big( \norm{\eu^{n}}_{L^2(\Gah)} + \norm{\eu^n}_{\ah}\big) \norm{\vh}_{L^2(\Gah)}.
    \end{aligned}
\end{equation}
Combining now \eqref{eq: Trilinear Errors HR inside 11 UNS}-\eqref{eq: Trilinear Errors HR inside 10 UNS} yields that \eqref{eq: Trilinear Errors HR inside 7 UNS} is bounded by
\begin{equation}
\begin{aligned}\label{eq: Trilinear Errors HR inside C2 UNS}
        \textbf{C}_2 &\leq c\norm{\bfu^n}_{W^{2,\infty}(\Ga)} \big(\norm{\eu^{n}}_{L^2(\Gah)} + \norm{\eu^{n}}_{\ah}\big) \norm{\vh}_{L^2(\Gah)}\\
        &\ +c\norm{\sigma_{\bfu}^{n-1}}_{L^2(\Gah)}^{1/2}\norm{\sigma_{\bfu}^{n-1}}_{\ah}^{1/2}\norm{\eu^{n}}_{\ah}\norm{\vh}_{L^2(\Gah)}^{1/2}\norm{\vh}_{\ah}^{1/2}\\ &\ +\norm{\sigma_{\bfu}^{n-1}}_{L^2(\Gah)}^{1/2}\norm{\sigma_{\bfu}^{n-1}}_{\ah}^{1/2}\norm{\eu^{n}}_{L^2(\Gah)}^{1/2}\norm{\eu^{n}}_{\ah}^{1/2}\norm{\vh}_{\ah}.
    \end{aligned}
\end{equation}
To obtain the final result we simply merge the two estimates $\textbf{C}_1$ \eqref{eq: Trilinear Errors HR inside C1 UNS} and $\textbf{C}_2$ \eqref{eq: Trilinear Errors HR inside C2 UNS} into \eqref{eq: Trilinear Errors HR inside start UNS}.
\end{proof}

In contrast to \Cref{lemma: Auxiliary convergence result UNS} we now prove the following stronger result, without the help of an inverse Stokes operator.

\begin{lemma}[Auxiliary result II]\label{lemma: Auxiliary convergence result HR UNS}
Assume $\underline{k_\lambda = k_u-1}, \, k_g\geq 2$ and that the regularity \cref{assumption: Regularity assumptions for velocity estimate 2 UNS} hold with time-step  $\Delta t \leq ch$. Let $(\bfu,\{p,\lambda\})$ be the solution of \eqref{eq: unstead NV Lagrange}  and let $\uh^k$ and $\{\ph^k,\lh^k\}$, $k=1,...,n$ be the discrete solutions of \eqref{eq: weak lagrange fully discrete UNS} with initial condition $\uh^0=\mathcal{R}_h^b\bfu^0=\mathcal{R}_h(\bfu^0,\{p^0,\lambda^0\})$. Then the following estimate holds for $1 \leq n \leq N$
\begin{equation}
    \begin{aligned}\label{eq: Auxiliary convergence result HR UNS}
        \sum_{k=1}^n\Delta t \norm{\frac{(\sigma_{\bfu}^k - \sigma_{\bfu}^{k-1})}{\Delta t}}_{L^2(\Gah)}^2  + \norm{\sigma_{\bfu}^n}_{\ah}^2 \leq C_2\big(1 + C\frac{h^{2r_u}}{\Delta t }\big)\big((\Delta t)^2 &+ h^{2r_u} + h^{2(k_{pr}+1)} +h^{2(k_{\lambda}+1)}\big),
    \end{aligned}
\end{equation}
where $r_u = min\{k_{u},k_g-1\}$ with positive constant  $C_2 = C_2(u,p,\lambda,\bff) = \sup_{t\in[0,T]} \big( \norm{\partial_tp}_{H^{k_{pr}+1}(\Ga)}^2   +\norm{\partial_t\lambda}_{H^{k_{\lambda}+1}(\Ga)}^2  +\norm{\partial_t\bfu}_{H^{k_u +1}(\Ga)}^2 +  \norm{\bff}_{L^2(\Ga)} + \norm{\bfu}_{W^{2,\infty}(\Ga)}^2+ \norm{\partial_{t}\bfu}_{W^{1,{\infty}}(\Ga)}^2 \big) + \norm{\partial_{tt}\bfu}_{L^2([0,T];L^2(\Ga))}^2$ and  constant $C$ as in \eqref{eq: discrete remainder velocity estimates kl=ku-1 HR UNS}.
\end{lemma}
\begin{proof}
To start, we test the error equation \eqref{eq: error equation 2 UNS} with $\vh = \sigma_{\bfu}^n - \sigma_{\bfu}^{n-1}$, and $\{\qh,\xi_h\} = \sigma_{\{p,\lambda\}}^n$, thus we derive for $1 \leq n \leq N$, since $\sigma_{\bfu}^0 =0$ due to the choice of the initial condition, that
\begin{equation}
    \begin{aligned}\label{eq: Auxiliary convergence result inside HR UNS}
        \frac{1}{\Delta t }\norm{\sigma_{\bfu}^n - \sigma_{\bfu}^{n-1}}_{L^2(\Gah)}^2 &+ \norm{\sigma_{\bfu}^n}_{\ah}^2 - \norm{\sigma_{\bfu}^{n-1}}_{\ah}^2 + \norm{\sigma_{\bfu}^n - \sigma_{\bfu}^{n-1}}_{\ah}^2 \leq 2\Big( |\sum_{i=1}^3 \textsc{Err}_{i}^{C}(\sigma_{\bfu}^n - \sigma_{\bfu}^{n-1})| \\
        &+  | \textsc{Err}_{1}^{I}(\sigma_{\bfu}^n - \sigma_{\bfu}^{n-1})| + |\mathcal{C}(\sigma_{\bfu}^n - \sigma_{\bfu}^{n-1})|\Big).
    \end{aligned}
\end{equation} 
We want to approximate the right-hand side appropriately. This will be possible with the help of the newly established bounds. A simple use of Young's inequality gives us the following bounds:
\begin{itemize} 
\item Using \Cref{lemma: Interpolation errors HR UNS} the interpolation error is bounded as
      \begin{equation*}
        \begin{aligned}
            | \textsc{Err}_{1}^{I}(\sigma_{\bfu}^n - \sigma_{\bfu}^{n-1})| \leq C_2 \Delta t (h^{2r_u} +h^{2(k_{pr}+1)} + h^{2(k_{\lambda}+1)}) + \frac{1}{4\Delta t}\norm{\sigma_{\bfu}^n - \sigma_{\bfu}^{n-1}}_{L^2(\Gah)}^2.
        \end{aligned}
    \end{equation*}
\item  For the consistency errors considering the results in \Cref{lemma: Consistency errors HR UNS} we see that
        \begin{equation*}
            \begin{aligned}
               |\sum_{i=1}^3 \textsc{Err}_{i}^{C}(\sigma_{\bfu}^n - \sigma_{\bfu}^{n-1})| \leq  C_2  \Delta t ((\Delta t)^2 + h^{2k_g-2}) + \frac{1}{4\Delta t}\norm{\sigma_{\bfu}^n - \sigma_{\bfu}^{n-1}}_{L^2(\Gah)}^2.
            \end{aligned}
        \end{equation*}

\item Finally, using the inverse and Young's inequality appropriately, $\mathcal{C}(\cdot)$ can be bounded with the help of the estimates in \Cref{lemma: Trilinear Errors HR UNS} by
\begin{equation*}
    \begin{aligned}
        \mathcal{C}(\sigma_{\bfu}^n - \sigma_{\bfu}^{n-1}) &\leq C_2\Delta t \big(\norm{\eu^{n-1}}_{L^2(\Gah)}^2 + \norm{\eu^{n}}_{L^2(\Gah)}^2 + \norm{\eu^{n-1}}_{\ah}^2 +\norm{\eu^{n}}_{\ah}^2\big) + \frac{1}{4\Delta t}\norm{\sigma_{\bfu}^n - \sigma_{\bfu}^{n-1}}_{L^2(\Gah)}^2 \\ 
        & + ch^{-2}\norm{\sigma_{\bfu}^{n-1}}_{L^2(\Gah)}^2\Delta t \norm{\eu^n}_{\ah}^2+ \frac{1}{4\Delta t}\norm{\sigma_{\bfu}^n - \sigma_{\bfu}^{n-1}}_{L^2(\Gah)}^2\\
        &  + \frac{c}{\Delta t}\norm{\sigma_{\bfu}^{n-1}}_{L^2(\Gah)}\norm{\eu^{n}}_{L^2(\Gah)}\Delta t \big(\norm{\sigma_{\bfu}^{n-1}}_{\ah}^2 + \norm{\eu^{n}}_{\ah}^2\big) + \frac{1}{2}\norm{\sigma_{\bfu}^n - \sigma_{\bfu}^{n-1}}_{\ah}^2.
    \end{aligned}
\end{equation*}
\end{itemize}

Now combining all the above, applying the velocity error estimates in \eqref{eq: discrete remainder velocity estimates kl=ku-1 HR UNS}, \eqref{eq: Velocity Error Estimates kl=ku-1 HR UNS} and employing a kickback argument for the two terms $\frac{1}{\Delta t}\norm{\sigma_{\bfu}^n - \sigma_{\bfu}^{n-1}}_{L^2(\Gah)}^2$ and $\norm{\sigma_{\bfu}^n - \sigma_{\bfu}^{n-1}}_{\ah}^2$ and summing for $k=1,...,n$ yields
    \begin{align}
        \sum_{k=1}^n\Delta t \norm{\frac{(\sigma_{\bfu}^k - \sigma_{\bfu}^{k-1})}{\Delta t}}_{L^2(\Gah)}^2  &+ \norm{\sigma_{\bfu}^n}_{\ah}^2  + \sum_{k=1}^n\norm{\sigma_{\bfu}^k - \sigma_{\bfu}^{k-1}}_{\ah}^2
        \leq C_2\big((\Delta t)^2 + h^{2r_u} 
    + h^{2k_{pr}+2} +h^{2k_{\lambda}+2}\big) \nonumber\\
    &+ C\big(\frac{1}{\Delta t }+ch^{-2}\big)\big((\Delta t)^2 + h^{2r_u} 
    \big) \Delta t \sum_{k=1}^n (\norm{\sigma_{\bfu}^{n-1}}_{\ah}^2 + \norm{\eu^n}_{\ah}^2),  
    \end{align}
since $\sigma_{\bfu}^0 =0$ due to the choice of the initial condition, where $C$ the constant in \eqref{eq: discrete remainder velocity estimates kl=ku-1 HR UNS}. Now, since $k_u,\,k_g \geq 2$ and $\Delta t \leq ch$, another application of \eqref{eq: discrete remainder velocity estimates kl=ku-1 HR UNS}, and \eqref{eq: Velocity Error Estimates kl=ku-1 HR UNS} give our desired result.
\end{proof}

We now finish the convergence results by providing error estimates for the two pressures. First, as mentioned before, it suffices to control the discrete remainders $\{\sigma_p^n,\sigma_{\lambda}^n\}$.
\begin{lemma}\label{lemma: discrete remainder pressure estimates HR UNS}
Assume $\underline{k_\lambda = k_u-1}, \, k_g \geq 2$ and that the regularity \cref{assumption: Regularity assumptions for velocity estimate 2 UNS} hold with time-step  $\Delta t \leq ch$. Let $(\bfu,\{p,\lambda\})$ be the solution of \eqref{eq: unstead NV Lagrange}  and let $\uh^k$ and $\{\ph^k,\lh^k\}$, $k=1,...,n$ be the discrete solutions of \eqref{eq: weak lagrange fully discrete UNS} with initial condition $\uh^0=\mathcal{R}_h^b\bfu^0=\mathcal{R}_h(\bfu^0\{p^0,\lambda^0\})$. Then the following estimate holds for $1 \leq n \leq N$
\begin{equation}
    \begin{aligned} 
\label{eq: discrete remainder pressure estimates HR UNS}
            \Delta t \sum_{k=1}^{n} \norm{\sigma_{\{p,\lambda\}}^k}_{L^2(\Gah)}^2 \leq  C_2\big(1 + C\frac{h^{2r_u}}{\Delta t }\big)\big((\Delta t)^2 &+ h^{2r_u} + h^{2(k_{pr}+1)} +h^{2(k_{\lambda}+1)}\big),
    \end{aligned}
\end{equation}
where $r_u = min\{k_{u},k_g-1\}$
and constants $C_2$, $C$ the same as in \Cref{lemma: Auxiliary convergence result HR UNS} and \cref{lemma: discrete remainder velocity estimates kl=ku-1 UNS} respectively.  
\end{lemma}
\begin{proof}
The proof follows similarly to \Cref{lemma: discrete remainder pressure estimates improved UNS}, but with the new auxiliary result \Cref{lemma: Auxiliary convergence result HR UNS}. So, recalling the $L^2 \times L^2$ \textsc{inf-sup} condition \eqref{eq: discrete inf-sup condition Gah Lagrange UNS} 
\begin{equation}
        \begin{aligned}\label{eq: Pressure conv Estimate UNS inside inf-sup HR UNS}
            \norm{\sigma_{\{p,\lambda\}}^n}_{L^2(\Gah)} \leq \sup_{\vh^n \in \bfV_h} \frac{\bhtil(\vh^n,\sigma_{\{p,\lambda\}}^n)}{\norm{\vh^n}_{\ah}}.
        \end{aligned}
    \end{equation}
and using the error equation \eqref{eq: error equation 2 UNS} coupled with the estimates proved above we find
\begin{equation*}
    \begin{aligned}
        \bhtil(\vh^n,\sigma_{\{p,\lambda\}}^n) &\leq \norm{\sigma_{\bfu}^n-\sigma_{\bfu}^{n-1}}_{L^2(\Gah)}\norm{\vh}_{L^2(\Gah)} + \norm{\sigma_{\bfu}^n}_{\ah}\norm{\vh}_{\ah} \\
        &\ + C_2(\Delta t + h^{r_u} + h^{k_{pr}+1} + h^{k_{\lambda}+1})\norm{\vh}_{L^2(\Gah)}\\
        &\ +C_2(\norm{\eu^{n-1}}_{L^2(\Gah)} + \norm{\eu^{n}}_{L^2(\Gah)} + \norm{\eu^{n-1}}_{\ah}+  \norm{\eu^{n}}_{\ah}\big) \norm{\vh}_{L^2(\Gah)} \\
        &\ + c\norm{\sigma_{\bfu}^{n-1}}_{L^2(\Gah)}^{1/2}\norm{\sigma_{\bfu}^{n-1}}_{\ah}^{1/2}\norm{\eu^{n}}_{\ah}\norm{\vh}_{\ah}.
    \end{aligned}
\end{equation*}
Substituting the inequality into \eqref{eq: Pressure conv Estimate UNS inside inf-sup HR UNS} factoring out the test function $\norm{\vh}_{\ah}$, squaring and applying the operator $\Delta t \sum_{k=1}^n$, it is clear then that the velocity error convergence results \eqref{eq: discrete remainder velocity estimates kl=ku-1 HR UNS}, and \eqref{eq: Velocity Error Estimates kl=ku-1 HR UNS} and the afore-mentioned auxiliary estimate \eqref{eq: Auxiliary convergence result HR UNS} yields us our final result.
\end{proof}

Finally, by applying the triangle inequality on the decomposition \eqref{eq: decomposition error 2 Ritz 2 UNS}, the above \Cref{lemma: discrete remainder pressure estimates HR UNS} and the Ritz-Stokes interpolation error \eqref{eq: Error Bounds Ritz-Stokes std UNS} we obtain our final theorem.
\begin{theorem}[Pressures Error Estimate II]\label{theorem: pressure estimate HR UNS}
Under the Assumptions \ref{assumption: Regularity assumptions for velocity estimate 2 UNS} and \Cref{lemma: discrete remainder pressure estimates HR UNS} the following pressure error estimates holds for $1 \leq n \leq N$
\begin{equation}
    \begin{aligned}
      \Delta t \sum_{k=1}^{n} \norm{\{\ep^n,\el^n\}}_{L^2(\Gah)}^2 \leq C_2\big(1 + C\frac{h^{2r_u}}{\Delta t }\big)\big((\Delta t)^2 &+ h^{2r_u} + h^{2(k_{pr}+1)} +h^{2(k_{\lambda}+1)}\big),
    \end{aligned}
\end{equation}
where $r_u = min\{k_{u},k_g-1\}$
and constants $C_2$, $C$ the same as in \Cref{lemma: Auxiliary convergence result HR UNS} and \cref{lemma: discrete remainder velocity estimates kl=ku-1 UNS} respectively.  
\end{theorem}

\begin{remark}[Lower Regularity convergence estimate]\label{remark: Lower Regularity convergence estimate UNS}
As mentioned before, to prove the pressure error estimates for $\underline{k_\lambda =k_u-1}$ we needed \cref{assumption: Regularity assumptions for velocity estimate 2 UNS} instead of the lower regularity \cref{assumption: Regularity assumptions for velocity estimate}. This occurred because we could not follow the steps in \cref{sec: pressure a-priori kl=ku UNS} and more particularly because \cref{lemma: dual estimate UNS} does not hold in that case, which in turn means that we cannot find an estimate for the inverse Stokes operator \eqref{eq: Discrete inverse Stokes UNS} in the energy $\ah$-norm. Let us explain in more detail why \cref{lemma: dual estimate UNS} does not hold and how a change in the scheme \eqref{eq: weak lagrange fully discrete UNS} could enable us to find pressure error estimate for the low regularity \cref{assumption: Regularity assumptions for velocity estimate}:
\begin{itemize}
    \item[-] In the proof of \cref{lemma: dual estimate UNS} in \cref{appendix: proof of dual estimate UNS} we see in \eqref{eq: dual estimate inside 2 UNS} that we need an energy stability estimate for the Leray projection $\Pi_h^{div}$ \eqref{eq: discrete Leray UNS}. When $\underline{k_\lambda =k_u-1}$ this is not possible, since we cannot ``control'' the normal part appropriately, therefore only a bound of the form $\norm{\Pi_h^{div}\bfu}_{\ah} \leq c\norm{\bfu}_{H^2(\Ga)}$ is feasible.
    \item[-] If one considers instead the leray projection $(\bfP\Pi^{div}_h\bfu,\vh)_{L^2(\Gah)} = (\bfu^{-\ell},\vh)_{L^2(\Gah)}$, for all $\vh \in \bfV_h^{div}$ then it is possible to prove $\norm{\bfP\Pi_h^{div}\bfu}_{\ah} \leq c\norm{\bfu}_{H^1(\Ga)}$.
    \item[-] Going back to  \eqref{eq: dual estimate inside 2 UNS} it is then clear  that by also  defining a new inverse Stokes operator such that $\ah(\mathcal{A}_h^P \uh, \vh) = (\uh,\bfP\vh)_{L^2(\Gah)}$, for all $\vh \in \bfV_h^{div}$, we can find an estimate similar to  \eqref{eq: dual estimate UNS} in  \cref{lemma: dual estimate UNS}, w.r.t. this new $\mathcal{A}_h^P$ inverse Stokes operator instead.
    \item[-] Now, we need an auxiliary bound similar to \eqref{eq: auxilary stab bounds UNS} for $\mathcal{A}_h^P$. However, with our current scheme \eqref{eq: weak lagrange fully discrete UNS} that is not possible, since the first term of \eqref{eq: Auxiliary stab bound inside 1} cannot be controlled, unless we consider a tangential $L^2$-norm (due to the new inverse Stokes operator $\mathcal{A}_h^P$).
    \item[-] This means that in our discrete scheme \eqref{eq: weak lagrange fully discrete UNS} we have to consider the tangential $L^2$-norm for approximating the time derivative. Therefore, the ensuing results all follow (remember that the $\ah$-norm does include an $L^2$-norm) where we replace the bounds and error estimates for the $L^2$-norm, with the tangential $L^2$-norm, e.g. $\norm{\bfP\uh}_{L^2(\Gah)}\sim \norm{\bfPh\uh}_{L^2(\Gah)}$.
\end{itemize}
\end{remark}

\section{Numerical results}\label{sec: Numerical results}
In \Cref{Sec: set-up UNS}, we start with a general setup of our experiments, defining the parameters, and specifying the notation.
In \Cref{Sec: num Varying curvature surface UNS} we examine the choice of approximation of the F.E. space $\Lambda_h$, i.e. the choice of $k_{\lambda}$, and test our numerical results against the theoretical findings. Finally, in \Cref{Sec: num comparison UNS} we compare our scheme \eqref{eq: weak lagrange fully discrete UNS}, which utilizes an extra Lagrange multiplier (L.M.) to enforce the tangentiality condition, to an extension of the penalty formulation (P.M.) for the unsteady surface Navier-Stokes;
see \cite{reusken2024analysis} for the steady Stokes case, where also an appropriate penalty parameter is used.

\subsection{Setup}\label{Sec: set-up UNS}
The numerical results were implemented using the \emph{Firedrake} package \cite{FiredrakeUserManual}, where the linear systems were solved using GMRES with the help of LU preconditioners. Experiments are carried out on sequences of regular bisection mesh-refinements, to calculate the experimental orders of convergence.

Instead of using the approximation of the inertia term $\chtil(\bullet;\bullet,\bullet)$ introduced in \eqref{eq: skew-symmetrized ch real Ph}, and \Cref{remark: inertia term Ph} and used throughout the analysis, we use the following for our implementation
\begin{equation}
    c_h(\uh^{n-1};\uh^n,\vh) : = \int_{\Gah}((\uh^{n-1}\cdot \nbgcov)\uh^n)\vh,
\end{equation}
where we see no difference in convergence quality for all quantities. Using this approximation for the inertia term means that we eventually avoid any kind of approximation involving geometric quantities such as the Weingarten map, Gauss curvature, etc. Finally, we introduce the following notation:

\begin{equation*}
    \begin{aligned}
    \eu^{L^2(\ah)} &:=  \Big(\Delta t \sum_{n=1}^N \norm{\nbgcovh(\bfu^n - \uh^n)}_{L^2(\Gah)}^2\Big)^{1/2}, \quad 
        \ep^{L^2({L^2})} := \Big(\Delta t \sum_{n=1}^N \norm{p^n-\ph^n}_{L^2(\Gah)}^2\Big)^{1/2}, \\
        \eu^{L^{\infty}({L^2})} &:=  \max_{0 \leq n \leq N}\norm{\bfu^n-\uh^n}_{L^2(\Gah)}, \qquad\qquad\qquad \bfe_{\bfPg \bfu}^{L^{\infty}({L^2})} := \max_{0 \leq n \leq N}\norm{\bfPh(\bfu^n-\uh^n)}_{L^2(\Gah)},\\
        \bfe_{\bfng}^{L^{\infty}({L^2})} &:=  \max_{0 \leq n \leq N}\norm{\uh^n\cdot\nh}_{L^2(\Gah)}.
    \end{aligned}
\end{equation*}

\subsection{Example 1: Varying curvature surface}\label{Sec: num Varying curvature surface UNS}
In this example we consider the closed and compact surface $\Ga$ in \cite[Example 4.8]{DziukElliott_acta} which is described by the level set function 
\begin{equation*}
    \phi(x) = \frac{1}{4} x_1^2 + x_2^2 + \frac{4 x_3^2}{(1 + \frac{1}{2} \sin(\pi x_1))^2} - 1 \quad x \in \mathbb{R}^3.
\end{equation*}
It is constructed by mapping from a discretised unit sphere $\mathcal{S}$ with the help of the mapping
\begin{equation*}
    F(p) = (2p_0,\, p_1,\, \frac{1}{2}p_2(1 + \frac{1}{2}\sin(2\pi p_0))) \quad p \in \mathcal{S},
\end{equation*}
such that $\Ga = F(\mathcal{S})$. We describe an exact solution of the unsteady surface Navier-Stokes problems \eqref{eq: unstead NV Lagrange}, \eqref{eq: unstead NV} as
\begin{equation}
\begin{aligned}
    \bfu = \textbf{curl}_{\Ga} \psi, \ &\text{with} \ \psi =  (1-2t)\frac{1}{2\pi}\cos(2\pi x_1)\cos(2\pi x_2)\cos(2\pi x_3),\\
    &p = \sin(\pi x_1)\sin(2\pi x_2)\sin(2\pi x_3).
\end{aligned}
\end{equation}

 \begin{figure}[h]
\centering
  \subfloat[Velocity $t=0$]{\includegraphics[width=0.45\textwidth]{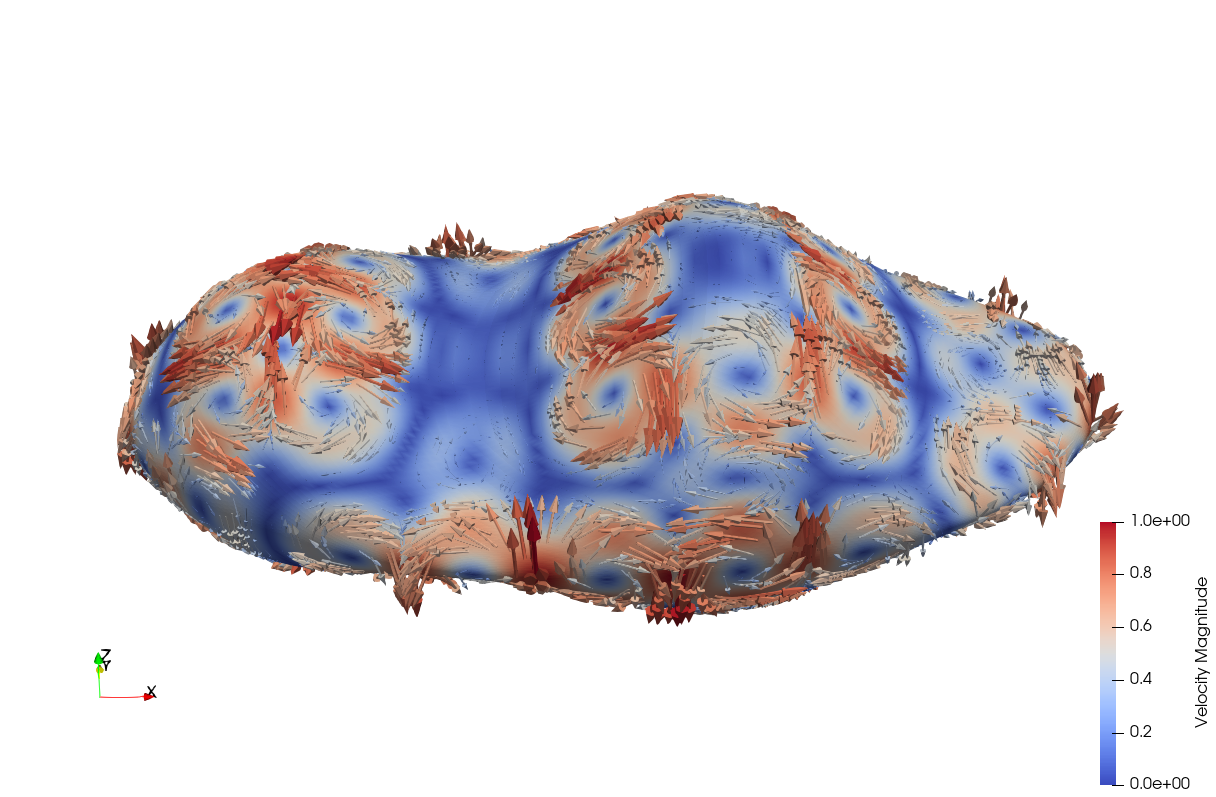}}
  \subfloat[Velocity $t=0.25$]{\includegraphics[width=0.45\textwidth]{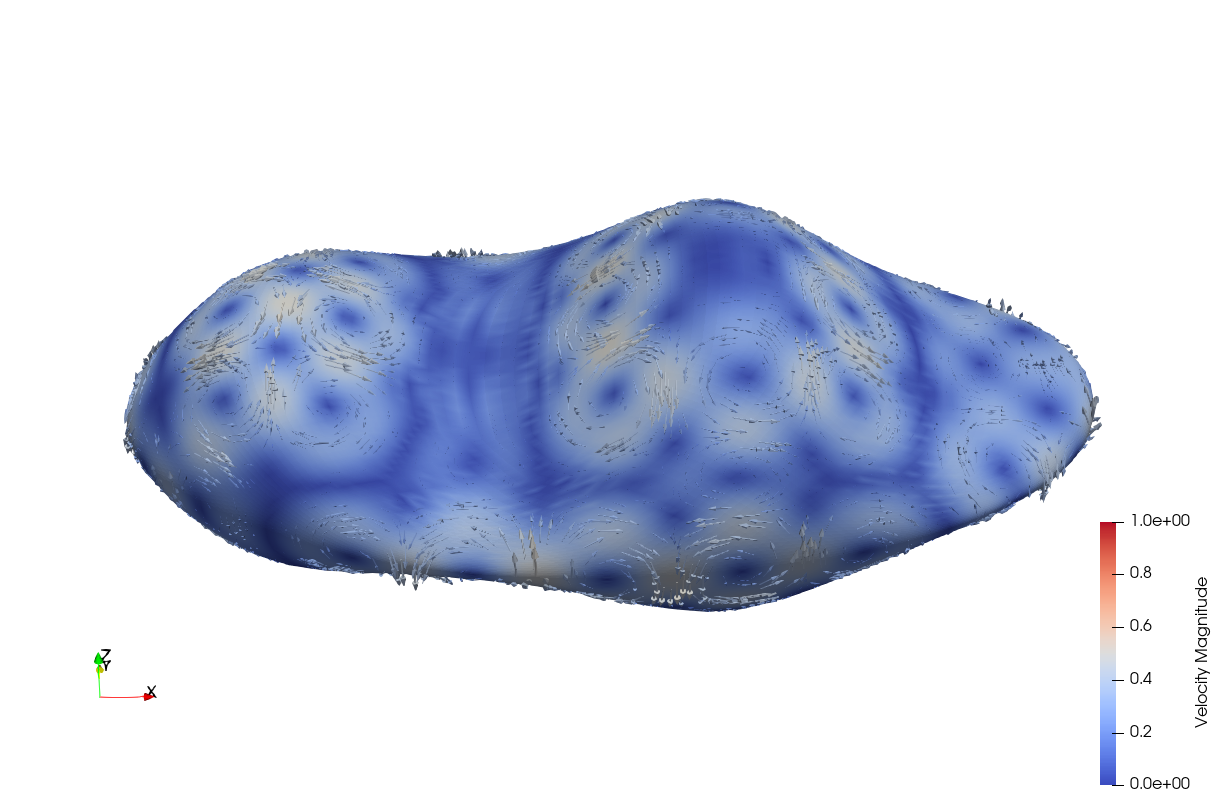}}
  \hfill
   \subfloat[Velocity $t=0.75$]{\includegraphics[width=0.45\textwidth]{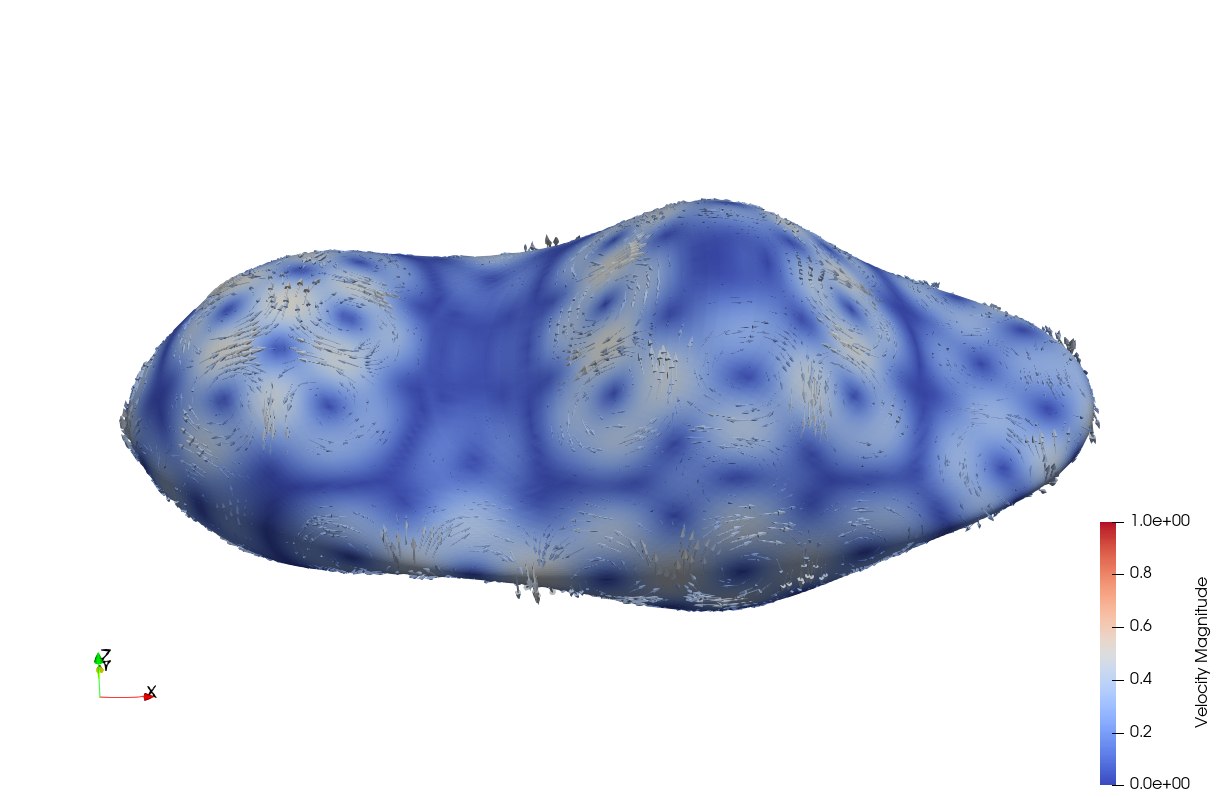}}
   \hfill
    \subfloat[Velocity $t=1$]{\includegraphics[width=0.45\textwidth]{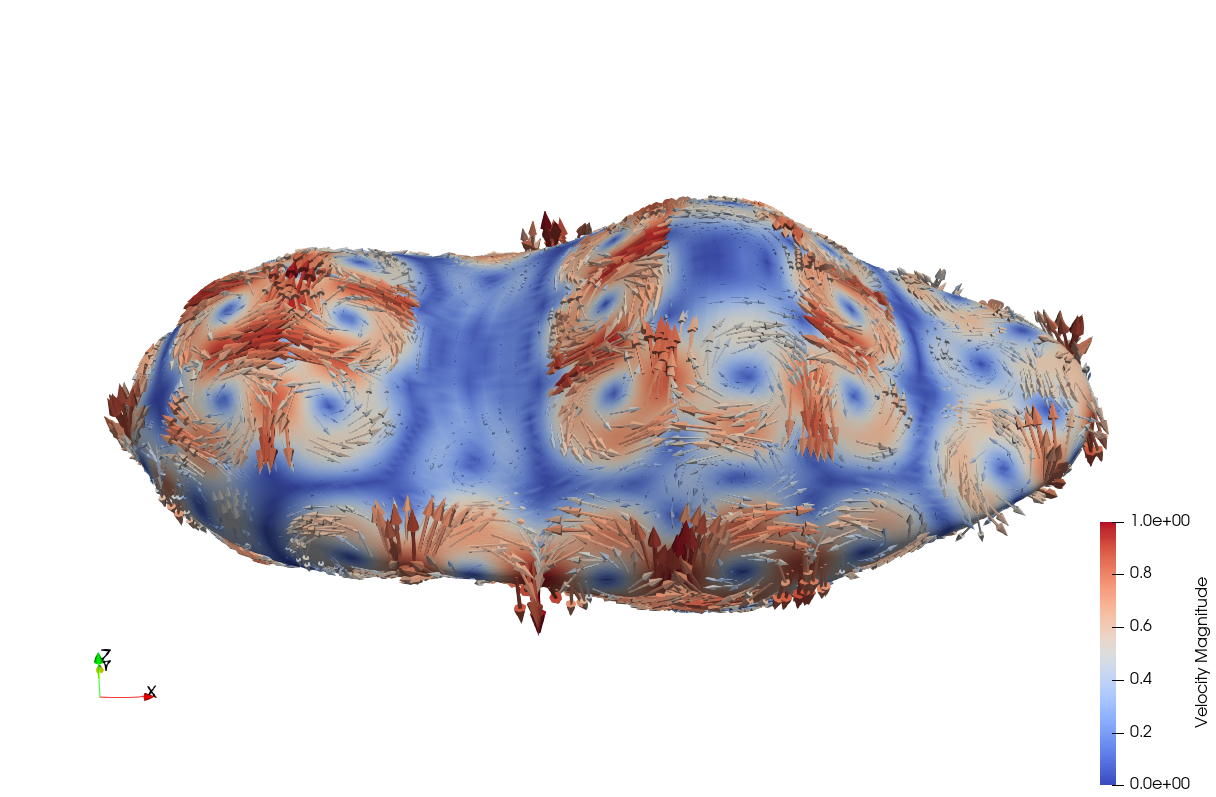}}
  \caption{Varying Curvature Surface | Velocity $\uh^n$ at different times for $k_g=3$, $k_u=2$, $k_{pr}=1$, $k_{\lambda}=1$, with mesh size $h=0.09$ and $\Delta t = 0.01$. The length and direction of the arrows depict the strength and orientation of the current.}
  \label{fig: Dziuk Velocity}
\end{figure}

From \cite{reusken2018stream} we know that the surface curl is tangential. So, our solution is tangential and divergence-free, while $p \in L^2_0(\Ga)$. The data $\bff$, $g$ , then, on the right-hand side and the Lagrange multiplier $\lambda$ can be calculated numerically with the help of the exact solution as interpolation of the smooth data. Obviously in our case $g=0$.

We compare the primitive solutions depending on the parameter choice $k_{\lambda}$. That is, the choice of the approximation space $\Lambda_h$ of the extra Lagrange multiplier $\lh$. We have our two standard \emph{cases}: When $k_\lambda = k_u=2$, and when $k_\lambda = k_u-1=1$. So, we use 
$\mathrm{\mathbf{P}}_{2}$ -- $\mathrm{P}_{1}$ -- $\mathrm{P}_{2}$ or $\mathrm{\mathbf{P}}_{2}$ -- $\mathrm{P}_{1}$ -- $\mathrm{P}_{1}$ \emph{Taylor-Hood} finite elements. Specifically, for the $\underline{k_\lambda = k_u-1}$ case, as seen in \cref{theorem: Velocity Error Estimates UNS,theorem: pressure estimate HR UNS} due to the limiting geometric convergence rate $\bigo(h^{k_g-1})$ we need to use \emph{super-parametric surface finite elements}, that is, consider higher geometric (surface) approximation compared to the degree of the velocity F.E. space $\bfV_h$, e.g. $k_g = k_u+1 =3$. On the other hand, this is not needed when $k_\lambda = k_u$.

We set the viscosity parameter $\mu =1/2$ and the density distribution $\rho=1$,  and start with an initial mesh size of $h_0=0.66$ and time-step $\Delta t_0 = 0.5$. Then, we consider a series of refinements such that the spatial refinement is halved while the temporal refinement is reduced by a factor of four, so we have $\Delta t \sim h^2$. That allows us to obtain optimal convergence; see \Cref{theorem: Velocity Error Estimates UNS,theorem: Pressures Error Estimate UNS,theorem: pressure estimate HR UNS}

\begin{figure}
    \centering
    \includegraphics[width = 0.85\textwidth]{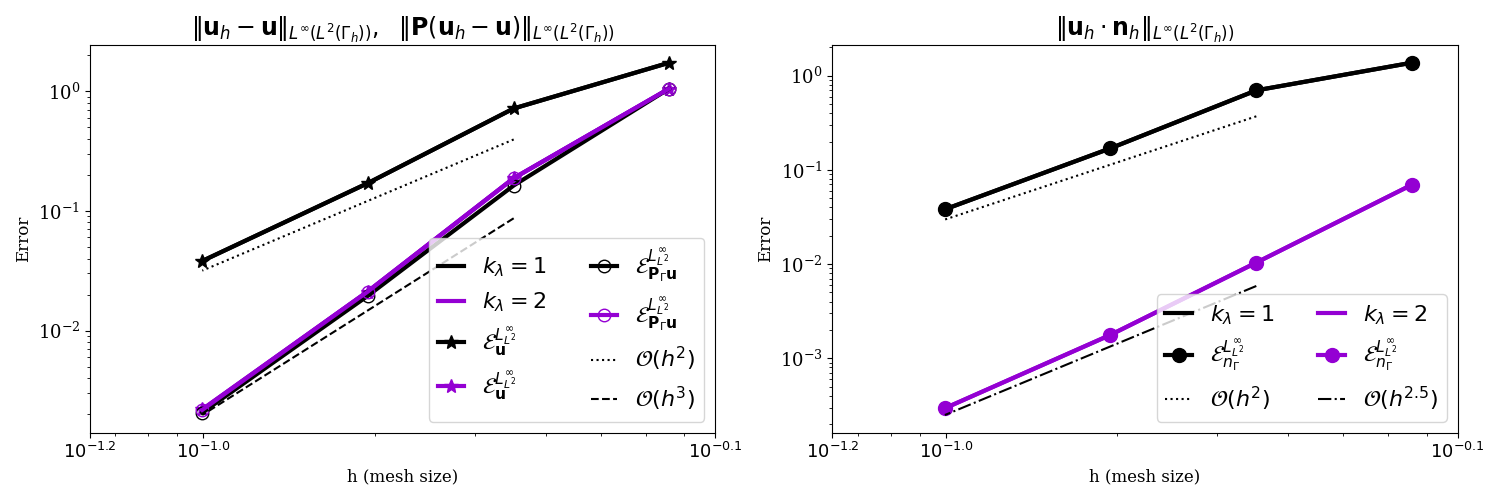}
    \caption{Varying Curvature Surface | Velocity $-$ Errors $\eu^{L^{\infty}({L^2})}$, $\bfe_{\bfPg \bfu}^{L^{\infty}({L^2})}$,  $\bfe_{\bfng}^{L^{\infty}({L^2})}$ | For different choice of $k_{\lambda}, \, k_g$| $(k_\lambda=1,\, k_g=3)$, $(k_\lambda=2,\, k_g=2)$.}
    \label{fig: DziukLagrange}
\end{figure}

The convergence results are presented in \cref{fig: DziukLagrange} and \cref{fig: DziukLagrange 2}. Regarding the covariant derivative error (energy norm) $\eu^{L^2(\ah)}$,  we observe, as expected from \Cref{theorem: Velocity Error Estimates UNS}, a second order $\bigo(h^2)$ convergence for both \emph{cases} (remember for $\underline{k_\lambda = k_u-1}$ we use super-parametric finite elements $k_g=3$). The same can be said about the $L^2_{L^2}$-norm pressure error $\ep^{L^2({L^2})}$, see \cref{theorem: Pressures Error Estimate UNS,theorem: pressure estimate HR UNS}

About the tangential $L^{\infty}_{L^2}$-norm we observe higher than expected $\bigo(h^3)$ convergence for both \emph{cases} (since $\Delta t \sim h^2$), which shows that the spatial error is dominant; see \cref{theorem: Velocity Error Estimates UNS}. This behavior, regarding the spatial error, aligns with previous findings on the optimal convergence of the tangential velocity in the Stokes case; see \cite[Section 6.4]{elliott2024sfem}. In fact, due to the tangential Ritz-Stokes bounds \eqref{eq: Error Bounds Ritz-Stokes L2 improved UNS} and \eqref{eq: Error Bounds Ritz-Stokes L2 UNS} we should expect this type of convergence, despite not deriving such estimate. Regarding the  normal approximation, determined by $\bfe_{\bfng}^{L^{\infty}({L^2})}$, we see that it exhibits the expected characteristics of $\bigo(h^2)$ order of convergence when $k_{\lambda} = k_u -1$, while when $k_{\lambda} = k_u$ we observe $\bigo(h^{2.5})$ convergence rate, which despite not deriving such estimate, exhibits similar behavior to the Stokes case; see \cite[Theorem 6.16]{elliott2024sfem} or the Ritz-Stokes estimate \eqref{eq: Error Bounds Ritz-Stokes L2 improved UNS}.

On the other hand, regarding the full $L^{\infty}_{L^2}$ velocity error $\eu^{L^{\infty}({L^2})}$, we observe that the convergence rate for the $k_\lambda=k_u$ \emph{case} ($\bigo(h^3)$) , is one order higher $(\bigo(h^2))$ compared to the $k_\lambda=k_u-1$ \emph{case}. This, again, aligns with previously proved estimates for the Stoke cases in \cite[Section 6.4]{elliott2024sfem}.


\begin{figure}
    \centering
    \includegraphics[width = 0.85\textwidth]{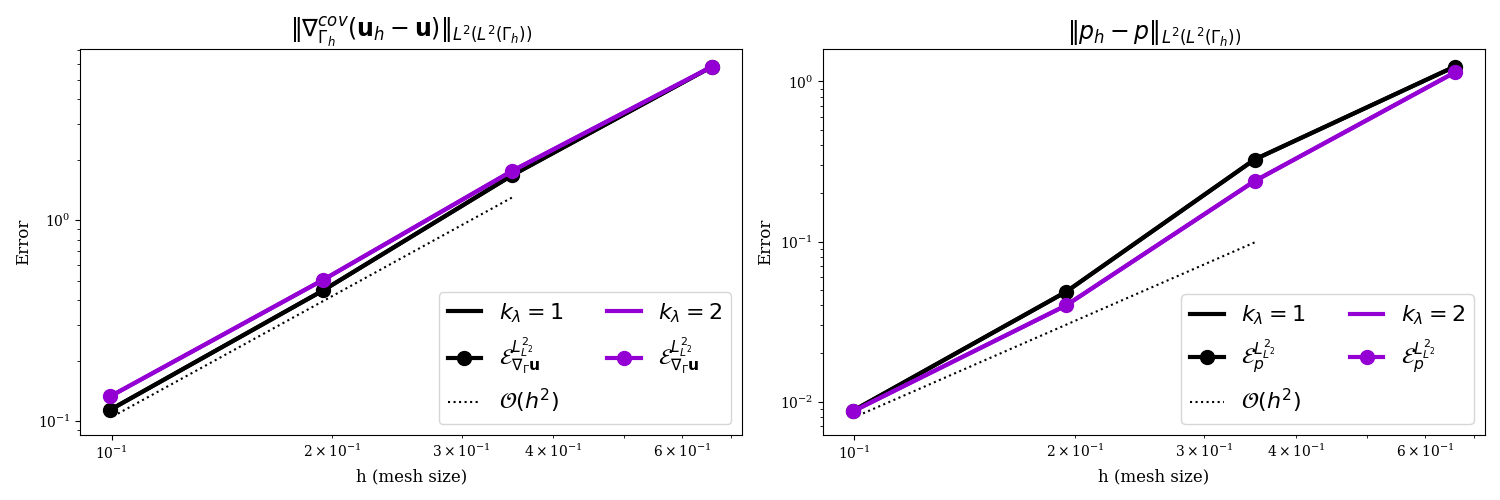}
    \caption{Varying Curvature Surface | Left: Velocity $-$ Error
 $\eu^{L^2(\ah)}$| Right: Pressure $-$ Error $\ep^{L^2({L^2})}$ | For different choice of $k_{\lambda}, \, k_g$| $(k_\lambda=1,\, k_g=3)$, $(k_\lambda=2,\, k_g=2)$.}
    \label{fig: DziukLagrange 2}
\end{figure}

\underline{To sum up the above results}: We see that the choice of $k_{\lambda}$ only plays an important role in the optimal convergence and control of the normal component of the velocity, since the tangent part of our solution convergences optimally regardless of the choice of $k_{\lambda}$ (depending, of course, on the geometric approximation $k_g$ as well). On the other hand, as noticed in \cite{fries2018higher} and in our experiments, setting $k_{\lambda}=k_u$ drastically \emph{worsens} the condition of the underlying system's equations as $h$ approaches zero. So, there is an interplay between choosing $\underline{k_{\lambda}=k_u}$ for optimal convergence (including the full  $L^{\infty}_{L^2}$ velocity) using \emph{iso-parametric surface finite elements} while worsening the system's condition and choosing $\underline{k_{\lambda}=k_u-1}$, where the condition of the system of equations is better, but one gets optimal convergence (only for tangent part of the solution) only when  using \emph{super-parametric surface finite elements}, i.e. higher approximation of the geometry.


\subsection{Example 2: Comparison} \label{Sec: num comparison UNS}
In this example we perform a simple comparison test for the Penalty method (\emph{P.M.}) and the Lagrange multiplier method (\emph{L.M.}).

We consider a  unit sphere $\mathcal{S}$ with the exact  smooth solutions of \eqref{eq: unstead NV Lagrange}, \eqref{eq: unstead NV} to be
\begin{equation}
    \bfu = J(t)\bfK_z, \quad p = (1+t)^3x_1x_2^2x_3,
\end{equation}
where $J(t) = (1 + x_3(2+0.5t)^3)$ and $\bfK_z$ the Killing vector field around the z-axis ($x_3$ in our case), with $\norm{\bfK_z} = 4$, thus it satisfies $E(\bfK_z)\equiv 0$, where $E$ the rate-of-strain tensor. The pressure also satisfies $p \in L^2_0(\Ga)$. The right-hand side $\bff, \,g$ and the Lagrange multiplier $\lambda$ can then be calculated with the help of the exact solution above. 

We consider a natural extension of the penalty formulation, as presented in \cite{reusken2024analysis} for the surface Stokes equations, in the case of the time-dependent surface Navier-Stokes equations; see also \cite{olshanskii2019penalty}. For this penalty formulation, we use $\mathrm{\mathbf{P}}_2-\mathrm{P}_1$ \emph{Taylor-Hood} iso-parametric surface finite elements with surface order approximation $k_g = 2$. In order to obtain optimal convergence we need to take care of the new penalty term introduced
\begin{equation*}
    \tau(\uh^n\cdot\nhtil,\vh\cdot\nhtil).
\end{equation*}
More specifically, we need to choose a higher-order approximation for the normal $\nhtil$, than the usual $\nh$, i.e. $\norm{\bfn-\nhtil}_{L^{\infty}(\Gah)}\leq ch^{k_g+1}$, and also choose the penalty term appropriately. For the construction of such normal $\nhtil$ see \cite[Remark 3.3]{hansbo2020analysis} for further details. In several literature \cite{hansbo2020analysis,olshanskii2019penalty,Olshanskii2018} it has also been shown that choosing $\tau = \alpha h^{-2}$ leads to optimal convergence. In our example, we observed that choosing $\tau = 2.5 h^{-2}$ produced optimal convergence.

\begin{figure}[h]
    \centering
    \includegraphics[width = 0.8\textwidth]{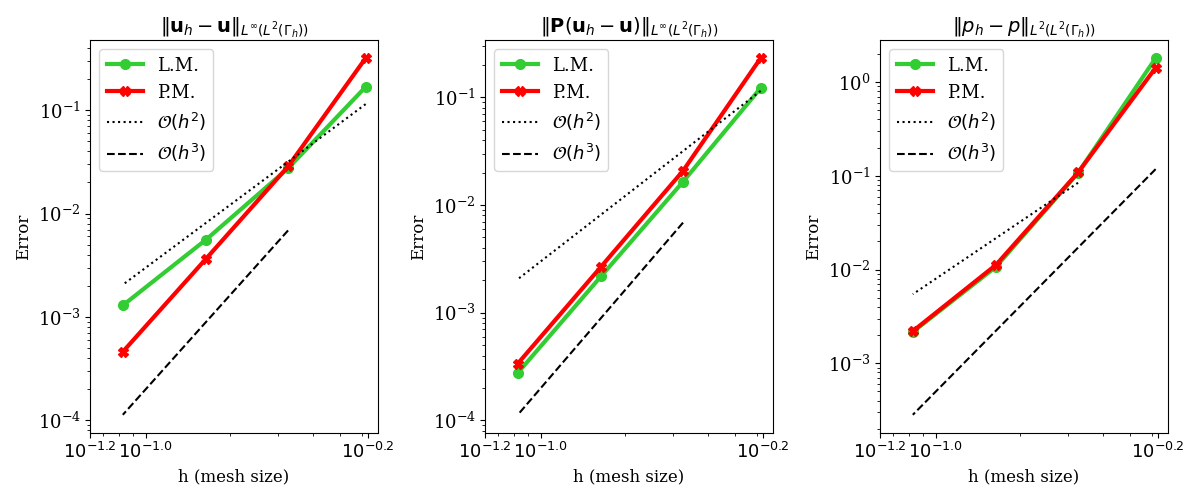}
    \caption{Sphere | Velocity-pressure $L^2_{L^2}$ norm $-$ Errors | \textcolor{BrickRed}{P.M.}: $\{k_g=2,k_u=2,k_{pr}=1,k_p=3\}$, \textcolor{OliveGreen}{L.M.}: $\{k_g=3,k_u=2,k_{pr}=1,k_{\lambda}=1\}$.}
    \label{fig: Sphere u-Pu-p}
\end{figure}

For the Lagrange multiplier method (L.M.), we prefer to use $\mathrm{\mathbf{P}}_2$ -- $\mathrm{P}_1$ -- $\mathrm{P}_1$ \emph{Taylor-Hood} \emph{super-parametric finite elements}, with surface order approximation $k_g=3$. See the above example in \Cref{Sec: num Varying curvature surface UNS} regarding the choice of approximation of the F.E. space $\Lambda_h$, i.e. the choice of $k_{\lambda}$.

In both cases, we set the viscosity parameter $\mu =1/2$ and the density distribution $\rho=1$. For both formulations of the problem, we also used a sufficiently accurate quadrature rule. We started with an initial mesh size of $h_0=0.62$ and time-step $\Delta t_0 = 0.5$. We then consider a series of refinements such that the spatial refinement is halved while the temporal refinement is reduced by a factor of four, such that $\Delta t \sim h^2$. That, then, allows us to obtain optimal convergence; see \Cref{theorem: Velocity Error Estimates UNS,theorem: pressure estimate HR UNS}

\begin{figure}[h]
    \centering
    \includegraphics[width = 0.8\textwidth]{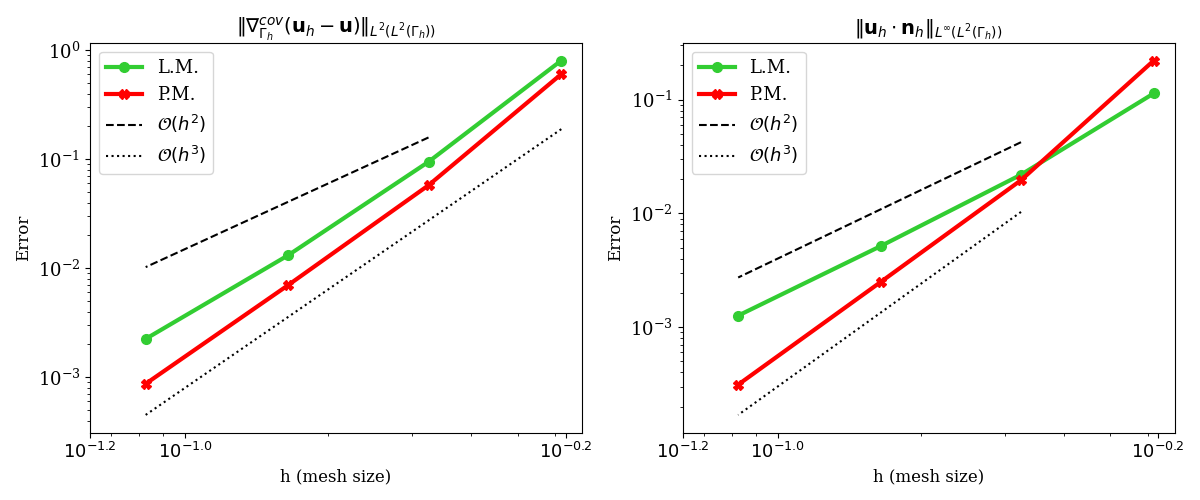}
    \caption{Sphere | Velocity $\mathcal{E}_{\nbgcov\bfu}^{L^{2}_{L^2}}-$Error (Left), $\mathcal{E}_{\bfng}^{L^{\infty}_{L^2}}-$Error (Right)\\
    | \textcolor{BrickRed}{P.M.}: $\{k_g=2,k_u=2,k_{pr}=1,k_p=3\}$,  \textcolor{OliveGreen}{L.M.}: $\{k_g=3,k_u=2,k_{pr}=1,k_{\lambda}=1\}$ .}
    \label{fig: Sphere Du-nu}
\end{figure}

In \cref{fig: Sphere u-Pu-p} we notice that the $L^{\infty}_{L^2}$ errors in the full velocity differ for the two formulations. Specifically, for the Lagrange multiplier method  $\eu^{L^{\infty}({L^2})}$ convergences with second order $\bigo(h^2)$, while for the penalty method we have optimal convergence of third order $\bigo(h^3)$. On the other hand, we see that both the $L^{\infty}_{L^2}$ errors for the tangential velocity $\bfe_{\bfPg \bfu}^{L^{\infty}({L^2})}$ agree. This follows \Cref{theorem: Velocity Error Estimates UNS} and the results that involve the optimal convergence of the tangent part of the velocity, as seen in the Stokes case \cite{elliott2024sfem}, which should also be valid in our context of time-dependent surface Navier-Stokes equations; see \cref{lemma: Error Bounds Ritz-Stokes std UNS}. Finally, we notice that both $L^2_{L^2}$ errors of the pressure $\ep^{L^2({L^2})}$ exhibits second order convergence.

In \cref{fig: Sphere Du-nu}, we also see that the $L^2_{L^2}$ error of the covariant derivative $\eu^{L^2(\ah)}$ for the two methods agree, again displaying optimal second-order convergence ($\bigo(h^2)$). Finally, we observe that the tangential condition is ``better'' enforced in the Penalty Method (P.M.), More specifically, we see that the normal velocity error $\bfe_{\bfng}^{L^{\infty}({L^2})}$ converges with order $\bigo(h^3)$ when we enforce the tangential condition via the penalty term, and $\bigo(h^2)$ when we enforce it via the Lagrange multiplier. Choosing instead \emph{iso-parametric T.-H. surface finite elements} with $k_\lambda=k_u$, would mitigate this discrepancy as seen in our previous example \cref{Sec: num Varying curvature surface UNS}.

\appendix
\section{Leray-Helmholtz decomposition}\label{appendix: leray UNS}
In this appendix we show the Leray-Helmholtz decomposition presented in Lemma \ref{lemma: Helmholtz-Leray decomposition UNS}.
\begin{lemma}[Helmholtz-Leray decomposition]\label{lemma: Helmholtz-Leray decomposition appendix}
For every $\bfu \in (H^1(\Ga))^3$ there exists unique $\phi \in H^2(\Ga)$, $\bfu_n \in (H^1(\Ga))^3$ and $\Pi^{div}(\bfu)\in \bfV^{div}$ such that 
\begin{equation*}
\bfu = \Pi^{div}(\bfu) + \bfu_n  + \nbg\phi.
\end{equation*}
\end{lemma}
\begin{proof}
For $\bfu \in \bfH^1(\Ga)$, since $\bfng \in C^3(\Ga)$ we readily see that $\bfu_n = (\bfu \cdot \bfng)\bfng = u_n \bfng \in \bfH^1(\Ga)$. Now we consider the following surface Laplace problem: Let $\phi \in H^1(\Ga) \cap L^2_0(\Ga)$ be solution to,
\begin{equation*}
    \int_{\Ga}\nbg \phi \cdot \nbg \xi =  \int_{\Ga} \xi \divg(\bfu - \bfu_n),
\end{equation*}
for every $\xi \in H^1(\Ga)$. Since $\bfu \in \bfH^1(\Ga)$, we then see that $\bfu - \bfu_n \in \bfH^1_T$ and thus $ \divg(\bfu - \bfu_n) \in L^2(\Ga)$. But, we may also notice that 
\begin{equation*}
    \int_{\Ga} \divg(\bfu - \bfu_n) = \int_{\Ga} (\bfu_T\cdot \bfng) \kappa =0,
\end{equation*}
hence we also have that  $ \divg(\bfu - \bfu_n) \in L^2_0(\Ga)$. Therefore, we see that the solution is unique and well defined. With the help of the inverse Laplacian operator $\Delta_{\Ga}^{-1} : H^{-1}(\Ga) \to H^1(\Ga)$ we also see that it satisfies $\Delta_{\Ga}^{-1} \divg(\bfu - \bfu_n) = \phi$. The elliptic regularity theorem \cite{DziukElliott_acta} shows that $\norm{\nbg(\Delta_{\Ga}^{-1} \divg(\bfu - \bfu_n))}_{H^1(\Ga)} \leq \norm{\phi}_{H^2(\Ga)} \leq c\norm{\divg(\bfu - \bfu_n)}_{L^2(\Ga)}$. Setting now 
$$\Pi^{div}(\bfu) =  \bfu - \bfu_n - \nbg \phi$$
we clearly see that $\divg(\Pi^{div}(\bfu)) =0$ and $\Pi^{div}(\bfu)\cdot\bfng =0$, and and therefore  $\Pi^{div}(\bfu) \in \bfV^{div}$. Finally, the uniqueness follows directly from the uniqueness of the solution of the surface Laplace equation.
\end{proof}

\section{Proof of \Cref{lemma: dual estimate UNS}}\label{appendix: proof of dual estimate UNS}
\begin{proof}[Proof of \Cref{lemma: dual estimate UNS}]
Lifting from $\Gah$ to $\Ga$, using the decomposition in \Cref{lemma: Helmholtz-Leray decomposition appendix} since $\vh \in \bfV_h$ and $\vhl \in H^1(\Ga)$, and going back to $\Gah$, 
we see that
    \begin{align}\label{eq: dual estimate inside 0 UNS}
        (\wh,\vh)_{L^2(\Gah)} = (\wh,\Pi^{div,-\ell}\vh)_{L^2(\Gah)} + (\wh\cdot\bfn,\vh\cdot\bfn)_{L^2(\Gah)}
         + (\wh,(\nbg\phi)^{-\ell})_{L^2(\Gah)}  
    \end{align}
where $\Pi^{div,-\ell}\vh = (\Pi^{div}\vhl)^{-\ell}$, $\bfn = \bfng^{-\ell}$, and $\phi\in H^2(\Ga)$ such that $\norm{\phi}_{H^2(\Ga)}\leq \norm{\divg\vhl}_{L^2(\Ga)}\leq c\norm{\divgh\vh}_{L^2(\Gah)}\leq c\norm{\vh}_{\ah}$, where the last inequality holds due to the perturbation bound \eqref{eq: Geometric perturbations a 1 UNS} and the fact that $\norm{\divgh\vh}_{L^2(\Gah)} \leq c\norm{\nbgcovh\vh}_{L^2(\Gah)}$.

We bound its term of \eqref{eq: dual estimate inside 0 UNS} appropriately. Let us start backwards. The third term is bounded by
\begin{equation}
    \begin{aligned}\label{eq: dual estimate inside 1 UNS}
        (\wh,(\nbg\phi)^{-\ell})_{L^2(\Gah)} &= (\wh,(\nbg\phi)^{-\ell} - \nbgh \phi^{-\ell})_{L^2(\Gah)} + (\wh, \nbgh \phi^{-\ell})_{L^2(\Gah)} \\
        & = (\wh,(\nbg\phi)^{-\ell} - \nbgh \phi^{-\ell})_{L^2(\Gah)} + (\wh, \nbgh (\phi^{-\ell} - \Ih(\phi^{-\ell})))_{L^2(\Gah)}\\
        & \leq ch\norm{\wh}_{L^2(\Gah)}\norm{\vh}_{\ah}\\
    \end{aligned}
\end{equation}
where we used the fact that $\wh \in \bfV_h^{div}$ and $\Ih(\phi^{-\ell}) \in S_{h,k_g}^{k_{pr}}$ the standard Lagrange interpolant (In fact $\Ih(\phi^{-\ell}) \notin L^2_0(\Gah)$ but see that we take gradient of this interpolant), and the geometric bound \eqref{eq: Geometric perturbations nbg UNS}.

For the second term, recall again that $\wh \in \bfV_h^{div}$ with $k_{\lambda} = k_u$, therefore we use \eqref{eq: divfree L2 inner normal kl=ku UNS} and geometric error \eqref{eq: geometric errors 2 UNS} to obtain
\begin{equation}
    \begin{aligned}\label{eq: dual estimate inside 2 UNS}
        (\wh\cdot\bfn,\vh\cdot\bfn)_{L^2(\Gah)} \leq c(h+h^{k_g})\norm{\wh}_{L^2(\Gah)}\norm{\vh}_{\ah}
    \end{aligned}
\end{equation}


 We are left with the first term. It follows by the discrete Leray-projection \eqref{eq: discrete Leray UNS} that 
\begin{equation*}
    \begin{aligned}
        (\wh,\Pi^{div,-\ell}\vh)_{L^2(\Gah)} = (\wh,\Pi_h^{div}(\Pi^{div,-\ell}\vh))_{L^2(\Gah)} + \underbrace{(\wh,\Pi^{div,-\ell}\vh - \Pi_h^{div}(\Pi^{div,-\ell}\vh))_{L^2(\Gah)}}_{:=0},
    \end{aligned}
\end{equation*}
and since $\Pi^{div,-\ell}\vh$ such that $\Pi^{div}\vhl \cdot \bfng =0$ and the discrete Leray projection $\Pi_h^{div}(\Pi^{div,-\ell}\vh) \in \bfV_h^{div}$ we are able to use the inverse Stoke operator \eqref{eq: Discrete inverse Stokes UNS} to show that
\begin{equation}
    \begin{aligned}\label{eq: dual estimate inside 4 UNS}
         (\wh,\Pi_h^{div}(\Pi^{div,-\ell}\vh))_{L^2(\Gah)} &= \ah(\mathcal{A}_h\wh,\Pi_h^{div}(\Pi^{div,-\ell}\vh))\\
         &\leq c_d\norm{\mathcal{A}_h\wh}_{\ah}\norm{\Pi_h^{div}(\Pi^{div,-\ell}\vh)}_{\ah}\\
         &\leq c_d\norm{\mathcal{A}_h\wh}_{\ah}\norm{\vh}_{\ah},
    \end{aligned}
\end{equation}
where we used the Leray projection stability bound \eqref{eq: Leray stability bound improved UNS} in the last inequality.
Therefore we see that
\begin{equation}
    \begin{aligned}\label{eq: dual estimate inside 3 UNS}
        (\wh,\Pi^{div,-\ell}\vh)_{L^2(\Gah)} \leq c_d\norm{\mathcal{A}_h\wh}_{\ah}\norm{\vh}_{\ah} + c_dh\norm{\wh}_{L^2(\Gah)}\norm{\vh}_{\ah}.
    \end{aligned}
\end{equation}
Combining \eqref{eq: dual estimate inside 1 UNS}-\eqref{eq: dual estimate inside 3 UNS} along with the fact that $\norm{\vh}_{H^1(\Gah)}\leq \norm{\vh}_{\ah}$ gives us our desired result.
\end{proof}

\bibliographystyle{siam}
\bibliography{bibliography}
\end{document}